\documentclass{article}
\usepackage{amsmath,amssymb,amsthm,graphicx,epsfig,float,url}
\usepackage[colorlinks=true,citecolor={Plum},linkcolor={Periwinkle}]{hyperref}
\usepackage{cleveref}[2012/02/15]

\usepackage{pdfsync}
\usepackage[usenames, dvipsnames]{xcolor}
\usepackage{tikz}
\usepackage{subfig}
\usepackage{bbm}
\usepackage{stmaryrd}
\usepackage{mathrsfs}  
\usepackage{caption}
\usepackage{float}
\usepackage{pgfplots}
\pgfplotsset{compat=newest}
\usepackage{dsfont}
\usepackage[makeroom]{cancel}
\usetikzlibrary{patterns}
\newcommand{\R}{\textnormal{I\kern-0.21emR}}
\newcommand{\N}{\textnormal{I\kern-0.21emN}}

\renewcommand{\geq}{\geqslant}
\renewcommand{\leq}{\leqslant}

\def\B{{\mathbb B}}
\def\e{{\varepsilon}}
\allowdisplaybreaks

\usepackage[dvipsnames]{xcolor}

\newtheorem*{theorem*}{Theorem}

\newtheorem{notation}{Notation}  

\newtheorem{theorem}{Theorem}  
\newtheorem{proposition}{Proposition}

\newtheorem{definition}{Definition}
\newtheorem{lemma}{Lemma}
\newtheorem{claim}{Claim}

\theoremstyle{definition}\newtheorem{remark}{Remark}

\def\O{{\Omega}}
\def\n{{\nabla}}
\def\p{{\varphi}}

\usepackage{xargs}
 \usepackage[colorinlistoftodos,textsize=small]{todonotes}
 \newcommandx{\domenec}[2][1=]{\todo[linecolor=purple,backgroundcolor=blue!20,bordercolor=purple,#1]{#2}}
 \newcommandx{\change}[2][1=]{\todo[linecolor=blue,backgroundcolor=blue!25,bordercolor=blue,#1]{#2}}
 \newcommandx{\info}[2][1=]{\todo[linecolor=green,backgroundcolor=green!25,bordercolor=green,#1]{#2}}
 \newcommandx{\improvement}[2][1=]{\todo[linecolor=yellow,backgroundcolor=yellow!25,bordercolor=yellow,#1]{#2}}
  \newcommandx{\biblio}[2][1=]{\todo[linecolor=blue,backgroundcolor=magenta!25,bordercolor=blue,#1]{#2}}

\usepackage{accents}
\newcommand{\dbtilde}[1]{\tilde{\raisebox{0pt}[0.85\height]{$\tilde{#1}$}}}

\begin{document}
\nocite{*}
\title{Constrained control of gene-flow models}


\author{Idriss Mazari\footnote{CEREMADE, UMR CNRS 7534, Universit\'e Paris-Dauphine, Universit\'e PSL, Place du Mar\'echal De Lattre De Tassigny, 75775 Paris cedex 16, France, (\texttt{mazari@ceremade.dauphine.fr})},  \quad Dom\`enec Ruiz-Balet$^\dag{} ^\S$ ,\quad Enrique Zuazua
\footnote{ Chair in Applied Analysis, Alexander von Humboldt-Professorship,
Department of Mathematics,
Friedrich-Alexander-Universit\"{a}t Erlangen-N\"{u}rnberg 91058 Erlangen, Germany} \footnote{Chair of Computational Mathematics, Fundaci\'on Deusto, Av. de las Universidades, 24,  48007 Bilbao, Basque Country, Spain} \footnote{ Departamento de Matem\'eticas, Universidad Aut\'onoma de Madrid, 28049 Madrid, Spain}}
\date{\today}

\maketitle
\begin{abstract}
In ecology and population dynamics, gene-flow refers to the transfer of a trait (e.g. genetic material)  from one population to another. This phenomenon is of great relevance in studying the spread of diseases or the evolution of social features, such as languages. From the mathematical point of view, gene-flow is modelled using bistable reaction-diffusion equations. The unknown is the proportion $p$ of the population that possesses a certain trait, within an overall population $N$. In such models, gene-flow is taken into account by assuming that the population density $N$ depends either on $p$ (if the trait corresponds to fitter individuals) or on the location $x$ (if some zones in the domain can carry more individuals). Recent applications stemming from mosquito-borne disease control problems or from the study of bilingualism have called for the investigation of the controllability properties of these models. At the mathematical level, this corresponds to boundary control problems and, since we are working with proportions, the control $u$ has to satisfy the constraints $0\leq u \leq 1$. In this article, we provide a thorough analysis of the influence of the gene-flow effect on boundary controllability properties. We prove that, when the population density $N$ only depends on the trait proportion $p$, the geometry of the domain is the only criterion that has to be considered. We then tackle the case of population densities $N$ varying in $x$. We first prove that, when $N$ varies slowly in $x$ and when the domain is narrow enough, controllability always holds. This result is proved using a robust domain perturbation method. We then consider the case of sharp fluctuations in $N$: we first give examples that prove that controllability may fail. Conversely, we give examples of heterogeneities $N$ such that controllability will always be guaranteed: in other words the controllability properties of the equation are very strongly influenced by the variations of $N$. All negative controllability results are proved by showing the existence of non-trivial stationary states, which act as barriers. The existence of such solutions and the methods of proof are of independent interest.  Our article is completed by several numerical experiments that confirm our analysis.
\end{abstract}


\medskip



\section{Introduction}
\paragraph{Motivations.}

In ecology and population dynamics, gene-flow refers to the transfer of a trait (e.g. genetic material)  from one population to another. This phenomenon strongly depends on the structure of the population density \cite{slatkin1987gene} as well as on the proportion of individuals that possess this trait inside the population. Typically, the populations involved in gene-flow phenomena are spatially separated, and the gene-flow effect occurs through spatial migrations. Given that it is a key factor in the evolution and differentiation of species, gene-flow has received a lot of attention from the biology community \cite{lenormand2002gene,gemmell2018genetic,bolnick2007natural,slarkin1985gene,bohonak1999dispersal}. One should also note that this effect has been observed as well in the evolution of plants \cite{mcdermott1993gene,levin1974gene}. This paper is devoted to the  controllability of biological systems involving this effect: is it possible, acting only on the boundary of the domain, to control the proportion of the trait within the population?

These boundary control problems arise naturally from population dynamics models and have several interpretations. For instance, one might consider the following situation: given a population of  mosquitoes, a proportion of which is carrying a disease, is it possible, acting only on the proportion of sick mosquitoes on the boundary, to drive this population to a state where only sane mosquitoes remain? From the application point of view, some mosquitoes that are immune to diseases such as malaria \cite{Raddi1128}, and it is therefore interesting to study the evolution of such proportions over space and time. Governments have recently tackled the issue of controlling diseases transmitted by mosquitoes by releasing genetically modified mosquitoes \cite{BBC}, and such questions have drawn the attention of the mathematical community in the past years \cite{PrivatVaucheletStrugarek}. Another example is that of linguistic dynamics: considering a population of individuals, a part of which is monolingual (speaking only the dominant language), the other part of which is bilingual (speaking the dominant and a minority language), is it possible, acting only on the proportion of bilingual speakers on the boundary of the domain, to drive the population to a state where there remains a non-zero proportion of bilingual speakers, thus ensuring the survival of the minority language? Such models are proposed, for instance, in \cite{UriarteIriberri}. In these works, as well as in a variety of other interpretations \cite{Barton,Bolnick,Gavrilet,Mirrahimi,StrugarekVauchelet,UriarteIriberri} of bistable models, the gene flow effect has been acknowledged as crucial in the underlying phenomenon: it states that, when we are interested in the proportion $p$ of a subgroup of a population density $N$, $N$ may depend on either $p$ (if for instance the subgroup is fitter) or on the space variable $x$ (if some zones in the domain are more favorable and can carry more individuals).  The goal of this article is to underline the complexity of the interaction between this gene-flow and the controllability properties of the system.

\paragraph{A motivating example: the spatially heterogeneous case.}
Let us give an example. We consider a population density $N=N(t,x)$, and we are interested in the dynamics of a proportion of the population, which will be denoted by $p=p(t,x)$. The classical gene-flow hierarchical system reads as follows: 

\begin{equation}\label{Eq:Motivational}
 \begin{cases}
  \frac{\partial N}{\partial t}-\Delta N=g(N,x)\text{ in }\R_+\times \O\\
    \frac{\partial p}{\partial t}-\Delta p-2\langle \n\ln N,\n p\rangle=f(p)\text{ in }\R_+\times \O\\
  \frac{\partial N}{\partial \nu}=0,\quad \frac{\partial p}{\partial \nu}=0\text{ on }\R_+\times \partial \O\\
    N(0,x)>0,\quad 0\leq p(0,x)\leq 1  \text{ in }\O.
 \end{cases}
\end{equation}
Typically, we can \color{black} choose a nonlinearity $g$ of monostable type\color{black}, such as $g(N,x)=N(\kappa(x)-N)$. In that case,  $\kappa(x)>0$ models the resources distribution available for the population  inside the domain. Monostable equations have been studied a lot since the seminal \cite{KOLMOGOROV37}. The nonlinearity $f$ is on the other hand assumed to be bistable; a typical example is $f(p)=p(p-\theta)(1-p)$ for some $\theta \in (0;1)$, see Figure \ref{Fi:Bi}.

Bistable reaction-diffusion equations are well-suited to describe the evolution of a subgroup of a population and are characterized by the so-called \emph{Allee effect}: there exists a threshold for the proportion of this subgroup such that, in the absence of spatial diffusion, above this threshold,  this subgroup will invade the whole domain (and drive the other subgroup to extinction) while, under this threshold, it will go extinct.

\begin{figure}[H]
\begin{center}
\begin{tikzpicture}[scale=2]
      \draw[scale=2,domain=-0.1:1.09,smooth,variable=\x,gray,dashed] plot ({\x},{2*\x*(1-\x)});
      \draw[scale=2,domain=0:1,smooth,variable=\x,blue, very thick] plot ({\x},{2*\x*(1-\x)});
 \draw[->,scale=2,gray] (-0.2,0) -- (1.3,0) ;
      \draw[->,scale=2,gray] (0,-0.2) -- (0,0.7) ;
      \node[scale=1] at (0,-0.1) {$0$};
      \node[scale=1] at (2,-0.1) {$\kappa(x)$};
      \node[scale=1] at (-0.25,1.25) {$g(\cdot,x)$};

    \end{tikzpicture}
\begin{tikzpicture}[scale=2]
      \draw[scale=2,domain=-0.1:1.0655,smooth,variable=\x,gray,dashed] plot ({\x},{4*\x*(\x-0.3)*(1-\x)});
      \draw[scale=2,domain=0:1,smooth,variable=\x,blue, very thick] plot ({\x},{4*\x*(\x-0.3)*(1-\x)});
      \draw[->,scale=2,gray] (-0.2,0) -- (1.3,0) ;
      \draw[->,scale=2,gray] (0,-0.2) -- (0,0.7) ;
      \node[scale=1] at (0,-0.1) {$0$};
      \node[scale=1] at (0.6,-0.1) {$\theta$};
      \node[scale=1] at (2,-0.1) {$1$};
      \node[scale=1] at (-0.25,1.25) {$f(\cdot)$};
    \end{tikzpicture}
\end{center}
\caption{Graph of a typical monostable nonlinearity $g(N,x)=N(\kappa(x)-N)$ (left) and typical bistable non-linearity $f(p)=p(p-\theta)(1-p)$ (right) .}\label{Fi:Bi}
\end{figure}
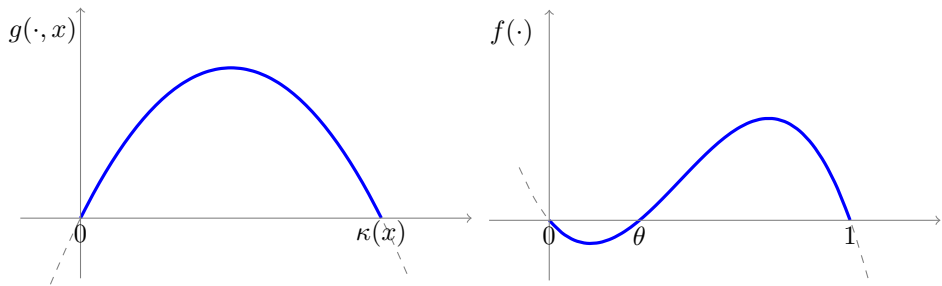

In that context, the goal is to drive the proportion $p$ to a spatially homogeneous equilibrium. Since it is often the case that we can not control the equation inside the domain (for instance when trying to control a mosquito population) we have to resort to boundary controls. As we will see, the shape of $N$ will have a drastic influence on such controllability properties.

In this article we assume that the first component has reached a steady state, and that the population distribution $N$ is stationary; in other words, we have $N=N(x)$. This reduces the gene-flow system to a scalar equation, in which we place a boundary control action.

\begin{equation*}
 \begin{cases}
  \frac{\partial p}{\partial t}-\Delta p-2\langle \n\ln N(x),\n p\rangle=f(p) \text{ in }\R_+\times \O\\
  p(t,\cdot)=u(t,\cdot)\text{ on }\R_+\times \partial \O \\
  0\leq p(0,x)\leq 1\text{ in }\O.
 \end{cases}
\end{equation*}
The state $p$ is a proportion, hence our boundary control action has to satisfy the bounds $$0\leq u(x,t)\leq 1.$$ Such bounds on the control are known to lead to  fundamental obstructions to the controllability as noticed in \cite{PoucholTrelatZuazua,BRZ}. 

Since our results fit in a growing body of literature, let us recall several results that hold in the absence of the gene-flow effects before stating our contributions.

\subsection{Known results regarding the constrained controllability of  bistable equations}

In \cite{PoucholTrelatZuazua,BRZ}, the controllability to $0$, $\theta$ or $1$ of the equation 
$$\frac{\partial p}{\partial t}-\Delta p=f(p)$$ with a constraint on the boundary control is  carried out using the staircase method \cite{PighinZuazua,CoronTrelat}. The emergence of nontrivial steady states (that is, steady states that are not identically equal to 0, 1 or $\theta$) is the main cause of lack of controllability for large domains, while controllability holds by constructing paths of steady states. 
In \cite{TrelatZhuZuazua}, the  equation
$$\frac{\partial p}{\partial t}-\Delta p=p(p-\theta(t))(1-p)$$ is considered, but this time, it is the Allee parameter $\theta=\theta(t)$ that is the control parameter and the target is a travelling wave solution. 
In \cite{PrivatVaucheletStrugarek},  an optimal control problem for the equation without diffusion
$$\frac{\partial p}{\partial t}=f(p)+u(t),$$ and with an interior control  $u$ (rather than a boundary one) is considered. We underline that, in their study, $u$ only depends on the time, and not on the space variable.

\subsection{Main contributions of the paper}

Our main contributions can be informally stated as follows:
\begin{itemize}
 \item \textbf{Slowly varying population density.} In that context, we assume that the population density $N$ varies very little from one point to the other. In other words, we assume that there exists a constant $N_0$, a function $n=n(x)$ and a small $\e>0$ such that $N=N(x)=N_0+\varepsilon n(x)$. If we think about the motivating example \eqref{Eq:Motivational}, this amounts to requiring that the resources distribution $\kappa$ is itself slowly varying $\kappa=\kappa_0+\e \eta(x)$ for some constant $\kappa_0$ and some function $\eta$. For this reason, we will also refer to this model as \emph{slowly varying spatial heterogeneity}. In that context, the heterogeneity does not qualitatively affect  the result from the homogeneous setting, see Theorem \ref{Th:SHSlow}. The proof is based on a very fine domain perturbation method.
 
 \item \textbf{Strongly varying population density}. In contrast to Theorem \ref{Th:SHSlow}, we consider the case when $N$ has rapid variations, which in turn may be interpreted as the effect of a \emph{strongly varying spatial heterogeneity}. In this case, the situation changes dramatically. Namely:
 \begin{itemize}
    \item Rapid variations of population inside the domain lead to lack of controllability due to the emergence of nontrivial steady-states which act as barriers. As an example of such a phenomenon, we study the case $\O=\mathbb{B}(0;R)$,  $N_\sigma(x)=e^{-\frac{||x||^2}\sigma}$, and we show that, whenever $\sigma>0$ is small enough, there exist non-trivial stationary solutions to the state equation on $p$, with boundary values either $0$ or $1$.  In Theorem \ref{NewTheo}, we give explicit assumptions on the drift to ensure the apparition of such non-trivial steady-states. This means, in terms of applications, that if there is piece of the domain with a high variation of the concentration of individuals, the control will fail.
           \item On the other side, if the population shows a rapid decay towards the interior, then, surprisingly, there exist\textcolor{black}{s} a critical threshold in $\sigma$ for which, independently of the size of the domain, there are no non-trivial solutions acting as barriers, so that controllability is achievable. In Theorem \ref{Th:Fin}, we show that this is  the case for $$N_\sigma(x)=e^{\frac{\|x\|^2}{\sigma}},$$ and prove this result for any $\sigma$ small enough. This is in sharp contrast with the homogeneous setting (homogeneous being understood in the sense that no drift is present) in which there was a critical size of the domain for which there was always one barrier \cite{PoucholTrelatZuazua,BRZ}. This result is proved using spectral analysis.
      
    \item In  Theorem \ref{Th:Rate}, we derive explicit decay rates on the spatial heterogeneity $N$ to ensure controllability, thus obtaining a result analogous to the results set  in the homogeneous setting \cite{PoucholTrelatZuazua,BRZ}.

 \end{itemize}
 \item\textbf{Infection-dependent limit. }So far, we have only mentioned the spatially heterogeneous case $N=N(x)$, but another context which is highly relevant for applications is that of infection-dependent models. This model, which also accounts for the gene-flow effect, is for instance obtained in \cite[Section 6]{NadinStrugarekVauchelet}, in which the authors reduce a system of 2x2 coupled reaction-diffusion equations to a scalar one and prove convergence, under some assumptions, to either spatially heterogeneous models or to infection-dependent models, which write \begin{equation*}
\frac{\partial p}{\partial t}-\mu \Delta p+2|\n p|^2\frac{h'(p)}{h(p)}=p(1-p)(p-\theta)
\end{equation*} for some $\theta\in (0;1)$ and some function $h$. This equation is most notably studied in \cite{NadinStrugarekVauchelet,StrugarekVauchelet}. This amounts to requiring that the population density $N$ depends on $p$: $N=N(p)$. In that case,  we show that the controllability results are exactly analogous to the ones obtained in the homogeneous setting, see Proposition \ref{Th:GF}.

\end{itemize}

\paragraph{Structure of the paper}

The structure of the paper is the following:

\begin{itemize}
 \item In Section \ref{MainRes}, we present the mathematical setting and the results 
 \item Sections \ref{Proof1},\ref{Proof2},\ref{Proof3},\ref{Proof4},\ref{Proof5} are devoted to the proofs of the theorems in order of presentation.
 \item Finally, in Section \ref{Conclusion}, we draw some concluding remarks and state some open problems.
\end{itemize}

\section{Setting and main results}\label{MainRes}

\paragraph{The equations and the control systems}
Let us first write down our systems. We say that $f:\R\to \R$ is \emph{bistable non-linearity} if: 
\begin{enumerate}
\item $f$ is $\mathscr C^\infty$ on $[0,1]$,
\item There exists $\theta\in (0;1)$ such that $0\,, \theta$ and 1 are the only three roots of $f$ in $[0,1]$. $\theta$ is called the Allee parameter.
\item $f'(0)\,, f'(1)<0$ and $f'(\theta)>0$,
\item Without loss of generality, we assume that $\int_0^1f>0$. In the typical example $f(p)=p(p-\theta)(1-p)$, this amounts to requiring that $\theta$ satisfies 
$\theta<\frac12$.

\end{enumerate}
We gave an example of such a bistable non-linearity in Figure \ref{Fi:Bi}.

We write our two models, the spatially heterogeneous and the infection-dependent one, in a synthetic way: we consider, in general, a population density of the form $N=N(x,p)$.  Infection-dependent models correspond to $N=N(p)$ and spatially heterogeneous models correspond to $N=N(x)$. With a bistable non-linearity $f$ and such a function $N$, in a domain $\O\subset \R^d$, the equation we consider writes, in its most general form
\begin{equation}\label{Eq:Main}\frac{\partial p}{\partial t}-\Delta p-2\langle\n_x\left(\ln(N(x,p))\right),\n p\rangle=f(p).\end{equation}

 Of particular relevance are the spatially homogeneous steady-states of this equation: $p\equiv 0$, $p\equiv \theta$ and $p\equiv 1$. \emph{Our objective in this article is to investigate whether or not it is possible to control any initial datum to these spatially homogeneous steady-states.}

 Let us formalize this control problem. Given an initial datum $p_0\in L^2(\O)$ such that 
$$0\leq p_0\leq 1$$ we consider the control system

\begin{equation}\label{Eq:AdvectionControl}
\left.\begin{cases}
\displaystyle\frac{\partial p}{\partial t}-\Delta p-2\langle \n \ln(N), \n p\rangle=f(p)&\text{ in }\R_+\times\O\,, 
\\p=u(t,x)&\text{ on }\R_+\times\partial \O ,
\\p(t=0,\cdot)= p_0,&
\end{cases}\right.
\end{equation}
where, for every $t\geq 0\,, x\in \partial \O$, \begin{equation}
\label{Eq:ConstraintControl}u(t,x)\in [0,1]\end{equation} is the control function. \eqref{Eq:ConstraintControl} is a natural constraint since we recall that $p$ stands for a proportion of the population. 
Our goal is to answer the following question:
\textit{Given any initial datum $0\leq p_0\leq 1$, is it possible to drive $p_0$ to $0$, $\theta$, or $1$ in (in)finite time with a control $u$ satisfying \eqref{Eq:ConstraintControl}?} In other words, can we drive any initial datum to one of the spatially homogeneous steady-states of the equation? If one thinks about infected mosquitoes, driving any initial population to 0 is relevant for controlling the disease while, if one thinks about mono or bilingual speakers, driving the initial datum to the intermediate steady-state $\theta$ ensures the survival of the minority language.

Let us denote  the steady-states as follows
$$\forall a\in \{0,\theta,1\}\,,z_a\equiv a.$$ By controllability, we mean the following: let $a\in \{0,\theta,1\}$, then
\begin{enumerate}
\item[$\bullet$] \underline{Controllability in finite time:} we say that \emph{$p_0$ is controllable to $a$  in finite time} if there exists a finite time $T<\infty$ such that there exists a control $u$ satisfying the constraints \eqref{Eq:ConstraintControl} and such that the solution $p=p(t,x)$ of \eqref{Eq:AdvectionControl} satisfies

$$p(T,\cdot)=z_a\text{ in }\O.$$
\item[$\bullet$] \underline{Controllability in infinite time:} we say that \emph{ $p_0$ is controllable to $z_a$ in infinite time}  if there exists a control $u$ satisfying the constraints \eqref{Eq:ConstraintControl} such that the solution $p=p(t,x)$ of \eqref{Eq:AdvectionControl} satisfies
$$p(t,\cdot)\underset{t\to \infty}{\overset{\mathscr C^0(\O)}\longrightarrow}z_a.$$
\end{enumerate}
\begin{remark}
Note that, in the definition of controllability in finite time, we do not ask that the controllability time be small; it cannot anyway be arbitrarily small because of the constraint $0\leq u\leq 1$ \cite{PighinZuazua}. Such constraints can indeed lead to  lack of controllability in fixed time horizon.  The question of the minimal controllability time for such bistable equations is, as far as the authors know, completely open at this point.\end{remark}
\begin{definition}\label{De:Contr}
{We say that \eqref{Eq:AdvectionControl} is controllable to $\color{black}a\color{black}$ in (in)finite time if it is controllable to $z_a$ in (in)finite time for any initial datum $0\leq p_0\leq 1.$}\end{definition}

We consider two cases for the flux $N=N(x,p)>0$ which have been discussed in \cite{Barton}:
\begin{enumerate}
\item[$\bullet$]\underline{The spatially heterogeneous model:} In this case, $N=N(x,p)$ is of the form
\begin{equation}\label{Eq:SH}\tag{$H_1$}
N=N(x){\color{black}>0}\color{black}.
\end{equation}

\item[$\bullet$]\underline{The infection-dependent model:} In this case, the function $N=N(x,p)$ assumes the form 

\begin{equation}\label{Eq:GF}\tag{$H_2$}
N(x,p)=N(p)\color{black}>0\color{black}.
\end{equation} 
\end{enumerate}

\subsection{Statement of the main  controllability results}

\subsubsection{A brief remark on the statement of the Theorems}
We are going to present controllability and non-controllability results for gene-flow models and spatially heterogeneous ones. Regarding obstructions to controllability, the main obstacles are the existence of non-trivial steady-states, namely solutions to 
$$-\Delta p-2\left\langle \frac{\n N}{N},\n p\right\rangle=f(p)\text{ in }\O$$ associated with the boundary conditions $p=0$ or $p=1$ on $\partial \O$. However, given that the existence of non-trivial solutions for the Dirichlet boundary conditions $p=0$ is obtained through a sub and super solution methods, the natural quantity appearing is the inradius of the domain.
\begin{definition}[Inradius of a domain]
Let $\O\subset \mathbb{R}^N$ be bounded.

 $$\rho_\O=\sup\{r>0\,, \exists x\in \O\,, \mathbb B(x,r)\subset \O\},$$
\end{definition}
The non-existence of non-trivial solutions is usually done through the study of the first Laplace-Dirichlet eigenvalue 
$$\lambda_1^D(\O):=\inf_{p\in W^{1,2}_0(\O)\,, p\neq 0}\frac{\displaystyle\int_\O |\n p|^2}{\displaystyle\int_\O p^2},$$ which explains why both quantities $\rho_\O$ and $\lambda_1^D(\O)$ appear in our first results. Using Hayman-type inequalities, see \cite{BrascoDePhilippis}, we could rewrite $\lambda_1^D(\O)$ in terms of the inradius when the set $\O$ is convex. Indeed, it is proved in \cite[Proposition 7.75]{BrascoDePhilippis} that, when $\O$ is a convex set with $\rho_\O<\infty$ then 
$$\frac1{c\rho_\O^2}\leq \lambda_1^D(\O)\leq \frac{C}{\rho_\O^2},$$ so that the theorems can be recast in terms of inradius only in the case of convex domains.

\subsubsection{Spatially heterogeneous models} In this case, we work under assumption (\ref{Eq:SH}), i.e with $N=N(x)>0\text{ in }\O$. As explained in the introduction, we will treat two cases: slowly varying heterogeneities and rapidly varying ones. 
\paragraph{Slowly varying heterogeneity}
Let us first study the case of a  {slowly varying total population size}: we consider an environment with small spatial changes in the total population size; this amounts to requiring that 
$$\left|\frac{\n N}{N}\right|\ll1,$$
where $N$ satisfies \eqref{Eq:SH}. We can then formally write
$$N\approx N_0+\frac\e2 n(x),$$
where $N_0$ is a constant\footnote{We can assume, without loss of generality, that $N_0=1$. Indeed, the equation \eqref{Eq:AdvectionControl} is invariant under the scaling $N\mapsto \lambda N$ where $\lambda\in \R_+^*$.}.
 We consider, for a   function $n\in \mathscr C^1(\R^d;\R)$ and a parameter $\e>0$, the control system
\begin{equation}\label{Eq:AdvectionControlSlow}
\left\{\begin{array}{ll}
\frac{\partial p}{\partial t}-\Delta p-\e \langle \n n, \n p\rangle=f(p)&\text{ in }\R_+\times \O\,,
\\p=u(t,x)&\text{ on }\R_+\times\partial \O,
\\0\leq u\leq 1,&
\\p(t=0,\cdot)=p_0\,, 0\leq p_0\leq 1\,, &\end{array}\right.
\end{equation}
For simplicity, we assume that $n$ is defined on $\R^d$ rather than on $\O$. Since we already assumed that $N$ was $\mathscr C^1$, this amounts to requiring that $n$ can be extended in a $\mathscr C^1$ function outside of $\O$, which once again would follow from regularity assumptions on $\O$.

\begin{theorem}\label{Th:SHSlow}
Let, $\O\subset \R^d$ be a $\mathscr C^2$ domain. Let $n \in \mathscr C^1(\R^d)$.  Then:
\begin{enumerate}
\item \underline{Lack of controllability for large inradii:} There exists $\rho^*=\rho^*(n,f)>0$ such that if $\rho_\O>\rho^*$, then \eqref{Eq:AdvectionControlSlow} is not controllable to 0 in (in)finite time in the sense of Definition \ref{De:Contr}: there exist initial data $p_0$ such that, for any control $u$ satisfying the constraints \eqref{Eq:ConstraintControl}, the solution $p$ of \eqref{Eq:AdvectionControlSlow} does not converge to 0 as $t\to \infty$.
\item\underline{Controllability for large Dirichlet eigenvalue and small spatial variations:} If $\lambda_1^D(\O)>||f'||_{L^\infty}$, there exists $\e_*=\e_*(n,f,\O)$ such that, when $\e\leq \e_*$, the Equation \eqref{Eq:AdvectionControlSlow} is controllable to 0 and 1 in infinite time f  and to $\theta$ in finite time in the sense of Definition \ref{De:Contr}.
\end{enumerate}
\end{theorem}
We note that controllability to 0 or 1 can not hold in finite time, as it would violate the comparison principle, see \cite{BRZ}.
To prove this theorem, we have to very finely adapt, using perturbative arguments, the staircase method of \cite{CoronTrelat}.
\paragraph{(Lack of) controllability for rapidly varying total population size: (un)blocking phenomenons}
The previous result, however general, is proved using a very implicit method that does not enable us to give explicit bounds on the perturbation size $\e$.

As mentioned, the lack of controllability occurs when barriers appear. For instance, if a non-trivial solution to 
\begin{equation*}
\left\{\begin{array}{ll}
-\Delta \p_0-2\langle\frac{ \n N}N, \n \p_0\rangle=f(\p_0)&\text{ in }\O\,,
\\\p_0=0&\text{ on }\partial \O,
\end{array}\right.
\end{equation*}
exists, then it must reach its maximum above $\theta$ (this follows from the optimality conditions for maximisers of the function) and thus, from the maximum principle, it is not possible to drive an initial datum $p_0\geq \p_0$ to 0 with constrained controls. This kind of counter-examples appear when the drift is absent, see \cite{BRZ,PoucholTrelatZuazua}.  They are usually constructed by means of sub and super solutions of the equation. What is more surprising however is that adding a drift may lead to the existence of non-trivial solutions to 
\begin{equation*}
\left\{\begin{array}{ll}
-\Delta \p_1-2\langle\frac{ \n N}N, \n \p_1\rangle=f(\p_1)&\text{ in }\O\,,
\\\p_1=1&\text{ on }\partial \O,
\end{array}\right.
\end{equation*}
which never happens when no drift is present. In that case, driving the population from an initial datum $p_0\leq \p_1$ to $z_1$ is impossible. 

In this paragraph we give \color{black}examples of assumptions on drifts $N$ such that \color{black} the equation is not controllable to either 0, $\theta$ or $1$ in a fixed ball $\mathbb B(0;R)$, whenever the drift's intensity is too large or, conversely, \color{black} such \color{black} that the equation is always controllable regardless of $R$ when the drift's intensity is large enough. \color{black} Of course, the assumptions we make are only sufficient to ensure (non-)controllability, and not necessary. \color{black}

To formalise what we mean by "intensity of the drift", let us then consider, for a fixed radius $R>0$ and a fixed real parameter $\sigma>0$, the equation 
\begin{equation}
\label{EqSigma}\begin{cases}\partial_tp-\Delta p-\frac2\sigma \left\langle\frac{\nabla N}N,\nabla p\right\rangle=f(p) \text{ in }\R_+\times\mathbb B(0;R)\,,\\ p(t,\cdot)=u(t,\cdot)\in \R_+\times \color{black}\partial \B(0;R)\,, \\0\leq u\leq 1\text{ in }\R_+\times\color{black} \partial \B(0;R),\\ p(0,\cdot)=p_0\,, 0\leq p_0\leq 1\text{ in }\color{black}\B(0;R).\end{cases}.\end{equation}
The parameter $\sigma$ quantifies the drift's intensity.
\paragraph{Blocking phenomenons for certain classes of drifts}
We introduce the following assumptions on the drift $N$: the first one reads
\begin{equation}\tag{$\bold{T1}$}\label{AssumptionDrift}\exists C>0\,, \frac{\partial_r N}{N}\leq -Cr\end{equation}
while the second writes
\begin{multline}\tag{$\bold{T2}$}\label{AssumptionDrift2}  N\in \mathscr C^1(\O)\,, \exists c_0,c_1>0\,, \forall r \in [0;R]\,,\\ \forall \theta_1,\dots,\theta_{d-1}\in [0;2\pi]\,,e^{-c_0\frac{r^2}2}\leq N(r,\theta_1,\dots,\theta_{d-1})\leq e^{-c_1\frac{r^2}2} .\end{multline}

Our main result is the following Theorem:
\begin{theorem}\label{NewTheo}
Let $R>0$ and $\O:=\mathbb B(0;R)$.\begin{enumerate}\item \underline{Lack of controllability to 1:} Assume that $N$ is a $\mathscr C^1$ function satisfying \eqref{AssumptionDrift}. There exists $\sigma_N>0$ such that, for any $\sigma\in(0;\sigma_N)$, equation \eqref{EqSigma} is not controllable to 1 in $\Omega$.
\item\underline{Lack of controllability to 0:} Assume that $N$ is a $\mathscr C^1$ function satisfying \eqref{AssumptionDrift2}. There exists $\sigma_N'>0$ such that, for any $\sigma\in(0;\sigma_N')$, equation \eqref{EqSigma} is not controllable to 0 in $\Omega$.
\item[]\underline{Corollary}: If $N$ is $\mathscr C^1$ and satisfies \eqref{AssumptionDrift}-\eqref{AssumptionDrift2}, there exists $\underline \sigma>0$ such that, for any $\sigma \in (0;\underline \sigma)$, equation \eqref{EqSigma} is not controllable to either 0 or 1 in $\Omega$.
\end{enumerate}

\end{theorem}
\begin{remark}

\begin{enumerate}

\item It should be noted that, since the lack of controllability is proved using the existence of non-trivial solutions of the equation with homogeneous boundary values 0 or 1, the controllability to $\theta$ can not hold for arbitrary initial conditions.
\item

Sharp changes  in the total population size have been known, since \cite{NadinStrugarekVauchelet}, to provoke blocking phenomenons for the traveling-waves solutions of the bistable equation, and our results seems to lead to the same kind of interpretation: when a sudden change occurs in $N$, it is hopeless for a population coming from the boundary to settle everywhere in the domain.
\item We will, for the sake of readability, split the proof of Theorem \ref{NewTheo} in two parts, one devoted to the existence of non-trivial steady-states with boundary value 1, one devoted to the existence of such non-trivial solutions with boundary value 0. 
\item The methods used to establish the existence of non-trivial solutions with boundary values 0 or 1 will however be very different: while the existence of a steady-state associated with homogeneous boundary value 0 relies on variational arguments, the existence of non-trivial steady-states associated with the homogeneous boundary value 1 will rest upon a very fine analysis of the phase plane portrait and will take up most of the sections of the proof. Such complexity in the proof is required by the fact that the steady state $z_1\equiv 1$ is a global minimiser of the energy functional on the space with homogeneous boundary conditions equal to 1.
\item This result seems to  indicate that the angular component of the drift has very little, if  any, influence on controllability.
 \end{enumerate}

\end{remark}
\color{black}

We illustrate the existence of non-trivial solutions in Figures \ref{Fi:4}, \ref{Fi:5}, \ref{Fi:6} and \ref{Fi:7} for the one-dimensional case, $\O=(-L;L)$. In this case, the drift under consideration is $N(r):=e^{-r^2},$ and the equations to which non-trivial solutions must be found are 
\begin{equation}\label{Eq:Contre0}
\left\{\begin{array}{ll}
-\frac{\partial^2p}{\partial x^2}+\frac{2x}\sigma \frac{\partial p}{\partial x}=f(p)\text{ in }[-L,L]\,,& 
\\p(\pm L)=0,&
\\0\leq p\leq 1,&
\end{array}\right.
\end{equation}
and \begin{equation}\label{Eq:Contre1}
\left\{\begin{array}{ll}
-\frac{\partial^2p}{\partial x^2}+\frac{2x}\sigma \frac{\partial p}{\partial x}=f(p)\text{ in }[-L,L]\,,& 
\\p(\pm L)=1,&
\\0\leq p\leq 1,&
\end{array}\right.
\end{equation}

We will study the energy $$\mathscr E:(p,v)\mapsto \frac12v^2+F(p),$$
{where $F(p)=\int_0^pf(s)ds$.}

\begin{figure}[H]\begin{center}
 \includegraphics[scale=0.3]{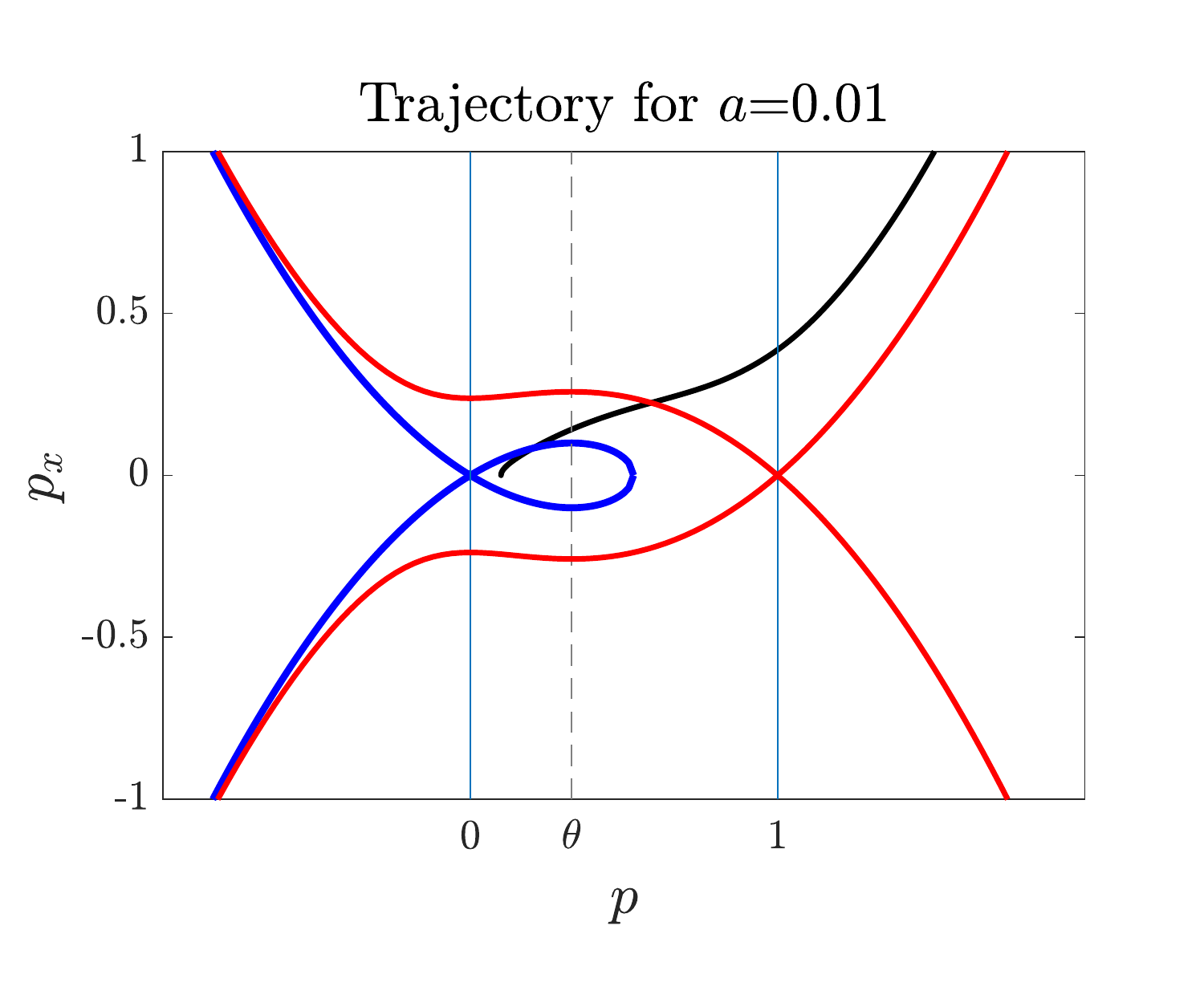}
  \includegraphics[scale=0.3]{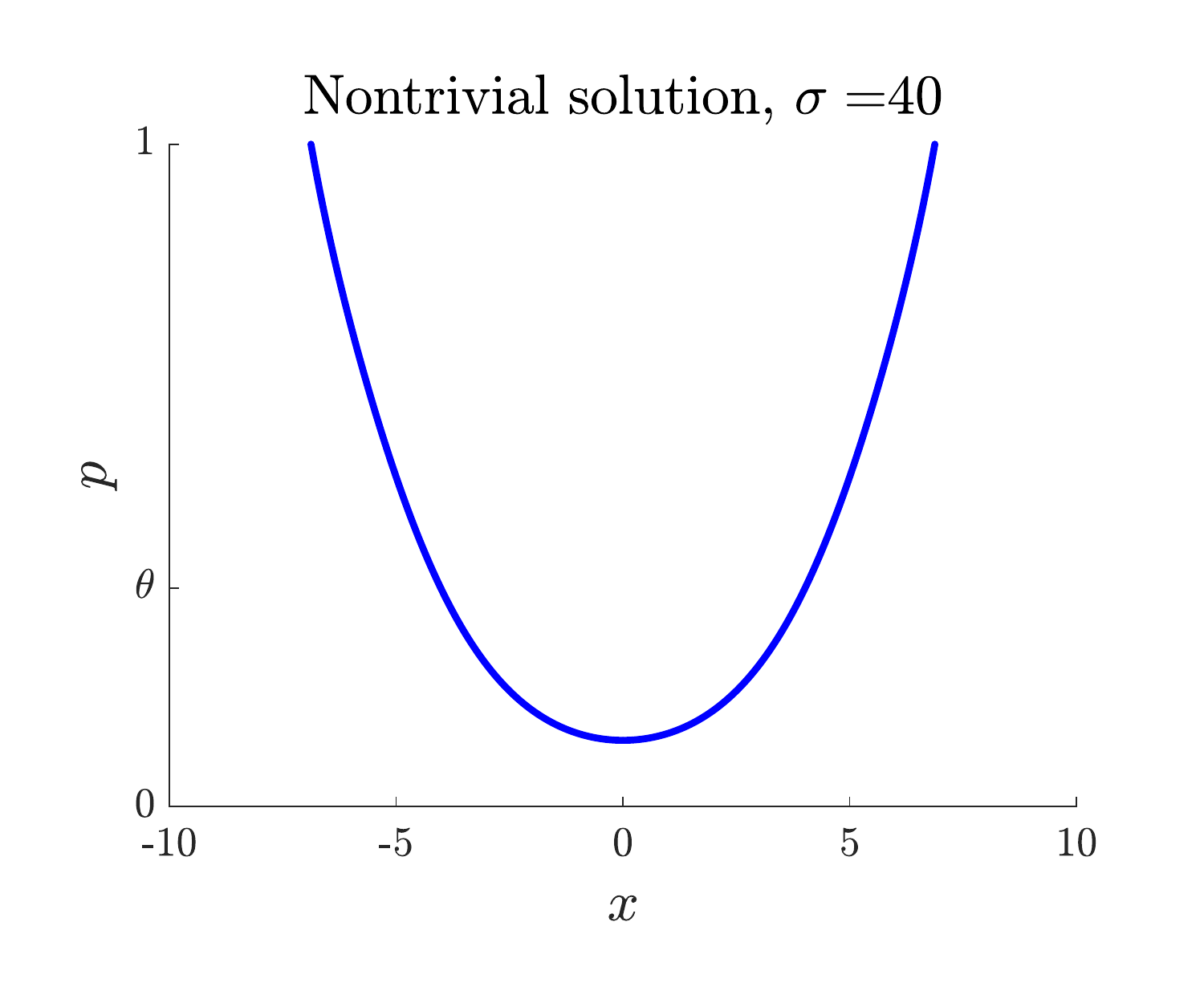}\end{center}
  \caption{$\sigma=40$ and $f(s)=s(1-s)(s-\theta)$, $\theta=0.33$.
  Phase portrait (Left):  the trajectory corresponding to the nontrivial solution is in black, the energy set $\{\mathscr E=F(1)\}$ in red, the energy set $\{\mathscr E=F(0)\}$ in blue. 
 Nontrivial solution of \eqref{Eq:Contre1} (Right).   }\label{Fi:4}
\end{figure}

\begin{figure}[H]\begin{center}
 \includegraphics[scale=0.3]{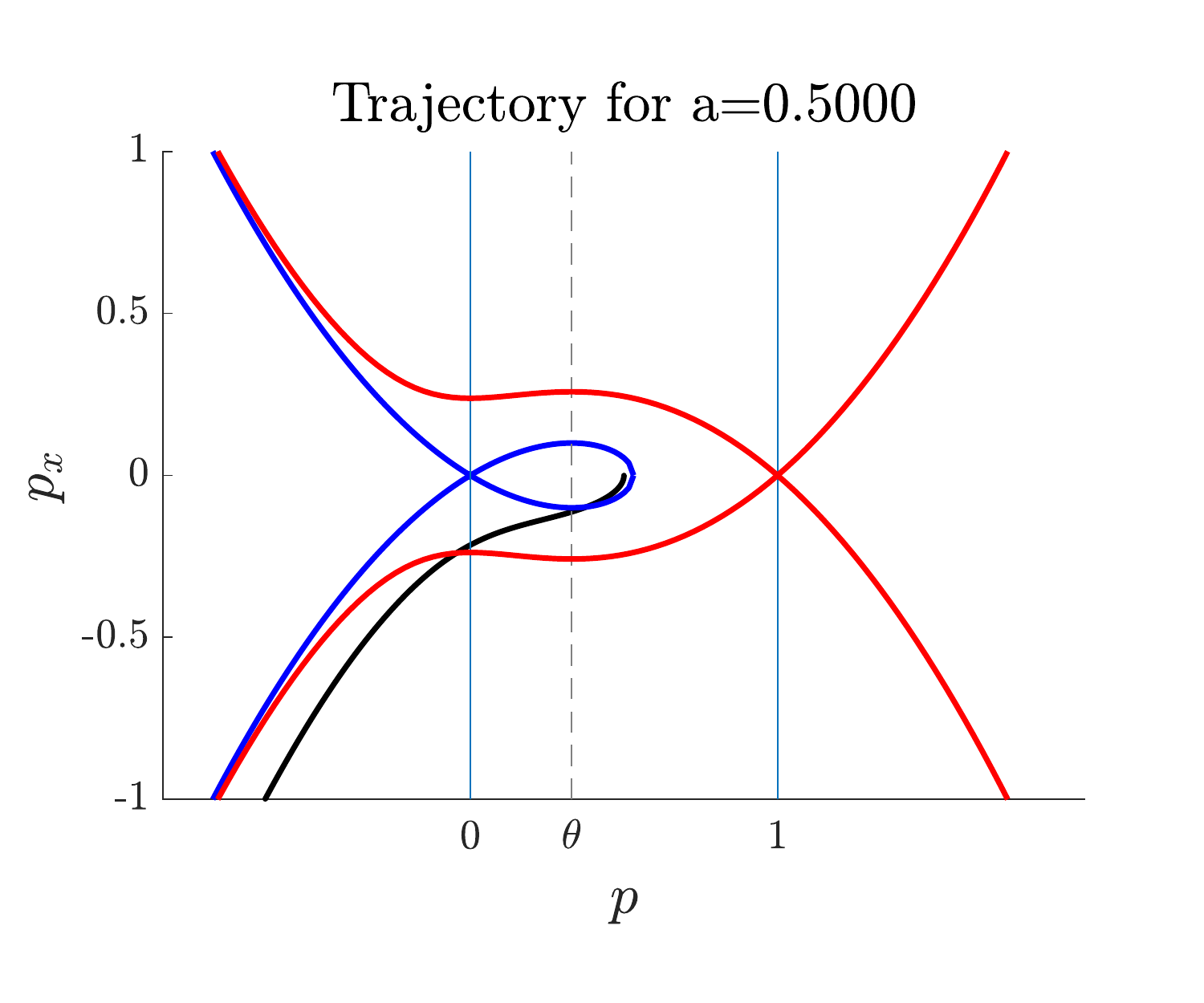}
  \includegraphics[scale=0.3]{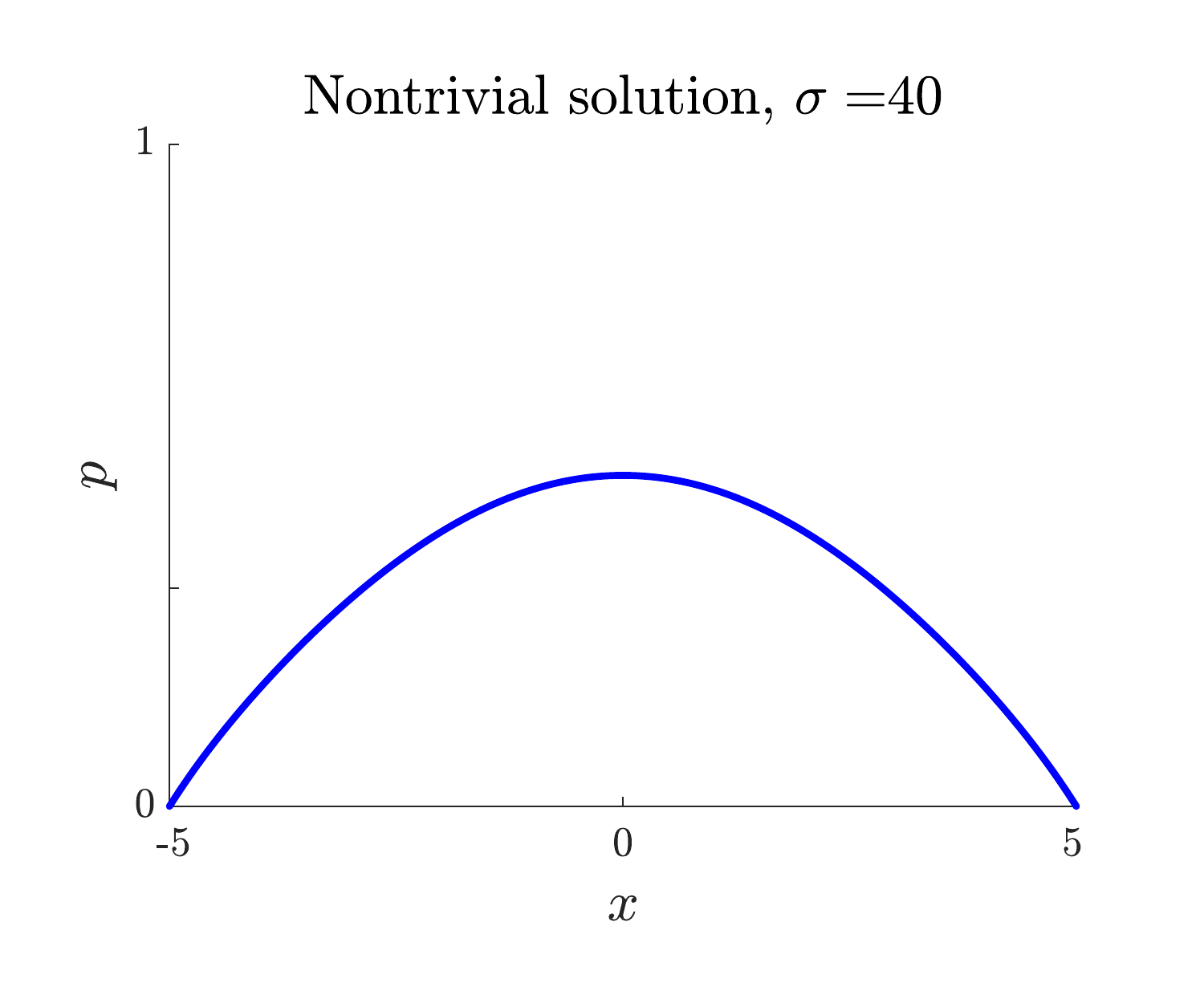}\end{center}
  \caption{Same class of parameters $\sigma,\theta,f$. Phase portrait (Left): the trajectory corresponding to the nontrivial solution is in black, the energy set $\{\mathscr E=F(1)\}$ in red, the energy set $\{\mathscr E=F(0)\}$ in blue.  Nontrivial solution  of \eqref{Eq:Contre0} (Right).
}\label{Fi:5}
  \end{figure}
  We also observe this "double-blocking" phenomenon (i.e the existence of non-trivial solutions to \eqref{Eq:Contre1} and \eqref{Eq:Contre0} in the same interval) numerically, when trying to control an initial datum to $\theta$:
  
\begin{figure}[H]\begin{center}
 \includegraphics[scale=0.3]{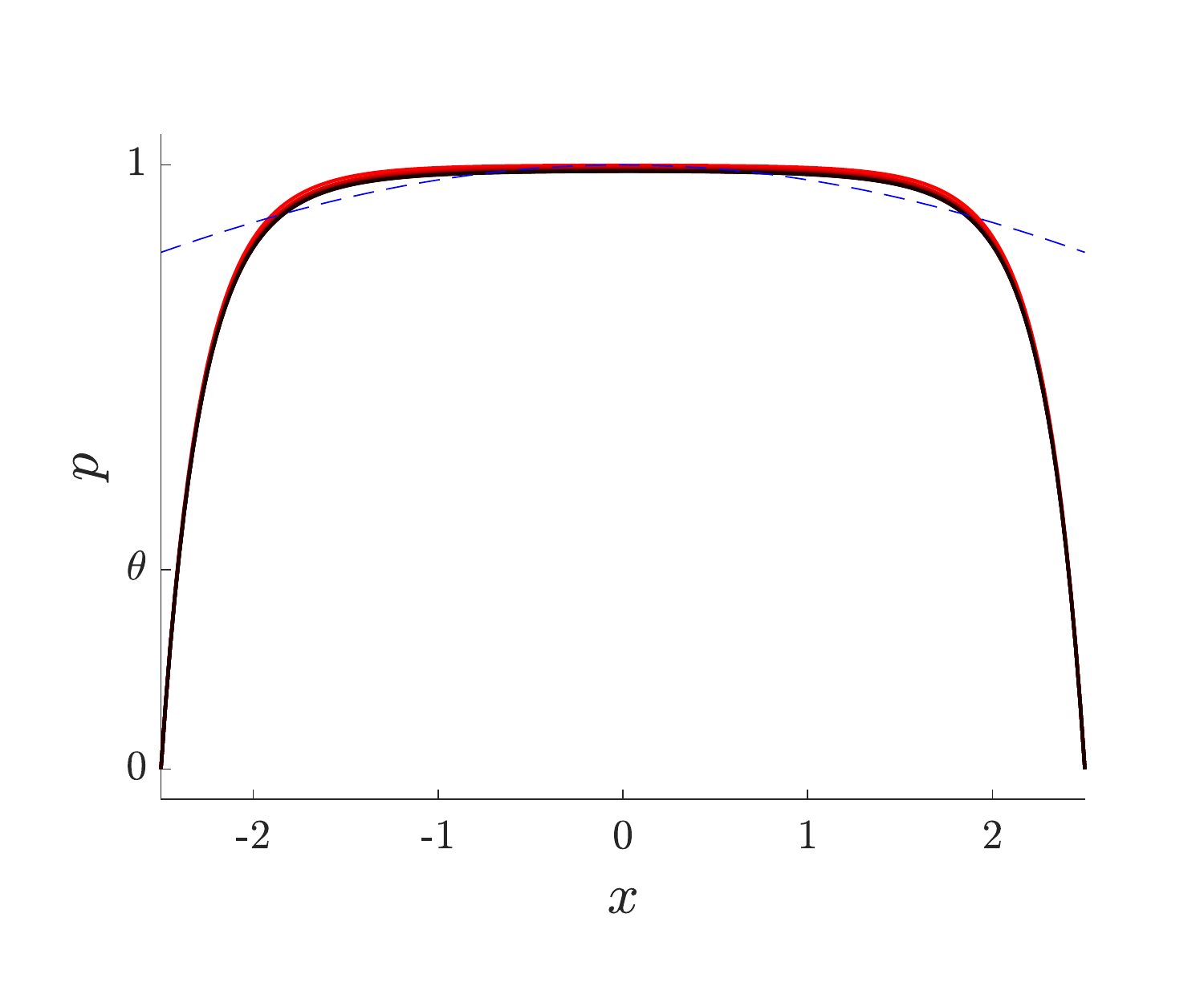}
  \includegraphics[scale=0.3]{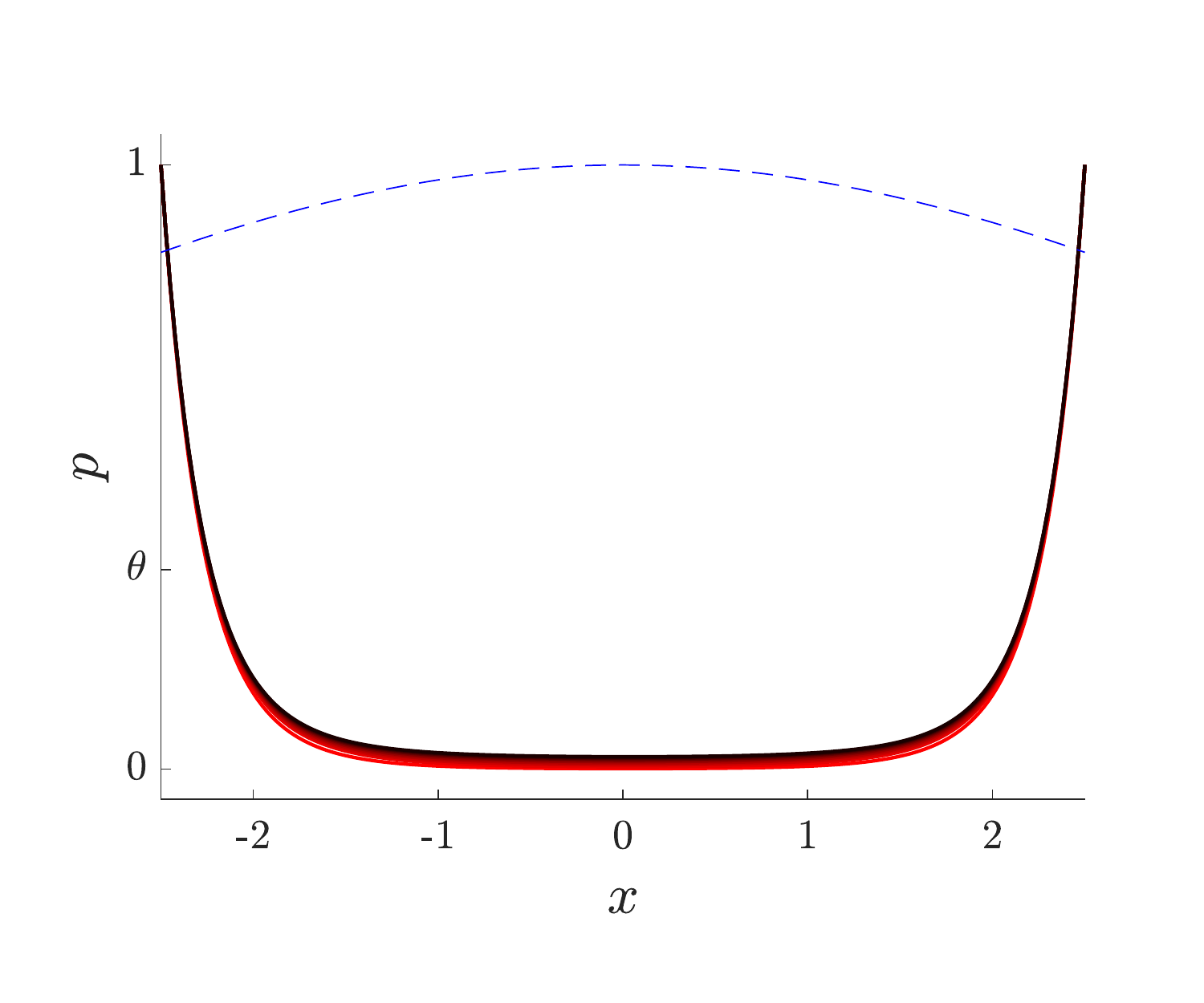}
\caption{The blue dashed line represents $N^{1/\sigma}=e^{-x^2/\sigma}$, for $\sigma=40$, 
  $L=2.5$, (Left) initial datum $p_0=1$, (Right) initial datum $p_0=0$. We try to control the initial conditions solutions to 0 (left) or 1 (right). Darker red represents a further instance of time and black represents the final time. We clearly observe the lack of controllability due to the presence of a barrier.}\label{Fi:6}\end{center}
\end{figure}
There can also be controllability from 0 to $\theta$, but not from $1$ to $\theta$ for some drifts, as shown, numerically, below:\begin{center}

\begin{figure}[H]\begin{center}
 \includegraphics[scale=0.3]{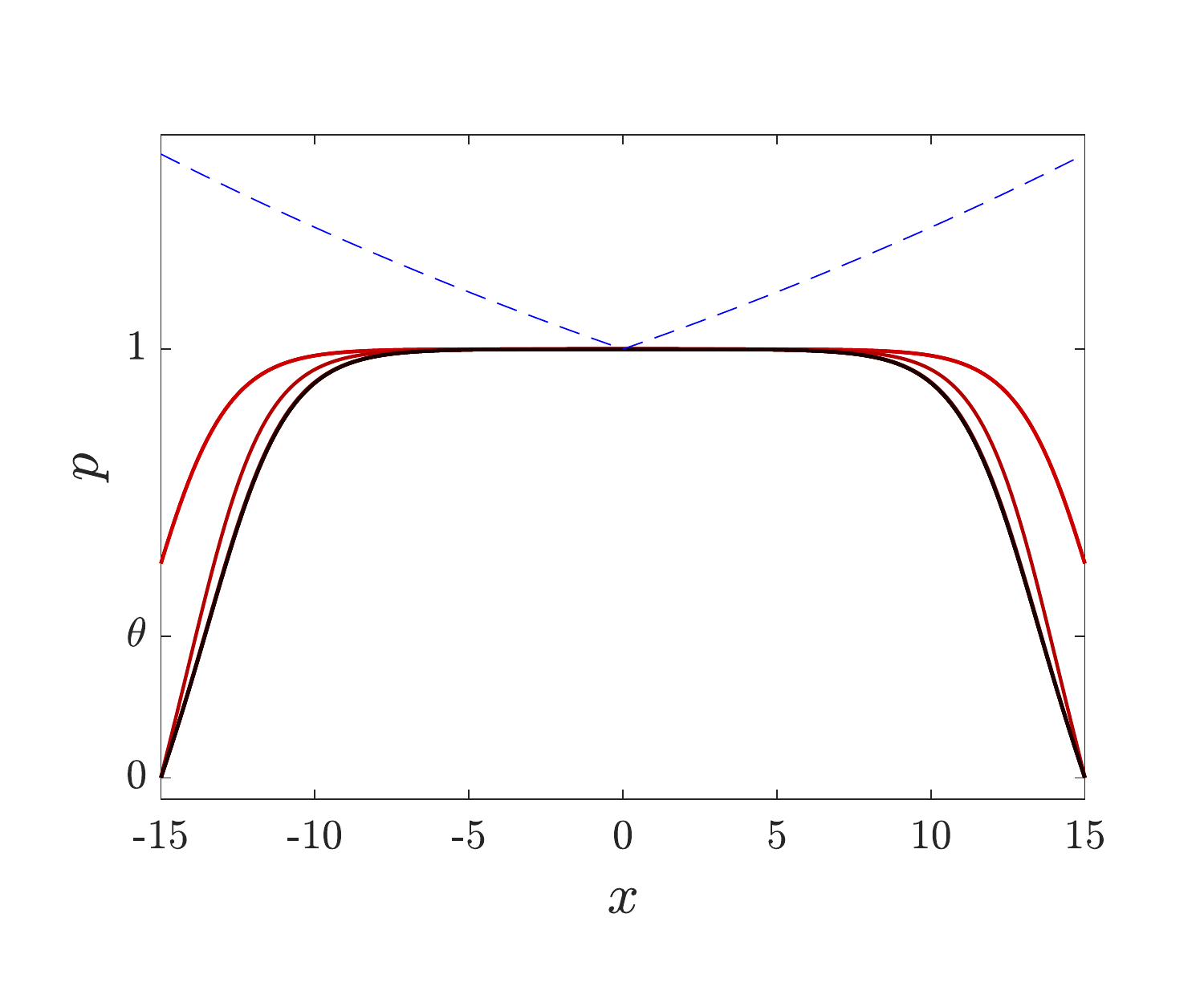}
  \includegraphics[scale=0.3]{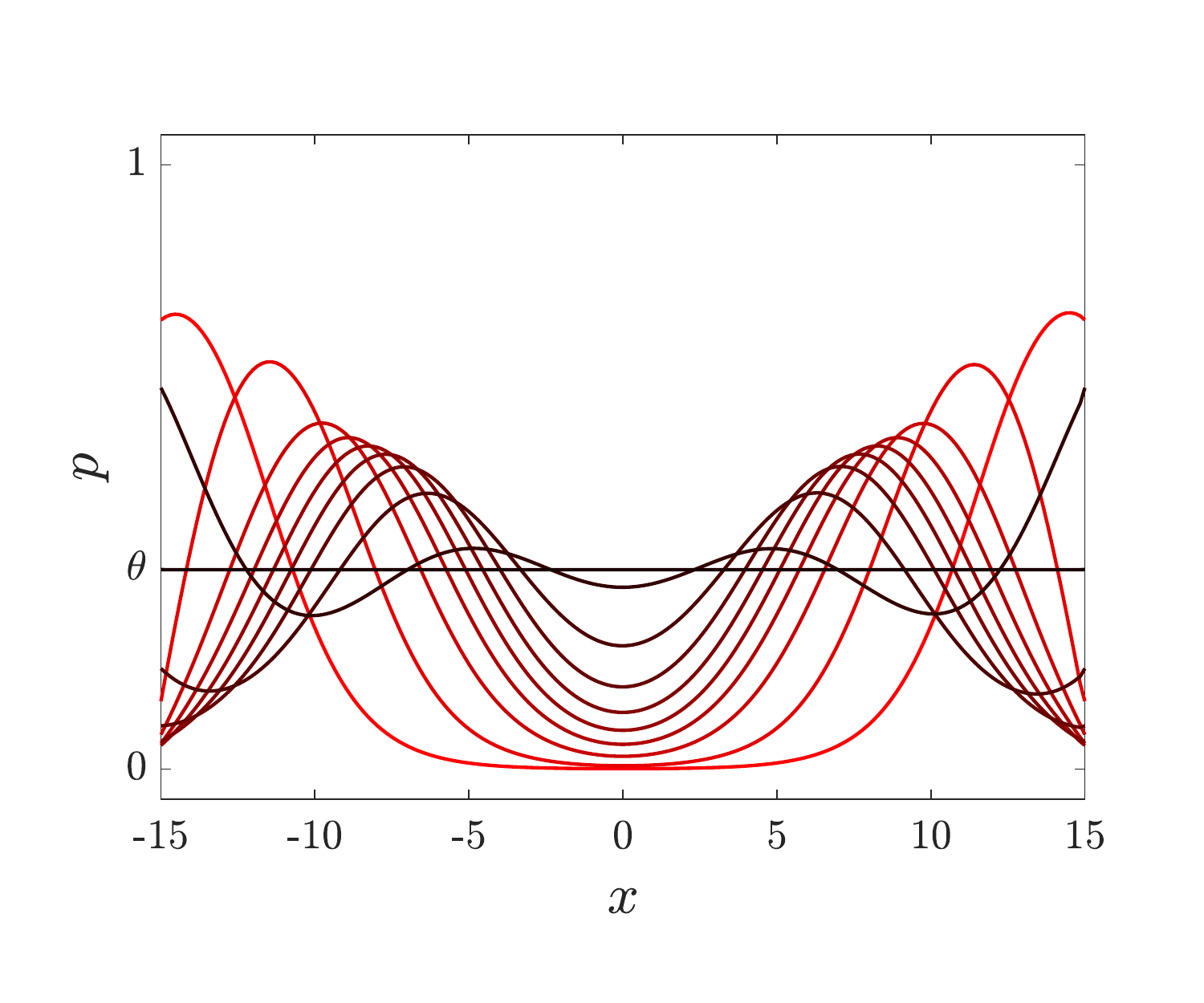}\end{center}
\caption{The blue dashed line represents $N^{1/\sigma}=e^{|x|/\sigma}$, $\sigma=40$, $T=150$, $L=15$. (Left) initial datum $p_0=1$, (Right) initial datum $p_0=0$. Darker red represents a further instance of time and black being the final time.}\label{Fi:7}
\end{figure}\end{center}
\color{black}
\paragraph{Unblocking phenomenons}
Assumptions \eqref{AssumptionDrift} and \eqref{AssumptionDrift2} essentially state that, when the drift is, loosely speaking, "pushing" towards the boundary intensely enough, barriers will appear and prevent controllability to 0,1 or $\theta$. We now give, for the sake of completeness, an example of a drift which is pushing "towards" the interior of the domain, and which helps controllability, in the sense that, if it is intense enough, all barriers will disappear. This last situation will be referred to as "unblocking phenomena". For the sake of readability, we once again prove our result in the case of a ball.
\begin{theorem}\label{Th:Fin}
Let $N(x):=e^{\frac{\Vert x\Vert^2}2}$. There exists $\sigma_+>0$ such that, for every  $\O:=\mathbb B(0;R)$ with $R>0$ and for any $0<\sigma\leq \sigma_+$, \eqref{EqSigma} is controllable to 0 and 1 in infinite time   and to $\theta$ in finite time.
\end{theorem}
\begin{remark}
The term "unblocking" is justified by the fact that, as noted already, when the drift is not present, some barriers may appear depending on the value of $R$ see \cite{PoucholTrelatZuazua,BRZ}. The lack of barriers allows us to control to 0, 1 or $\theta$ via the static strategy $u\equiv$0, 1 or $\theta$ in infinite time and, to control to $\theta$, we may apply a local exact controllability result (see Proposition \ref{Pr:ControlExact} below).
\end{remark}
\begin{remark}
As will be explained in the proof, the key point in Theorem \ref{Th:Fin} is that  the following inequality holds:
$$\lambda_1(\R^d,N):=\inf_{\psi\in W^{1,2}_0(\R^d)\,, \psi \neq 0}\frac{\int_{\R^d}N^2|\n \psi|^2}{\int_{\R^d}N^2\psi^2}>0.$$
As a consequence of \cite[Corollary 1.10]{EscobedoKavian}, it is thus possible to restate our result as follows: let $\Lambda\in \mathscr C^2(\R^d;\R_+)$ be such that 
$$\lim_{\Vert x\Vert\to \infty}\left(\Delta \Lambda+\frac12|\n \Lambda|^2\right)=+\infty$$ and define $N_\sigma:=e^{\Lambda\left(\frac\cdot\sigma\right)}$. Then, there exists $\sigma_\Lambda>0$ such that for any $0<\sigma\leq \sigma_\Lambda$ and any $R>0$, if $\O=\mathbb B(0;R)$, the equation
\begin{equation}
\begin{cases}\partial_tp-\Delta p-2 \left\langle \nabla p,\frac{\nabla N_\sigma}{N_\sigma}\right\rangle=f(p) \text{ in }\R_+\times\O\,,\\ p(t,\cdot)=u(t,\cdot)\in \R_+\times \partial \O\,, \\0\leq u\leq 1\text{ in }\R_+\times \partial \O,\\ p(0,\cdot)=p_0\,, 0\leq p_0\leq 1\text{ in }\O.\end{cases}.\end{equation} is controllable to 0, $\theta$ or 1.
\end{remark}

\color{black}


\paragraph{The case of radial drifts}

In the case where the total population size $N:\O\to \R_+^*$ can be extended into a radial function $N:\R^d\to \R_+^*$, we can give an explicit bound on the decay rate of $N$ to ensure the controllability of \eqref{Eq:AdvectionControl}.
In other words, when the total population size is the restriction to the domain $\O$ of a radial function, we can obtain controllability results.
\begin{theorem}\label{Th:Rate}
Let $\O$ be a bounded smooth domain in $\R^d$. Let $N\in \mathscr C^1(\R^d; \R_+^*)\,, \inf N>0$ and $N$ be radially symmetric. Let 
$$\lambda_1^D(\O,N):=\inf_{p\in W^{1,2}_0(\O)}\displaystyle\frac{\displaystyle\int_\O N^2|\n p|^2}{\displaystyle\int_\O N^2 p^2}$$ be the weighted eigenvalue associated with $N$.
If 
\begin{equation}\label{Eq:ConSpectral}||f'||_{L^\infty}\color{black}<\color{black} \lambda_1^D(\O,N)\end{equation}
 and if 
\begin{equation}\label{Eq:Rate}\tag{$A_1$}N'(r)\geq -\frac{d-1}{2r} N(r),\end{equation}
then Equation \eqref{Eq:AdvectionControl} is controllable to $0$ in infinite time and to $\theta$ in finite time.
\end{theorem}
This Theorem is proved using energy methods and adapting the proofs of \cite{BRZ}.

\subsubsection{High-infection rate models}
For the infection-dependent model \eqref{Eq:GF}, i.e when $N$ assumes the form
$$N=N(p)>0,$$  the main equation of \eqref{Eq:AdvectionControl}
 reads
$$\frac{\partial p}{\partial t}-\Delta p -2\frac{N'}N(p)|\n p|^2=f(p).$$
Then the controllability properties of the equation are the same as in \cite{BRZ}:
\begin{proposition}\label{Th:GF}
Let $\O\subset \R^d$ be a smooth bounded set. 
When $N\in \mathscr C^1(\R)$ satisfies \eqref{Eq:GF}, there exists $\rho^*=\rho^*(f)$ such that, for any smooth bounded domain $\O$, 
\begin{enumerate}
\item \underline{Lack of controllability for large inradii:}  If $\rho_\O>\rho^*$, then \eqref{Eq:AdvectionControl} is not controllable to 0 in (in)finite time in the sense of Definition \ref{De:Contr}: there exist initial data $0\leq p_0\leq 1$ such that, for any control $u$ satisfying the constraints \eqref{Eq:ConstraintControl}, the solution $p$ of \eqref{Eq:AdvectionControl} does not converge to 0 as $t\to \infty$.
\item\underline{Controllability for large Dirichlet eigenvalue} If $\lambda_1^D(\O)>||f'||_{L^\infty}$, then \eqref{Eq:AdvectionControl} is controllable to 0, 1 in infinite time for any initial datum $0\leq p_0\leq 1$,  and to $\theta$ in finite time  for any initial datum $0\leq p_0\leq 1$.
\end{enumerate}
\end{proposition}
A possible interpretation of this result is that even if the domain has a large measure,  if it is also very thin, it makes sense that a boundary control should work while if it has a big bulge, it is intuitive that a lack of boundary controllability should occur. 

\def\t{{\theta}}

\section{Proof of Theorem \ref{Th:SHSlow}: slowly varying total population size}\label{Proof1}
\subsection{Lack of controllability to 0 for large inradius}
We prove here the first point of Theorem \ref{Th:SHSlow}. Recall that we want to prove that, if the inradius $\rho_\O$ is bigger than a threshold $\rho^*$ depending only on $f$, then equation \eqref{Eq:AdvectionControlSlow} is not controllable to 0 in (in)finite time.
\\Following \cite{PoucholTrelatZuazua}, we claim that this lack of controllability occurs when the equation
\begin{equation}\label{Eq:Le}
\left\{
\begin{array}{ll}
-\Delta \p_0-\e\langle \n n\,, \n \p_0\rangle=f(\p_0)&\text{ in }\O\,,
\\\p_0=0&\text{ on }\partial \O\,,
\\0\leq \p_0\leq1,&
\end{array}
\right.
\end{equation}
has a non-trivial solution, i.e a solution such that $\p_0\neq 0$. 
Indeed, we have the following Claim:
\begin{claim}\label{Cl:Inter}
If there exists a non-trivial solution $\p_0\neq 0$ to \eqref{Eq:Le}, then \eqref{Eq:AdvectionControlSlow} is not controllable to 0 in infinite time.
\end{claim}
\begin{proof}[Proof of Claim \ref{Cl:Inter}]
This is an easy consequence of the maximum principle. Indeed, let $\eta$ be a non-trivial solution of \eqref{Eq:Le} and let $p_0$ be any initial datum satisfying 
$$\eta\leq p_0\leq 1.$$
Let $u:\R_+\times \partial \O\to [0,1]$ be a boundary control. Let  $p^u$ be the associated solution of \eqref{Eq:AdvectionControlSlow}.
From the parabolic maximum principle \cite[Theorem 12]{Protter}, we have for every $t\in \R_+$,
$$\color{black}\p_0(\cdot)\color{black}\leq p^u(t,\cdot),$$
so that $p^u$ cannot converge to 0 as $t\to \infty$. This concludes the proof.
\end{proof}

It thus remains to establish the following Lemma:
\begin{lemma}\label{Le:NonTrivial}
There exists $\rho^*=\rho^*(n,f)$ such that, for any $\O$ satisfying 
$$\rho_\O>\rho^*$$ there exists a non-trivial solution $\p_0\neq 0$ to  equation \eqref{Eq:Le}. 
\end{lemma}
Since the proof of this Lemma is a straightforward adaptation of \cite[Proposition 3.1]{BRZ}, we postpone it to Appendix \ref{Ap:NonTrivial}.
\subsection{Controllability to 0 and 1}
We now prove the second part of Theorem \ref{Th:SHSlow}, which we rewrite as the following claim:
\begin{claim}\label{Cl:UniZero}Assume $\lambda_1^D(\O)>\Vert f'\Vert_{L^\infty}$.
\begin{enumerate}\item\underline{Controllability to 0:}
There exists $\rho_*=\rho_*(n,f)$ and $\e^*_0>0$ such that, for any $\O$, if $\rho_\O \leq \rho_*$ and $0<\e\leq \e^*_0$, Equation \eqref{Eq:AdvectionControlSlow} is controllable to 0  in infinite time. 
\item \underline{Controllability to 1:}
There exists $\e_1^*>0$ such that, for any  $0<\e\leq \e_1^*$, Equation \eqref{Eq:AdvectionControlSlow} is controllable to 1 in infinite time.
\end{enumerate}
\end{claim}
\begin{proof}[Proof of Claim \ref{Cl:UniZero}]
\begin{enumerate}
\item \underline{Controllability to 0:} 

The key part is to prove the following: 
\textit{ There exists $\rho_*>0$ such that, if $\rho_\O\leq\rho_*$, then $y\equiv 0$ is the only solution to }
\begin{equation}
\left\{\begin{array}{ll}\label{Eq:Intermediaire}
-\Delta y-\e\langle \n n\,, \n y\rangle=f(y)\,,&\text{ in }\O,\\y=0 &\text{ on }\partial \O.
\end{array}\right.
\end{equation}
Indeed, assuming that the uniqueness of  \eqref{Eq:Intermediaire} holds, consider the static control $u\equiv 0$ and the solution of 
$$\left\{\begin{array}{ll}
\frac{\partial p}{\partial t}-\Delta p-\e\langle \n n\,, \n p\rangle=f(p)\,,&\text{ in }\R_+\times \O\,, \\p=0& \text{ on }\R_+\times \partial \O,\\p(t=0,\cdot)=p_0&\text{ in }\O.\end{array}\right.$$
From standard parabolic regularity and the Arzela-Ascoli theorem, $p$ converges uniformly in $\O$ to a solution $\overline p$ of \eqref{Eq:Intermediaire}
However, by the uniqueness of  \eqref{Eq:Intermediaire}, we have $\overline p=0$, whence $$p(t,\cdot)\underset{t\to \infty}{\overset{\mathscr C^0(\overline \O)}\rightarrow}0,$$
which means that the static strategy drives $p_0$ to $0$.

We claim that the uniqueness of solutions to \eqref{Eq:Intermediaire} follows from  spectral arguments: first of all, uniqueness holds for \eqref{Eq:Intermediaire}
if the first eigenvalue $\lambda(\e,n,\O)$ of the operator
$$\mathcal L_{\e,n}:p\mapsto-\n \cdot\left(e^{\e n}\n p\right)$$ with Dirichlet boundary conditions satisfies
$$\lambda(\e,n,\O)>||f'||_{L^\infty}e^{\e ||n||_{L^\infty}},$$ as is standard from classical theory for non-linear elliptic PDE  \cite{LionsBerestycki}.
\\We now notice that, $n$ being \color{black}bounded\color{black}, the Rayleigh quotient formulation for the eigenvalue 
$$\lambda(\e,n,\O)=\inf_{p\in W^{1,2}_0(\O)}\frac{\int_\O e^{\e n}|\n p|^2}{\int_\O p^2}$$ yields that 
$$\lambda(\e,n,\O)\geq e^{-\e\Vert n\Vert_{L^\infty}} \lambda_1^D(\O)$$ where $\lambda_1^D(\O)$ is the first eigenvalue of the Laplace operator with Dirichlet boundary conditions. Thus we are reduced to checking that 
$$\lambda_1^D(\O)>||f'||_{L^\infty}e^{\e ||n||_{L^\infty}}$$ for $\e>0$ small enough. If the condition $\lambda_1^D(\O)>||f'||_{L^\infty}$ is fulfilled, taking the limit as $\e\to 0$ yields the desired result.
\item \underline{Controllability to 1}
Using the same arguments, we claim that controllability to 1 can be achieved through the static control $u\equiv 1$ provided the only solution to 
\begin{equation}\label{Eq:UniUn}\left\{\begin{array}{ll}
-\Delta\overline p-\e\langle \n n\,, \n \overline p\rangle=f(\overline p)\,,&\text{ in }\O\,,\\\overline p=1&\text{ on }\partial \O,\\0\leq \overline p\leq 1&\\\end{array}\right.\end{equation} is $\overline p\equiv 1$.
\\We already know (see \cite{BRZ,PoucholTrelatZuazua}) that uniqueness holds for $\e=0$. Now this implies that uniqueness holds for $\e$ small enough. Indeed, argue by contradiction and assume that, for every $\e>0$ there exists a non-trivial solution $\overline p_\e$ to \eqref{Eq:UniUn}. Since $\overline p_\e \neq 1$, \color{black}$\overline p_\e$ \color{black} reaches a minimum at some $\overline x_\e\in \O$, and so 
$$f(\overline p_\e(\overline x_\e))<0$$ which means that 
$$\overline p_\e(\overline x_\e)<\theta.$$
Standard elliptic estimates entail that, as $\e\to 0$, $p_\e$ converges in $W^{1,2}(\O)$ and in $\mathscr C^0(\overline \O)$ to $\overline p$ satisfying
\begin{equation}\label{Eq:UniZero}\left\{\begin{array}{ll}
-\Delta\overline p=f(\overline p)&\text{ in }\O\,,\\\overline p=1&\text{ on }\partial \O,\\0\leq \overline p\leq 1&\\\end{array}\right.\end{equation}
and such that there exists a point $\overline x$ satisfying 
$$\overline p(\overline x)\leq\theta$$ which is a contradiction since we know uniqueness holds for \eqref{Eq:UniUn}. This concludes the proof.\end{enumerate}

\end{proof}

\subsection{Proof of the controllability  to $\theta$ for small inradii}

\subsubsection{Structure of the proof: the staircase method}
We recall that we want to control the semilinear heat equation
\begin{equation}\label{Eq:Example}
\left\{\begin{array}{ll}
\frac{\partial p}{\partial t}-\Delta p-\e\langle \n n\,, \n p\rangle=f(p)&\text{ in }\R_+\times\O\,,
\\p(t,\cdot)=u(t,\cdot)&\text{ on }\R_+\times \partial \O,
\\p(t=0,\cdot)=y_0&
\end{array}
\right.
\end{equation}
to $z_\theta\equiv \theta$ with additional constraints on the control $u$, which we drop for the time being. We first state a local exact controllability result  \cite[Lemma 1]{PoucholTrelatZuazua}, \cite[Lemma 2.1]{PighinZuazua}:

\begin{proposition}\label{Pr:ControlExact}[Local exact controllability]
Let $T>0$. There exists $\delta_1(T)>0\,, C(T)>0$ such that for all steady state  $y_f$ of  \eqref{Eq:Example}, for all $0\leq y_d\leq 1$  satisfying 
$$||y_d-y_f||_{\mathscr C^0}\leq \delta_1(T)$$ then \eqref{Eq:Example} is controllable from $y_d$ to $y_f$ in finite time $T<\infty$ through a control $u$. Furthermore, letting $\overline u=y_f|_{\partial \O}$, the control function  $u=u(t)$ satisfies
\begin{equation}\label{Eq:EstimControl}||u(t)-\overline u||_{\mathscr C^0(\partial \O)}\leq C(T)\delta_1(T).
\end{equation}
\end{proposition}
We now assume that $\rho_\O\leq \rho^*$, that is, thanks to  Step 1 of the proof of Claim \ref{Cl:UniZero}, we assume that we have uniqueness for  \eqref{Eq:Intermediaire}.
We then proceed along three different steps:
\begin{enumerate}
\item[$\bullet$]\underline{Step 1:}
Starting from any initial condition $0\leq p_0\leq 1$, we first set the static control
$$u(t,x)=0.$$ 
Since $n$ is $\mathscr C^1$, standard parabolic estimates and  the Arzela-Ascoli theorem ensures that the solution $p^u$ of \eqref{Eq:AdvectionControlSlow} converges uniformly, as $t\to \infty$ to a solution $\overline \eta $ of \eqref{Eq:Intermediaire}.
However, from Claim \ref{Cl:UniZero}, $\rho_\O\leq \rho_*(n,f)$ implies that $z_0\equiv 0$ is the unique solution of this equation. Thus, this static control guarantees that, for every $\delta>0$, there exists $T_1>0$ such that, for any $t\geq T_1$
$$||p^u(t,\cdot)||_{L^\infty}\leq \frac{\delta}2.$$
\item[$\bullet$] \underline{Step 2:} We prove that there exists a steady state $p_0$ of \eqref{Eq:Example}, \color{black} that is, a solution of \[
-\Delta p_0-\e\langle \n n\,, \n p_0\rangle=f(p_0)\text{ in }\O\,,
\]
where we do not specify the boundary conditions, but \color{black}  such that 
$$0<\inf_{x \in \O} p_0(x)\leq ||p_0||_{L^\infty}\leq \frac\delta2$$ where $\delta>0$ is chosen to apply Proposition \ref{Pr:ControlExact}; we drive $p^u(T_1,\cdot)$ to $ p_0$ in finite time.
\item[$\bullet$] \underline{Step 3:} We drive $p_0$ to $\theta$ using the staircase method. \end{enumerate}
In this setting, we are thus reduced to the controllability of any initial datum to a small enough $p_0$ to $\theta$ in finite time.

\paragraph{The staircase method}
We want to use the staircase method of Coron and Tr\'elat, see \cite{CoronTrelat} for the one-dimensional case and \cite{PighinZuazua} for a full derivation. We briefly recall the  most important features of this method: assume that there exists a $\mathscr C^0$-continuous path of steady-states of \eqref{Eq:Example} $\Gamma=\{p_s\}_{s\in [0,1]}$  such that $p_0=y_0$ and $p_1=y_1$. Then \eqref{Eq:Example} is controllable from $y_0$ to $y_1$ in finite time. Indeed, as is usually done, we consider a time $T_1>0$ and a subdivision 
$$0=s_{i_1}<\dots<s_{i_K}=1$$ of $[0,1]$ such that 
$$\forall i\in \{1,\dots,K-1\}\,, ||p_{s_i}-p_{s_{i+1}}||_{\mathscr C^0(\O)}\leq \delta_1$$ where $\delta_1=\delta_1(T_1)$ is the controllability parameter  given by Proposition \ref{Pr:ControlExact}.  We then control each $p_{s_i}$ to $p_{s_{i+1}}$ in finite time by Proposition \ref{Pr:ControlExact}. This result does not necessarily yield constrained controls, but, thanks to estimate \eqref{Eq:EstimControl} we can enforce these constraints, by choosing a control parameter $\delta_1$ small enough. Thus, the key part is to find a continuous path of steady-states for the perturbed system with slowly varying total population size \eqref{Eq:AdvectionControlSlow}. However, it suffices to have a finite number of steady-states that are close enough to each other, starting at $y_0$ and ending at $y_1$. We represent the situation in Figure \ref{Fi:12} below:\begin{center}

\begin{figure}[H]\begin{center}
\begin{tikzpicture}[scale=2]
\draw[gray,thin,dashed] plot [smooth, tension=1] coordinates{ (2,0) (1.5,0.5) (1.8,1) (2.2,1.5) (1.7,2) };

\draw (2,0) node[anchor=south,Plum]{$p_0$};
\draw (2,0) node[Plum]{$\bullet$};

\draw (1.5,0.5) node[anchor=east,Plum]{$p_{s_1}$};
\draw (1.5,0.5) node[Plum]{$\bullet$};

\draw (1.8,1) node[anchor=west,Plum]{$p_{s_2}$};
\draw (1.8,1) node[Plum]{$\bullet$};

\draw (2.2,1.5) node[anchor=west,Plum]{$p_{s_3}$};
\draw (2.2,1.5) node[Plum]{$\bullet$};

\draw (1.7,2) node[anchor=south,Plum]{$p_1$};
\draw (1.7,2) node[Plum]{$\bullet$};

\draw (2,-0.05) edge [->,magenta,out=250, in=-100]  (1.5,0.45);

\draw (1.54,0.5) edge [->,magenta,out=-10, in=-55]  (1.8,0.95);

\draw (1.8,1.042) edge [->,magenta,out=100, in=-200]  (2.15,1.5);
\draw (2.16,1.52) edge [->,magenta,out=120, in=-150]  (1.66,2);

\end{tikzpicture}\end{center}
\caption{The dashed curve is the path of steady states (for instance in $W^{1,2}(\O)\cap \mathscr C^0(\overline \O)$), and the points are the close enough steady states. We represent the exact control in finite time $T_1>0$ with the pink arrows.}\label{Fi:12}\end{figure}
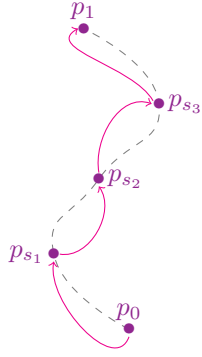\end{center}

\subsubsection{Perturbation of a path of steady-states}
We are going to perturb the path of steady-states using the implicit function Theorem in order to get a sequence of close enough steady-states.
\begin{remark}
Here, if we were to try and prove, for $\e$ small enough, the existence of a continuous path of steady states, the idea would be to start from a path  $(p_{0,s})_{s\in [0,1]}$ for $\e=0$ (which we know exists from \cite{BRZ,PoucholTrelatZuazua}) and to try and perturb it into a path for $\e>0$ small enough, thus giving us a path $\{p_{\e,s}\}_{s\in [0,1]\,, \e>0}$. However, doing it for the whole path requires some kind of implicit function theorem or, at least, some bifurcation argument. Namely, to construct the path, we would need to ensure that either 
$$\mathcal L^{s,\e}:=-\n \cdot(e^{\e n}\n )-e^{\e n}f'(p_{0,s})$$  has no zero eigenvalue for $\e=0$ or that it has a non-zero crossing number (namely, a non zero number of eigenvalues enter or leave $\R_+^*$ as $\e$ increases from $-\delta$ to $\delta$). In the first case, the implicit function theorem would apply;  in the second case, Bifurcation Theory (see \cite[Theorem II.7.3]{Kielhfer}) would ensure the existence of a branch $p_{\e,s}$ for $\e$ small enough. These conditions seem too hard to check for an arbitrary path of continuous of steady states. Hence, we focus on perturbing a finite number of points close enough on the path since, as we noted, this is enough to ensure exact controllability. 
\end{remark}

Henceforth, our goal is the following proposition:
\begin{proposition}\label{Pr:Bifurcation}
Let $\delta>0$. There exists $K\in\mathbb{N}$ and $ \e^*_\theta>0$ such that, for any $0<\e\leq  \e^*_\theta$, there exists a sequence $\{p_{\e,i}\}_{i=1,\dots K}$ satisfying:
\begin{enumerate}
\item[$\bullet$] For every $i=1,\dots,K$, $p_{\e,i}$ is a steady-state of \eqref{Eq:AdvectionControlSlow}:
$$-\Delta p_{\e,i}-\e\langle \n n\,, \n p_{\e,i}\rangle=f(p_{\e,i}),$$
\item[$\bullet$]
$p_{\e,K}=z_\theta\equiv \theta\,,\quad 0<\inf p_{\e,1}\leq ||p_{\e,1}||_{L^\infty}\leq \delta,$
\item[$\bullet$]For every $i=1,\dots,K$,
$$\frac{\delta}2\leq p_{\e,i}\leq ||p_{\e,i}||_{L^\infty}\leq 1-\frac\delta2,$$
\item[$\bullet$] For every $i=1,\dots,K-1$, 
$$\left|\left| p_{\e,i+1}-p_{\e,i}\right|\right|_{L^\infty}\leq \delta.$$
\end{enumerate}
\end{proposition}

As explained, this Proposition gives us the desired conclusion:
\begin{claim}\label{Cl:Conc}
Proposition \ref{Pr:Bifurcation} implies the controllability to $\theta$ for any initial datum $p_0$ in Equation \eqref{Eq:AdvectionControl}.\end{claim}

\begin{center}

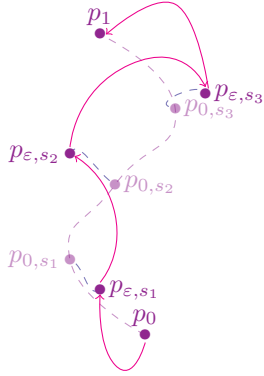
\begin{figure}[H]\begin{center}
\begin{tikzpicture}[scale=2]
\draw[Plum, opacity=0.5,thin,dashed] plot [smooth, tension=1] coordinates{ (2,0) (1.5,0.5) (1.8,1) (2.2,1.5) (1.7,2) };
\draw (2,0) node[anchor=south,Plum]{$p_{0}$};
\draw (2,0) node[Plum]{$\bullet$};

\draw (1.5,0.5) node[opacity=0.5,anchor=east,Plum]{$p_{0,s_1}$};
\draw (1.5,0.5) node[opacity=0.5,Plum]{$\bullet$};

\draw (1.7,0.3) node[opacity=1,anchor=west,Plum]{$p_{\e,s_1}$};
\draw (1.7,0.3) edge [dashed,Periwinkle,out=-150, in=-20]  (1.5,0.5);
\draw (1.7,0.3) node[opacity=1,Plum]{$\bullet$};

\draw (1.8,1) node[opacity=0.5,anchor=west,Plum]{$p_{0,s_2}$};
\draw (1.8,1) node[opacity=0.5,Plum]{$\bullet$};
\draw (1.5,1.2) node[opacity=1,anchor=east,Plum]{$p_{\e,s_2}$};
\draw (1.5,1.2) edge [dashed,Periwinkle,out=-290, in=-210]  (1.8,1);
\draw (1.5,1.2) node[opacity=1,Plum]{$\bullet$};

\draw (2.2,1.5) node[opacity=0.5,anchor=west,Plum]{$p_{0,s_3}$};
\draw (2.2,1.5) node[opacity=0.5,Plum]{$\bullet$};
\draw (2.4,1.6) node[opacity=1,anchor=west,Plum]{$p_{\e,s_3}$};
\draw (2.157,1.52) edge [dashed,Periwinkle,out=-200, in=100]  (2.4,1.6);
\draw (2.4,1.6) node[opacity=1,Plum]{$\bullet$};

\draw (1.7,2) node[opacity=1,anchor=south,Plum]{$p_1$};
\draw (1.7,2) node[opacity=1,Plum]{$\bullet$};

\draw[->,magenta,thin] plot [smooth, tension=1.5] coordinates{ (2,-0.05) (1.8,-0.2) (1.7,0.27) };
\draw[->,magenta,thin] plot [smooth, tension=1.2] coordinates{ (1.7,0.35) (1.8,0.8) (1.53,1.19) };

\draw[->,magenta,thin] plot [smooth, tension=1.5] coordinates{ (1.5,1.25) (1.9,1.8) (2.4,1.65) };
\draw[->,magenta,thin] plot [smooth, tension=0.8] coordinates{ (2.42,1.64) (2.2,2.2) (1.74,2) };




\end{tikzpicture}\end{center}
\caption{In dark purple, the perturbed steady states, linked to the unperturbed steady states. We do not know whether or not a continuous path of steady states linking these new states exists; however, such points enable us to do exact controllability again and to apply the \color{black}staircase method.}\end{figure}\end{center}


We  strongly rely on the explicit construction of the path of steady-states for $\e=0$ in \cite{BRZ,PoucholTrelatZuazua}. \paragraph{Known constructions of a path of steady-states ($\e=0$)}

For the multi-dimensional case, it has been shown in \cite{BRZ} that one can construct a path of steady-states linking $z_0\equiv 0$ to $z_\theta\equiv \theta$ in the following way: let $\O$ be the domain  where the equation is set  and let $R_\O>0$ be such that 
$$\O\subseteq \mathbb B(0;R_\O).$$
The path of steady state is defined as follows in \cite{BRZ}: first of all, if uniqueness holds for 
$$\begin{cases}
-\Delta p=f(p)&\text{ in }\B(0;R_\O)\,,
\\ p=0&\color{black}\text{ on }\partial \B(0;R_\O).\end{cases}$$ then, for $\eta>0$ small enough, there exists a unique solution to
$$\begin{cases}
-\Delta p_\eta=f(p_\eta)&\text{ in }\mathbb B(0;R_\O)\\
p_\eta=\eta &\text{ on }\partial \mathbb B(0;R_\O).\end{cases}$$ Define, for any $s\in [0,1]$,  $p^{0,s}$ as the unique solution to the problem
\begin{equation}\label{Eq:ConstructionBRZ}
\left\{\begin{array}{ll}
-\Delta p^{0,s}=f\left(p^{0,s}\right)\text{ in }\mathbb B(0;R_\Omega)\,,&
\\ p^{0,s}(0)=s\theta+(1-s)p_\eta(0)\,,&
\\p^{0,s}\text{ is  radial.}&
\end{array}
\right.
\end{equation}
Using the polar coordinates, the authors prove that the equation above has a unique solution, and that the map $s\mapsto p^{0,s}$ is continuous in the $\mathscr C^0$ topology. Using energy type methods, they prove that this solution is admissible, i.e that for any $0<t_0<1$,
$$0<\inf_{s \in [t_0;1]\,, x\in \mathbb B(0;R)} p^{0,s}(x)\leq \sup_{s \in [0,1]\,, x\in \mathbb B(0;R)} p^{0,s}(x)<1.$$ This  gives a path on $\mathbb B(0;R_\O)$. To construct the path on $\O$, it suffices to set
$$\tilde p^{0,s}:=\left.p^{0,s}\right|_{\O}.$$ Furthermore, by elliptic regularity or by studying the equation in polar coordinates, we see that, for every $s\in [0,1]$, $$p^{0,s}\in \mathscr C^{2,\alpha}(\mathbb B(0;R_\O))$$ for any $0<\alpha<1$.
Instead of perturbing  the functions $\tilde p^{0,s}\in \mathscr C^{2,\alpha}(\O)$, we will perturb the functions $p^{0,s}\in \mathscr C^2\big(\mathbb B(0;R_\O)\big)$.
\begin{notation}

Henceforth, the parameter $R_\O>0$ is fixed and, for any $s\in [0,1]$,  $p^{0,s}$ is the unique solution to \eqref{Eq:ConstructionBRZ}.\end{notation}
\paragraph{Proof of Proposition \ref{Pr:Bifurcation}}
\begin{proof}[Proof of Proposition \ref{Pr:Bifurcation}]
Let $\delta>0$. Let $\{s_i\}_{i=0,\dots,K}$ be a sequence of points such that 

\begin{equation}\label{Eq:P0}
0<p^{0,s_0}\leq ||p^{0,s_0}||_{L^\infty}\leq \frac\delta2,\end{equation} and
\begin{equation}\label{Eq:P1}
\forall i \in \{0,\dots,K-1\}\,, \left|\left|p^{0,s_i}-p^{0,s_{i+1}}\right|\right|_{L^\infty}\leq \frac{\delta}{4}.\end{equation}
We define, for any $i=1,\dots,K$,
$$p_{0,i}=p^{0,s_i}.$$Fix a parameter $\alpha \in (0;1)$.
We define a one-parameter family of mappings as follows: for any $i=1,\dots,K$, let  
$$\mathscr F_i:\left\{\begin{array}{ll}
\mathscr C^{2,\alpha}\left(\mathbb B(0;R_\O)\right)\times [-1;1]&\rightarrow \mathscr C^{0,\alpha}\left(\mathbb B(0;R_\O)\right)\times \mathscr C^0\left({\partial\mathbb B(0;R_\O)}\right),
\\&
\\(p,\e)&\mapsto\left(-\n \cdot\left(e^{\e n}\n p\right)-f(p)e^{\e n}\,, p|_{{\partial\mathbb B(0;R_\O)}}-p_{0,i}|_{{\partial\mathbb B(0;R_\O)}}\right).\end{array}\right.$$
We note that
$$\forall i\in \{1,\dots,K\}\,, \mathscr F_i(p_{0,i},0)=0.$$ 
We wish to apply the implicit function theorem, which is permitted provided the operator 
$$\mathscr L_i:\xi\mapsto -\Delta \xi-f'(p_{0,i})\xi$$with Dirichlet boundary conditions is invertible. If this is the case we know that there exists a continuous path \color{black}$\e\mapsto p_{\e,i}$ (for $\e\in [0;\e_0)$, where $\e_0>0$ is small enough) starting from $p_{0,i}$ such that 
$$\mathscr F_i(p_{\e,i}, \e)=0 \text{ for any $\e\in[ 0;\e_0)$}.$$ \color{black}
Denoting, for any differential operator $\mathscr A$ its spectrum by $\Sigma\left(\mathscr A\right)$,  this invertibility property amounts, thanks to elliptic regularity (see \cite{GilbargTrudinger}) to requiring that 
\begin{equation}\label{Eq:NonResonance}\forall i \in \{1,\dots,K\}\,, 0\notin \Sigma(\mathscr L_i).\end{equation} 
If Condition \eqref{Eq:NonResonance} is satisfied, then $p_{0,i}$ perturbs into $p_{\e,i}$ and we can define 
$$\tilde p_{\e,i}:=\left. p_{\e,i}\right|_\O$$ as a suitable sequence of steady states in $\O$. Since we are working with a finite number of points, taking $\e$ small enough guarantees
$$\forall i=1,\dots,K\,, ||p_{\e,i}-p_{0,i}||_{L^\infty}\leq \frac\delta4$$ and we would then have, for any $i=1,\dots,K-1$, 
\begin{align*}
||p_{\e,i+1}-p_{\e,i}||_{L^\infty}&\leq ||p_{\e,i+1}-p_{0,i+1}||_{L^\infty}
+||p_{0,i}-p_{\e,i}||_{L^\infty}
+||p_{0,i+1}-p_{0,i}||_{L^\infty}
\\&\leq \frac\delta4+\frac\delta4+\frac\delta2=\delta,
\end{align*}
which is what we require of the sequence.

Let us define the set of resonant points (i.e the points where \eqref{Eq:NonResonance} does not hold) as 
$$\Gamma:=\left\{j\in \{1,\dots,K\}\,, 0\in \Sigma(\mathscr L_j)\right\}.$$
We note that $1\notin \Gamma$ because the first eigenvalue of 
$$\mathscr L_1=-\Delta -f'(0)$$ is positive: indeed, since $f'(0)<0$ \textcolor{black}{and $\|p_{0,1}\|_{L^\infty}$ is small}, this first eigenvalue is bounded from below by the first Dirichlet eigenvalue of the ball $\mathbb B(0;R_\O)$. Hence $1\notin \Gamma$.
We proceed as follows:
\begin{enumerate}
\item Whenever $i\notin \Gamma$, we can apply the implicit function Theorem to obtain the existence of a continuous path $p_{\e,i}$ starting from $p_{0,i}$ such that
$$p_{\e,i}|_{{\partial\mathbb B(0;R_\O)}}=p_{0,i}|_{{\partial\mathbb B(0;R_\O)}}\,, \mathcal F(p_{\e,i},\e)=0,$$
so that, taking $\e$ small enough, we can ensure that, for any $i\notin \Gamma$, 
$$||p_{\e,i}-p_{0,i}||_{L^\infty}\leq \frac\delta4.$$
\item Whenever $i\in \Gamma$, we apply the implicit function theorem on a larger domain $\mathbb B(0;R_\O+\tilde \delta)$, $\tilde \delta>0$.\begin{center}

\begin{center}\begin{figure}[H]
\begin{center}
\begin{tikzpicture}[scale=1.5]
\draw[->] (0,0)--(4,0);
\draw[->] (0,0)--(0,3);
\draw[dashed,gray] (3,0)--(3,2);
\draw[dashed,gray] (3.5,0)--(3.5,2);
\draw[Plum] plot [smooth, tension=0.5] coordinates{(0,2.7)(0.1,2.65)(0.5,2)(1,2.2)(1.5,1.5) (2.2,1.7) (2.9,0.6)(3,0.5)};

\draw[Plum,dashed] plot [smooth, tension=0.5] coordinates{(3,0.48)(3.1,0.4)(3.5,0.2)};
\draw[blue]plot[smooth, tension=0.5]coordinates{(0,2.75) (0.1,2.67) (0.5,1.8) (1,2.3)(1.5,1.4)(2.2,1.7)(3,0.6)(3.5,0.2)};

\draw (3.5,0.2) node[blue]{$\bullet$};

\draw (3.5,0) node[gray]{$\bullet$};
\draw (3.8,0) node[anchor=north, gray]{$R_\O+\tilde \delta$};

\draw (3,0) node[gray]{$\bullet$};
\draw (3,0) node[anchor=north, gray]{$R_\O$};
\draw (4,0) node[anchor=south, black]{$x$};
\draw (0,3) node[anchor=east, black]{$y$};

\end{tikzpicture}\end{center}
\caption{The initial solution $p_{0,i}$ on $\mathbb B(0;R_\O)$ is continued into a solution on $\mathbb B(0;R_\O+\tilde \delta)$, and we apply the implicit function theorem on this domain to obtain the blue curve.}
\end{figure}
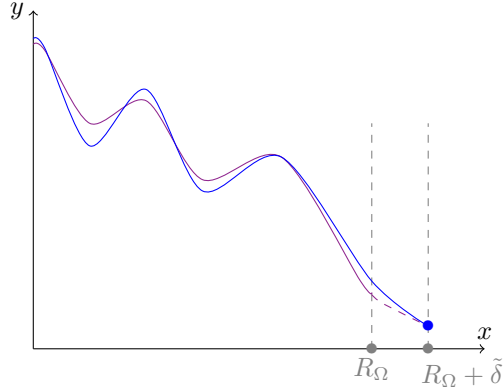\end{center}\end{center}

Let, for any $i\in \Gamma$, $\lambda_i(k,R_\O)$ be the $k$-th eigenvalue of $\mathscr L_i$ with Dirichlet boundary conditions  on $\mathbb B(0;R_\O).$  Let, for any $i\in \Gamma$,
$$k_i:=\sup\left\{k\,, \lambda_i(k,R_\O)=0\right\}.$$
Obviously, there exists $M>0$ such that $k_i\leq M$ uniformly in $i$, since $\lambda_i(k,R_\O)\to \infty$ as $k\to \infty$. We then invoke the monotonicity of the eigenvalues with respect to the domain. Let, for any $\tilde \delta$, $ p_{0,i}^{\tilde \delta}$ be the extension of $p_{0,i}$ to $\mathbb B(0;R_\O+\tilde \delta)$; this is possible given that $p_{0,i}$ is given by the radial equation \eqref{Eq:ConstructionBRZ}.

Let $\tilde{\mathcal L}_i:y\mapsto -\Delta y-f'(p_{0,i}^{\tilde \delta})y$ and $\tilde \lambda_i(\cdot,R_\O+\tilde \delta)$ be its eigenvalues. By the min-max principle of Courant (see \cite{henrot2006}) we have, for any $k\in \N$ and any $\tilde \delta>0$,
$$\tilde\lambda_i(k,R_\O+\tilde\delta)<\lambda_i(k,R_\O).$$
Hence, for every $i\in \Gamma$, there exists $\tilde\delta_i>0$ small enough so that, for any $0<\tilde \delta<\tilde \delta_i$, 
$$0\notin \Sigma\left(\tilde{\mathcal L}_i\right).$$ We then choose $\dbtilde \delta=\min_{i\in \Gamma}\tilde \delta_i$ and apply the implicit function theorem on $\mathbb B\left(0;R_\O+\frac{\dbtilde\delta}2\right)$. This gives the existence of $\tilde\e>0$ such that, for any $\e<\tilde \e$ and any $i\in \Gamma$, there exists a solution $p_{\e,i}^{\dbtilde\delta}$ of 
\begin{equation}
\left\{\begin{array}{ll}
-\Delta p_{\e,i}^{\dbtilde\delta}-\e \langle \n n\,, \n p_{\e,i}^{\dbtilde\delta}\rangle=f(p_{\e,i}^{\dbtilde\delta})\text{ in }\mathbb B(0;R_\O+\frac{\dbtilde\delta}2)\,,&
\\&
\\p_{\e,i}^{\dbtilde\delta}=p_{0,i}|_{\partial \mathbb{B}(0;R_\O+\frac{\dbtilde\delta}2)}\,,&
\\&
\\p_{\e,i}\underset{\e\to 0}{\overset{\mathscr C^0(\mathbb B(0;R_\O+\frac{ \dbtilde\delta}2))}\rightarrow}p_{0,i}^{\frac{\dbtilde\delta}2}
\end{array}\right.\end{equation}
Furthermore, 
$$p_{0,i}^{ \delta}|_{\partial \mathbb{B}(0;R_\O)}\underset{ \delta \to 0}\to p_{0,i}|_{{\partial\mathbb B(0;R_\O)}}.$$
Thus, by choosing $\dbtilde\delta$ small enough, we can guarantee that, by defining 
$$\tilde p_{\e,i}:=p_{\e,i}^{\dbtilde\delta}|_{\mathbb B(0;R_\O)}$$ we have for every $\e$ small enough
$$||\tilde p_{\e,i}-p_{0,i}||_{L^\infty}\leq \frac\delta4.$$
We note that $\tilde p_{\e,i}$ does not satisfy, on $\partial \mathbb B(0;R_\O)$ the same boundary condition as $p_{0,i}$, but this would be too strong a requirement.
\end{enumerate}
This concludes the proof of Proposition \ref{Pr:Bifurcation} and, thus, of Theorem \ref{Th:SHSlow}.
\end{proof}

\color{black}
\section{Proof of Theorem \ref{NewTheo}: blocking phenomenon}\label{Proof2}
We split the proof of this Theorem in two parts: the first one is devoted to the blocking phenomenon towards 1, the second to the blocking phenomenon towards 0.

\subsection{Proof of Theorem \ref{NewTheo}: Blocking phenomenon towards 1} 
We fix our drift $N$, as well as the constant $C$ given by Assumption \eqref{AssumptionDrift}. We define the first relevant equation:
\begin{equation}\label{NewEq1}
\begin{cases}
-\Delta p-\frac2\sigma \langle \nabla p,\frac{\nabla N}N\rangle=f(p)&\text{ in }\Omega=\mathbb B(0;R);
\\p=1&\text{ on }\partial \O.
\end{cases}\end{equation}

Non-controllability to 1 is implied by the existence of non-trivial solutions $p$ satisfying $0\leq p\leq 1$ to \eqref{NewEq1}; such solutions are called admissible. Thus the blocking phenomenon towards 1 of Theorem \ref{NewTheo} is an immediate consequence of the following Lemma:
\begin{lemma}\label{NewLem1}
Assume $N$ satisfies \eqref{AssumptionDrift}. There exists $\sigma_{N,1}>0$ such that, for any $\sigma \in (0;\sigma_{N,1})$, there exists a non-trivial admissible solution of \eqref{NewEq1}.
\end{lemma}

\subsubsection{Reduction to the Gaussian case}

The key argument in this proof is the use of a comparison principle, which will enable us to work only with Gaussian drifts, that is, with drifts of the form
$$\mathcal N_C(x):=e^{-\frac{C}2\Vert x\Vert^2}.$$  Here, $C$ is the constant given by Assumption \eqref{AssumptionDrift}.
Let us then fix this drift.

\paragraph{Blocking phenomenon towards 1 in the Gaussian case} Let us first consider the equation

\begin{equation}\label{NewEqGaussian1}
\begin{cases}
-\Delta \eta-\frac2\sigma \langle \nabla \eta,\frac{\nabla \mathcal N_C}{\mathcal N_C}\rangle=f(\eta)&\text{ in }\Omega=\mathbb B(0;R);
\\\eta=1&\text{ on }\partial \O.\end{cases}\end{equation}

The first result to be established is the following:
\begin{lemma}\label{NewLem2}
There exists $\overline\sigma_C>0$ such that, for any $\sigma \in (0;\overline \sigma_C)$, there exists a non-trivial radially symmetric solution $\eta_{C,1,\sigma}$ of \eqref{NewEqGaussian1}. This solution is radially symmetric and non-decreasing, and satisfies $0\leq \eta_{C,1,\sigma}\leq 1$.
\end{lemma}We prove it in the next paragraph. Throughout the rest of this section, such a  \textbf{$\overline\sigma_C>0$ is fixed.}

The link with Lemma \ref{NewLem1} is addressed in the following Lemma;
\begin{lemma}\label{NewLem3}
Lemma \ref{NewLem2} implies Lemma \ref{NewLem1}.\end{lemma}
\begin{proof}[Proof of Lemma \ref{NewLem3}]
Let $\sigma\in (0;\overline \sigma_C)$, and let $\eta_{C,1,\sigma}$ be the non-trivial radially symmetric solution given by Lemma \ref{NewLem2}, which we abbreviate as $\eta_1$. Since $\partial_r \eta_1\geq 0$ and $\eta_1$ is radially symmetric, it follows that 
$$-\left\langle \frac{\nabla N}N,\nabla \eta_1\right\rangle=-\frac{\partial_r N}{N}\partial_r \eta_1\geq C r\partial_r \eta_1=-\left\langle \frac{\nabla \mathcal N_C}{\mathcal N_C},\nabla \eta_1\right\rangle.$$ The last inequality is a consequence of Assumption \eqref{AssumptionDrift}. Hence,
$$-\Delta \eta_1-\frac2\sigma\left\langle \frac{\nabla N}N,\nabla \eta_1\right\rangle-f(\eta_1)\geq -\Delta \eta_1-\frac2{\sigma}\left\langle \frac{\nabla \mathcal N_C}{\mathcal N_C},\nabla \eta_1\right\rangle-f(\eta_1)=0.$$ In other words, $\eta_1$ is a super-solution of \eqref{NewEq1}. Since $z\equiv 0$ is always a sub-solution of \eqref{NewEq1}, the classical method of sub and super-solutions \cite[Theorem 5.17]{LeDret2018} ensures the existence of a non-trivial solution of \eqref{NewEq1}.
\end{proof}
We now prove Lemma \ref{NewLem2}.

\subsubsection{Proof of Lemma \ref{NewLem2}}\label{PrNL2}

We first simplify the proof by noticing the following claim:
\begin{claim}\label{ClIntermediaire}
Let $\sigma>0$ be arbitrary. Assume there exists a non-trivial radially symmetric solution $\eta_{C,1,\sigma}$ of \eqref{NewEqGaussian1} that is radially symmetric and non-decreasing, and satisfies $0\leq \eta_{C,1,\sigma}\leq 1$. Then, for any $\tilde \sigma\in (0;\sigma)$, there exists a solution $\eta_{C,1,\tilde\sigma}$ of \eqref{NewEqGaussian1} \end{claim}
\begin{proof}[Proof of Claim \ref{ClIntermediaire}]
We once again use the method of sub and super-solutions. Indeed, it suffices to notice that
$$-\frac1{\tilde \sigma}\left\langle \frac{\nabla \mathcal N_C}{\mathcal N_C},\nabla \eta_{C,1,\sigma}\right\rangle=\frac{Cr}{\tilde\sigma}\partial_r \eta_{C,1,\sigma} \geq -\frac1{ \sigma}\left\langle \frac{\nabla \mathcal N_C}{\mathcal N_C},\nabla \eta_{C,1,\sigma}\right\rangle. $$Hence
$$-\Delta\eta_{C,\sigma,1}-\frac2{\tilde\sigma}\left\langle \frac{\nabla N}N,\nabla\eta_{C,\sigma,1}\right\rangle-f(u_{C,\sigma,1})\geq -\Delta\eta_{C,\sigma,1}-\frac2{\sigma}\left\langle \frac{\nabla \mathcal N_C}{\mathcal N_C},\nabla\eta_{C,\sigma,1}\right\rangle-f(u_{C,\sigma,1})=0.$$ This hence gives us a super-solution for the equation with $\tilde \sigma$, and the conclusion follows in the same way as Lemma \ref{NewLem3}.
\end{proof}

\begin{proof}[Proof of Lemma \ref{NewLem2}] Given Claim \ref{ClIntermediaire}, it suffices to prove that a solution exists for at least one $\sigma>0$. To prove that this is the case, we use a phase plane analysis and a shooting method. Let us briefly outline the main steps:

For $\alpha \in (0;\theta)$, we consider the solution $p_{\alpha,\sigma}\color{black}=p_{\alpha,\sigma}(r)\color{black}$ of the differential equation\footnote{The existence and uniqueness of such equation is discussed in Claim \ref{Cl:ExisRate}}
\begin{equation}\label{ODE1}
\begin{cases}
-p_{\alpha,\sigma}''+\frac{2r}\sigma p_{\alpha,\sigma}'-\frac{d-1}r p_{\alpha,\sigma}'=f(p_{\alpha,\sigma}),
\\p_{\alpha,\sigma}(0)=\alpha\,, p_{\alpha,\sigma}'(0)=0.
\end{cases}
\end{equation}
We prove successively: 
\begin{enumerate}
\item \underline{Step 1 (Claim \ref{NCL1}):} For any $\alpha\in (0;\theta)$, there exists $r_{\alpha,\sigma,\theta}>0$ such that
$$p_{\alpha,\sigma}(r_{\alpha,\sigma,\theta})=\theta\,, \quad \alpha< p_{\alpha,\sigma}<\theta \text{ in }(0;r_{\alpha,\sigma,\theta})\,,  \quad p_{\alpha,\sigma}'>0 \text{ in }(0;r_{\alpha,\sigma,\theta}).$$
\item \underline{Step 2 (Claim \ref{NCL2}):}  There holds:
$$r_{\alpha,\sigma,\theta}\underset{\alpha\to 0^+}\rightarrow +\infty.$$ 
\item \underline{Step 3 (Claim \ref{NCL3}, Claim \ref{NCL4}):} For any $\sigma>0$, there exists $\alpha\in (0;\theta)$ such that
$$p_{\alpha,\sigma}'(r)\underset{r\to \infty}\rightarrow +\infty\,, p_{\alpha,\sigma}'(r)>0\text{ on }[r_{\alpha,\sigma,\theta};+\infty).$$ This will enable us to show that, when $\sigma$ is fixed, there exists $R(\sigma,1)$ such that there exists $\alpha \in (0;\theta)$ satisfying 
$$p_{\alpha,\sigma}(R(\sigma,1))=1\,, p_{\alpha,\sigma} \text{ is increasing in $(0;R(\sigma,1)).$}$$  Let $R_\sigma^*$ be the smallest value such that there exists such a non-trivial solution.  We will prove that, for any $R>R_{\sigma^*}$, there exists a non-trivial solution in $\mathbb B(0;R)$ with boundary value 1.
\item\underline{Step 4 (Claim \ref{NCL5}): } We prove that 
$$R_\sigma^*\underset{\sigma \to 0}\rightarrow 0,$$ hence concluding the proof by choosing $\sigma_C>0$ such that $R_{\sigma_C}<R$.

\end{enumerate}

\paragraph{Step 1} The goal of this paragraph is to prove the following Claim:
\begin{claim}\label{NCL1}
For any $\alpha\in (0;\theta)$, there exists $r_{\alpha,\sigma,\theta}>0$ such that
$$p_{\alpha,\sigma}(r_{\alpha,\sigma,\theta})=\theta\,, \quad \alpha< p_{\alpha,\sigma}<\theta \text{ in }(0;r_{\alpha,\sigma,\theta})\,,  \quad p_{\alpha,\sigma}'>0 \text{ in }(0;r_{\alpha,\sigma,\theta}).$$
\end{claim}
\begin{proof}[Proof of Claim \ref{NCL1}]
Since $p_{\alpha,\sigma}$ is continuous and since $p_{\alpha,\sigma}(0)=\alpha<\theta$, there exists $\delta>0$ such that 
$$p_{\alpha,\sigma}([0;\delta])\subset[0;\theta].$$
Let $r_{\alpha,\sigma,\theta}$ be defined as
$$r_{\alpha,\sigma,\theta}:=\sup\left\{\delta>0\,, p_{\alpha,\sigma}([0;\delta])\subset[0;\theta]\right\}>0.$$ We note that we do not yet rule out  the case $r_{\alpha,\sigma,\theta}=\infty.$ 

Let us first show that $p_{\alpha,\sigma}$ is increasing on $[0;r_{\alpha,\sigma,\theta})$. 


On $[0;r_{\alpha,\sigma,\theta})$, we have $f(p_{\alpha,\sigma}(r))<0$, so that
\begin{equation*}
 \begin{cases}
  p_{\alpha,\sigma}''\geq \left(\frac{2r}{\sigma}-\frac{d-1}{r}\right)p_{\alpha,\sigma}'\text{ in }(0;r_{\alpha,\sigma,\theta})\\
  p_{\alpha,\sigma}'(0)=0.
 \end{cases}
\end{equation*}
Integrating this inequality gives 
\begin{equation}\label{LK2}r\mapsto e^{-\frac{r^2}\sigma}r^{d-1}p_{\alpha,\sigma}'(r)\text{ is  non-decreasing}.\end{equation}

Furthermore, since $\alpha\in (0;\theta)$, $p_{\alpha,\sigma}$ is not constant in $(0;r_{\sigma,\alpha,\theta})$.This immediately gives
$$p_{\alpha,\sigma}'>0\text{ in }(0;r_{\alpha,\sigma,\theta}).$$
 This also proves that $r_{\alpha,\sigma,\theta}<+\infty$: we argue by contradiction. If $r_{\alpha,\sigma,\theta}=+\infty$ then \eqref{LK2} guarantees that $p_{\alpha,\sigma}'(r)\underset{r\to \infty}\rightarrow +\infty$, leading to an immediate contradiction. The same argument gives 
 $$p_{\alpha,\sigma}'(r_{\alpha,\sigma,\theta})>0.$$

\end{proof}
This Claim allows us to define 
$$r_{\alpha,\sigma,\theta}:=\inf\{r>0\,, p_{\alpha,\sigma}(r)=\theta\}\in (0;+\infty).$$

\paragraph{Step 2} The goal of this paragraph is the following Claim:
\begin{claim}\label{NCL2}
Let $\sigma>0$ be fixed. There holds
$$r_{\alpha,\sigma,\theta}\underset{\alpha\to 0^+}\rightarrow +\infty.$$ 

\end{claim}
\begin{proof}[Proof of Claim \ref{NCL2}]
The proof relies on the study of the function 
$$\xi(r):=\frac12(p_{\alpha,\sigma}^2+p_{\alpha,\sigma}'^2).$$  We introduce 
$$M:=\sup_{s\in (0;1)}\frac{-f(s)}s>0.$$ This quantity is finite due to the assumptions on $f$. Differentiating $\xi$  in $(0;r_{\alpha,\sigma,\theta})$ gives
\begin{align*}
\xi'(r)&=p_{\alpha,\sigma}'(r)\left(p_{\alpha,\sigma}(r)+p_{\alpha,\sigma}''(r)\right)\\
&=p_{\alpha,\sigma}'(r)\left(p_{\alpha,\sigma}(r)-f(p_{\alpha,\sigma})+2\frac{r}{\sigma}p_{\alpha,\sigma}'(r)-\frac{d-1}{r}p_{\alpha,\sigma}'(r)\right)\\
&\leq p_{\alpha,\sigma}'(r)\left(p_{\alpha,\sigma}(r)+ Mp_{\alpha,\sigma}(r)+2\frac{r}\sigma p_{\alpha,\sigma}'(r)\right)\text{ since $p_{\alpha,\sigma}'>0\text{ in }(0;r_{\alpha,\sigma,\theta})$}
\\&\leq p_{\alpha,\sigma}'(r)p_{\alpha,\sigma}(r)\left(M+1\right)+2\frac{r}\sigma p_{\alpha,\sigma}'(r)^2
\\&\leq \frac{M+1}2 (p_{\alpha,\sigma}'(r)^2+p_{\alpha,\sigma}(r)^2)+2\frac{r}{\sigma}( p_{\alpha,\sigma}'(r)^2+p_{\alpha,\sigma}(r)^2)
\\&\leq \xi(r)\left(M+1+4\frac{r}\sigma\right).
\end{align*}

Since $\xi(0)=\frac12\alpha^2$ we conclude from Gr\"{o}nwall's lemma that 
$$\xi(r)\leq \frac{\alpha^2}2e^{(M+1)r+2\frac{r^2}\sigma}.$$
Finally 
$$\xi(r_{\alpha,\sigma,\theta})\geq \frac12 p_{\alpha,\sigma}(r_{\alpha,\sigma,\theta})^2= \frac12\theta^2,$$ so that

$$\frac12\theta^2\leq  \frac{\alpha^2}2e^{(M+1)r_{\alpha,\sigma,\theta}+2\frac{r_{\alpha,\sigma,\theta}^2}\sigma}.$$
The conclusion follows.
\end{proof}
\begin{remark}\label{Remar}
If we define $r_{\alpha,\sigma,\frac\theta2}$ as the first root of $p_{\alpha,\sigma}(r_{\alpha,\sigma,\frac\theta2})=\frac\theta2$, the same proof shows that $$r_{\alpha,\sigma,\frac\theta2}\underset{\alpha \to 0}\rightarrow +\infty.$$
\end{remark}
\paragraph{Step 3} In this paragraph, we prove the two following claims:
\begin{claim}\label{NCL3}
 For any $\sigma>0$, there exists $\alpha\in (0;\theta)$ such that
$$p_{\alpha,\sigma}'(r)\underset{r\to \infty}\rightarrow +\infty\,, p_{\alpha,\sigma}'(r)>0\text{ on }[r_{\alpha,\sigma,\theta};+\infty).$$ \end{claim}

\begin{claim}\label{NCL4}
Let $\sigma>0$ be fixed. There exists $R(\sigma,1)$ such that there exists $\alpha \in (0;\theta)$ satisfying 
$$p_{\alpha,\sigma}(R(\sigma,1))=1\,, p_{\alpha,\sigma} \text{ is increasing in $(0;R(\sigma,1)).$}$$ As a consequence, a non-trivial solution of \eqref{NewEq1} exists in $\mathbb B(0;R(\sigma,1))$, and this solution is radially symmetric and non-decreasing. Furthermore, for any $R\geq R(\sigma,1)$, a non-trivial solution of \eqref{NewEq1} exists in $\mathbb B(0;R)$.\end{claim}

\begin{proof}[Proof of Claim \ref{NCL3}]
To prove this Claim, the first essential step is to prove that, when $\sigma>0$ is fixed, there exists $\underline m>0$ such that, for any $\alpha>0$ small enough
$$p_{\alpha,\sigma}'(r_{\sigma,\alpha,\theta})\geq \underline m.$$ This is done through energy arguments.

We first observe  that we can choose $\alpha>0$ small enough so that the energy
$$ \mathscr E_{\alpha,\sigma}:\color{black}\R_+\ni r\mapsto \frac12(p_{\alpha,\sigma}'(r))^2+F(p_{\alpha,\sigma}(r))$$ is increasing on $(r_{\sigma,\alpha,\frac\theta2};+\infty)$. Here, we recall that $r_{\sigma,\alpha,\frac\theta2}$ was defined in Remark \ref{Remar} as the first solution of $p(r_{\sigma,\alpha,\frac\theta2})=\frac\theta2$ in $(0;r_{\alpha,\sigma,\theta})$.

Indeed, this energy satisfies
$$\frac{d\mathscr E_{\alpha,\sigma}}{dr}=
\left(\frac{2r}{\sigma}-\frac{d-1}{r}\right)p'_{\sigma,\alpha}(r)^2.$$ This last term is positive whenever 
$$r\geq \sqrt{\frac{\sigma(d-1)}2}.$$ As a consequence, it is sufficient, to obtain the monotonicity of the energy on $(r_{\alpha,\sigma,\theta};+\infty)$, to ensure that 
$$r_{\alpha,\sigma,\theta}>\sqrt{\frac{\sigma(d-1)}2}.$$ Claim \ref{NCL2} guarantees that this is possible provided $\alpha>0$ is small enough. We will even require something stronger than $r_{\alpha,\sigma,\theta}>\sqrt{\frac{\sigma(d-1)}2}$, that is, we fix (thanks to Remark \ref{Remar}) $\alpha>0$ small enough so that 
$$r_{\alpha,\sigma,\frac\theta2}>\sqrt{\frac{\sigma(d-1)}2}.$$

$p_{\alpha,\sigma}$ satisfies
$$p_{\alpha,\sigma}''(r)=-f(p_{\alpha,\sigma})+p_{\alpha,\sigma}'\left(\frac{2r}\sigma-\frac{d-1}r\right).$$ The previous computation then shows that $\mathscr E_{\alpha,\sigma}$ is increasing on $(r_{\sigma,\alpha,\frac\theta2};+\infty)$, whence it follows that 
$$\mathscr E_{\alpha,\sigma}(r_{\alpha,\sigma,\theta})>\mathscr E_{\alpha,\sigma}(r_{\sigma,\alpha,\frac\theta2})\geq F\left(\frac\theta2\right).$$ In particular, we obtain
$$p_{\alpha,\sigma}'(r_{\alpha,\sigma,\theta})>\sqrt{2\left(F\left(\frac{\theta}2\right)-F\left(\theta\right)\right)}=:\underline m.$$
We quickly remark that, since $f$ is negative on $(0;\theta)$, $F(\theta)<F\left(\frac\theta2\right)$, so that the right-hand side of the previous inequality is indeed positive.

The key part is that this lower estimate on $p_{\alpha,\sigma}'(r_{\alpha,\sigma,\theta})$ is uniform in $\alpha$.

We now turn back to the equation on $p_{\color{black}\alpha,\sigma}$:
$$p_{\color{black}\alpha,\sigma}''(r)=-f(p_{\color{black}\alpha,\sigma}(\color{black}r\color{black}))+p_{\color{black}\alpha,\sigma}'(r)\left(\frac{2r}\sigma-\frac{d-1}r\right)=:g(r).$$
We will obtain that $p_{\alpha,\sigma}'$ is increasing and goes to $+\infty$ by studying the growth of $g$. First, notice that 
$$g(r_{\alpha,\sigma,\theta})\geq \left(2\frac{r_{\alpha,\sigma,\theta}}\sigma-\frac{d-1}{r_{\alpha,\sigma,\theta}}\right)\underline m>0$$ because of the uniform lower bound on $p_{\alpha,\sigma}'(r_{\alpha,\sigma,\theta})\geq \underline m$ and because $\alpha$ was chosen small enough to ensure $r_{\alpha,\sigma,\theta}>\sqrt{\frac{\sigma(d-1)}2}.$ 

As a consequence, $p_{\alpha,\sigma}'$ is locally increasing around $r_{\alpha,\sigma,\theta}$, which allows us to define 
$$A_1:=\sup\{A\in \R_+^*\,, p_{\alpha,\sigma}'\geq p_{\alpha,\sigma}'(r_{\alpha,\sigma,\theta})\text{ in }[r_{\alpha,\sigma,\theta};r_{\alpha,\sigma,\theta}+A]\}>0.$$ We are going to prove that $A_1=+\infty$. Let us first compute $g'(r)$:
\begin{align*}
g'(r)&=-f'(p_{\alpha,\sigma})p_{\alpha,\sigma}'+\left(\frac{2r}\sigma -\frac{d-1}{r}\right)p_{\alpha,\sigma}''+p_{\alpha,\sigma}'\left(\frac{2}\sigma +\frac{d-1}{r^2}\right)\\
&=p_{\alpha,\sigma}'\left( -f'(p_{\alpha,\sigma})+\frac2{\sigma}+\left(\frac{2r}\sigma-\frac{d-1}{r}\right)\left(\frac{2r}\sigma -\frac{f(p_{\alpha,\sigma})}{p_{\alpha,\sigma}'}-\frac{d-1}{r}\right)+\frac{d-1}{r^2}\right)\\
&=p_{\alpha,\sigma}'\left( -f'(p_{\alpha,\sigma})+\frac2{\sigma}+\left(\frac{2r}\sigma-\frac{d-1}{r}\right)^2-\left(\frac{2r}\sigma-\frac{d-1}{r}\right)\frac{f(p_{\alpha,\sigma})}{p_{\alpha,\sigma}'}+\frac{d-1}{r^2}\right)\\
&=p_{\alpha,\sigma}' G(r,p_{\alpha,\sigma},p_{\alpha,\sigma}')
\end{align*}

with
$$G(r,p,v):=\left( -f'(p)+\frac2{\sigma}+\left(\frac{2r}\sigma-\frac{d-1}{r}\right)^2-\left(\frac{2r}\sigma-\frac{d-1}{r}\right)\frac{f(p)}{v}+\frac{d-1}{r^2}\right).$$
If we can guarantee that 
\begin{equation}\label{Eq:Srs}\forall v\geq \underline m\,, \forall r\geq r_{\alpha,\sigma,\theta}\,, G(r,p,v)\geq 0,\end{equation}
then we are done by considering the system 
$$p_{\alpha,\sigma}''=g\,, g'=p_{\alpha,\sigma}'G,$$ and we will have established that $A_1=+\infty$. Let us now prove that \eqref{Eq:Srs} holds for $\alpha>0$ small enough: extending if needed $f$ into a $W^{1,\infty}$ function outside of $[0,1]$, we see that this condition is guaranteed if, for any $r\geq r_{\alpha,\sigma,\theta}$ we have

\begin{equation}
||f'||_{L^\infty}+\left(\frac{2r}\sigma-\frac{d-1}{r}\right)\frac{||f||_{L^\infty}}{\underline m}\leq \frac2\sigma+\left(\frac{2r}{\sigma}-\frac{d-1}{r}\right)^2+\frac{d-1}{r^2}.
\end{equation}
However, this inequality always holds for any $r\geq r_{\alpha,\sigma,\theta}$, provided $r_{\alpha,\sigma,\theta}$ is large enough, which is in turn guaranteed provided $\alpha>0$ is small enough. With such an $\alpha$ fixed, we have $A_1=+\infty$, and so we have 
$$\forall r\geq r_{\alpha,\sigma,\theta}\,, p_{\alpha,\sigma}'(r)\geq p_{\alpha,\sigma}'(r_{\alpha,\sigma,\theta})\geq \underline m>0.$$ As a byproduct, we get 
$$\forall r\geq r_{\alpha,\sigma,\theta}\,, g(r)\geq g(r_{\alpha,\sigma,\theta})>0,$$ so that, integrating the inequality
$$(p_{\alpha,\sigma}')'(r)=g(r)\geq g(r_{\alpha,\sigma,\theta})>0$$ we obtain 
$$p_{\alpha,\sigma}'(r)\underset{r\to \infty}\rightarrow +\infty.$$

\end{proof}

\begin{proof}[Proof of Claim \ref{NCL4}]
The existence of such an $R(\sigma,1)$ is an immediate consequence of Claim \ref{NCL3}.: indeed, choosing $\alpha>0$ small enough so that the conclusions of Claim \ref{NCL3} are satisfied and keeping in mind that $p_{\alpha,\sigma}$ is increasing on $[0;r_{\alpha,\sigma,\theta}]$, it suffices to define $R(\sigma,1,\alpha)$ as the first solution of 
$$p_{\alpha,\sigma}(R(\sigma,1,\alpha))=1$$ to obtain the desired conclusion. Let us now fix such an $\overline \alpha>0$ and define $R(\sigma,1):=R(\sigma,1,\overline \alpha)$.

To obtain the same conclusion for any $R>R(\sigma,1)$, it suffices to observe that, first, if $0<\alpha<\overline \alpha$, the solution $p_{\alpha,\sigma}$ satisfies the conclusion of Claim \ref{NCL3} and that $p_{\alpha,\sigma}<p_{\overline \alpha,\sigma}$ by a standard comparison argument, so that $\alpha \mapsto R(\sigma,1,\alpha)$ is non-increasing, and, second, that
 $R({\sigma,1,\alpha})\underset{\alpha\to 0}\rightarrow +\infty$. This behaviour as $\alpha \to 0^+$ is a simple consequence of the fact that $$R(\sigma,1,\alpha)>r_{\alpha,\sigma,\theta}\underset{\alpha \to 0^+}\rightarrow +\infty.$$ 
\end{proof}

We now define 
\begin{multline}
R_\sigma^*:=\inf\left\{R_1>0\,, \forall R'\geq R_1, \text{ there exists a non-trivial radially}\right.\\\left.\text{ symmetric non-decreasing solution of \eqref{NewEq1} in $\mathbb B(0;R')$}\right\}.\end{multline}

\paragraph{Step 4} In this final step, we prove the following Claim:
\begin{claim}\label{NCL5}
$$R_\sigma^*\underset{\sigma \to 0}\rightarrow 0.$$
\end{claim}
\begin{proof}[Proof of Claim \ref{NCL5}]

\def\r{{r_{\alpha,\sigma,\theta}}}
We argue by contradiction. Assume that there exists a sequence $\{\sigma_k\}_{k\in \N}$ such that 
$$R_{\sigma_k}^*\underset{k\to \infty}{\not \rightarrow}0\,, \sigma_k\underset{k\to \infty}\rightarrow 0.$$ With a slight abuse of notation, we assume that 
$$\underline R:=\underline \lim_{k\to\infty} R_{\sigma_k}^*>0.$$
Let $\alpha>0$ be fixed. From Claim \ref{NCL1} we know that, for every $\sigma>0$, there exists $\r>0$ such that 
$$p_{\alpha,\sigma}(\r)=\theta\,, p_{\alpha, \sigma} \text{ is increasing on } [0;\r].$$ Let 
$$p_k:=p_{\alpha,\sigma_k}\,,\quad r_k:=r_{\alpha,\sigma_k,\theta}.$$
We reach a contradiction by distinguishing two cases:
\begin{enumerate}
\item\underline{0 is an accumulation point of $\{r_k\}$:} Assume that, up to a subsequence, we have
$$r_k\underset{k\to \infty}\rightarrow 0.$$
The Mean Value Theorem ensures, for any $k\in \N$, the existence of $y_k\in (0;r_k)$ such that 
$$p_k'(y_k)=\frac{\theta-\alpha}{r_k}\underset{k\to \infty}\rightarrow +\infty.$$ 

We first note that we can obtain a crude estimate on $y_k$, namely, that, for $k$ large enough, we have 
\begin{equation}\label{W}y_k\geq \sqrt{\frac{\sigma_k(d-1)}2}=:r_k^*.\end{equation}
To get this estimate, we note that, on $\left(0;r_k^*\right)$, we have 
$$\left(\frac{2r}\sigma-\frac{d-1}r\right)<0,$$ and so 
$$p_k''\leq f(p_k).$$ It thus follows that 
$$\forall r \in [0;r_k^*]\,, p_k'(r)\leq r \Vert f\Vert_{L^\infty}\leq r_k^* \Vert f\Vert_{L^\infty}.$$ Since 
$$p_k'(y_k)\underset{k\to \infty}\rightarrow +\infty,$$ it follows that, for $k$ large enough, $y_k\geq r_k^*$. We claim that this implies that $p_k'\to+\infty$ uniformly in $\left[y_k;\frac{\underline R}2\right]$  as made precise in the following statement:
\begin{equation}\label{Q} \forall M\in \R_+^*\,, \exists k_M\in \N\,, \forall k\geq k_M\,, p_k'\geq M\text{ in }\left[y_k;\frac{\underline R}2\right].\end{equation}

To prove \eqref{Q}, we first note that $p_k'(y_k)\underset{k\to \infty}\rightarrow +\infty$ implies
$$p_{k}'(r_k)\underset{k\to \infty}\rightarrow +\infty.$$ Indeed, this follows from the fact that, since $\left(\frac{2r}{\sigma_k}-\frac{d-1}{r}\right) >0$ in $(y_k;+\infty)$ (because $y_k>r_k^*$), since $p_k'\geq 0$ in $(0;r_k)$ (because of Claim \ref{NCL1}) and since $f(p_k)\leq 0$ (because $p_k\leq \theta$ in $(y_k;r_k)$), we have 
$$p_k''=\left(\frac{2r}{\sigma_k}-\frac{d-1}{r}\right) p_k'-f(p_k)\geq 0 \text{ in }(y_k;r_k)$$ and so 
$$p_k'(r_k)\geq p_k'(y_k).$$ 

To prove that this implies \eqref{Q}, we use a comparison principle on $\left(r_k;\frac{\underline R}2\right)$: define $q_k$ as the solution of 
$$\begin{cases}q_k'=-f(p_k)\text{ in }\left(r_k;\frac{\underline R}2\right),\\ q_k(r_k)=p_k'(r_k).\end{cases}$$ A crude bound on $q_k$ is 
$$\forall t \in  \left(r_k;\frac{\underline R}2\right)\,, q_k(t)\geq q_k(r_k)-(t-r_k)\Vert f\Vert_{L^\infty}.$$ Since $r_k\underset{k\to \infty}\rightarrow 0$ and since $q_k(r_k)\underset{k\to \infty}\rightarrow +\infty$, $q_k$ diverges to $\infty$ uniformly on $(r_k;\frac{\underline{R}}2)$. A simple consequence is that $q_k>0$ for $k$ large enough.

We now define $z_k:=q_k-p_k'.$ We immediately obtain that 
$$z_k'-\left(\frac{2r}{\sigma_k}-\frac{d-1}{r}\right) z_k=-\left(\frac{2r}{\sigma_k}-\frac{d-1}{r}\right) q_k.$$ Then again, since 
$r_k\geq y_k$ it follows that $\left(\frac{2r}{\sigma_k}-\frac{d-1}{r}\right) >0$ in $\left(r_k;\frac{\underline R}2\right)$, hence 
$$z_k'-\left(\frac{2r}{\sigma_k}-\frac{d-1}{r}\right) z_k<0 \text{ in } \left(r_k;\frac{\underline R}2\right).$$
Integrating this differential inequality yields that
$$r\mapsto e^{-\frac{r^2}\sigma}r^{d-1}z_k\text{ is non-increasing on }\left[r_k;\frac{\underline R}2\right],$$ hence, for any $r\in \left[r_k;\frac{\underline R}2\right]$,
$$z_k(r)\leq e^{\frac{r^2}\sigma}r^{1-d} e^{-\frac{r_k^2}\sigma}r_k^{d-1}z_k(r_k)=0\text{ because }z_k(r_k)=0.$$
As a consequence $z_k<0$.

It follows that $$p_k'\geq q_k \text{ in } \left[r_k;\frac{\underline R}2\right]$$ and thus converges uniformly to $+\infty$ in that interval. As a consequence, the equation $p_k(x)=1$ has a unique root $x_k\in\left[r_k;\frac{\underline R}2\right]$ for any $k$ large enough, and is increasing in $(0;x_k)$, which is in contradiction with the definition of $\underline R$.

\item\underline{0 is not an accumulation point of $\{r_k\}$:} 

Assuming $0$ is not an accumulation point of $\{r_k\}_{k\in \N}$, a contradiction ensues in the following manner: we know that there thus exists a point $y>0$ such that
$$y\leq \underset{k\to \infty}{\underline \lim} r_k\,, \lim_{k\to \infty} p_k(y)\leq \theta-\delta$$ for some $\delta>0$. Then we note that, by explicit  integration of$$-p_k''+\left(\frac{2r}{\sigma_k}-\frac{d-1}{r}\right)p_k'=f(p_k)$$ 
we get 
\begin{equation}\label{Eq:Integ2}
r^{d-1}p_k'(r)=e^{\frac{r^2}\sigma}\int_0^re^{\frac{-t^2}\sigma} (-f(p_k(t)))t^{d-1}dt.
\end{equation}
Since $y\in (0;r_k)$ for every $k$ large enough and since $p_k$ is increasing in $(0;r_k)$, we have $\alpha\leq p_k\leq \theta-\delta$ for every $t\in [0;y]$, so that 
$$\exists \delta'>0\,, f(p_k)\leq -\delta'\text{ on }[ 0;y].$$
Plugging this in the integral formulation \eqref{Eq:Integ2} gives the lower bound
$$r^{d-1}p_k'(r)\geq \delta' e^{\frac{r^2}\sigma}\int_0^re^{-\frac{t^2}\sigma}t^{d-1}dt,\quad r\in[0,y]$$ Let us now study the interval $\left[\frac{y}2;y\right]$ and prove that $p_k'$ converges uniformly to $+\infty$ in $\left[\frac{y}2;y\right]$, which would immediately yield the desired contradiction.

We note that, for any $r\in \left[\frac{y}2;y\right]$, we have

$$\int_0^re^{-\frac{t^2}\sigma}t^{d-1}dt\geq\int_0^{y/2}e^{-\frac{t^2}\sigma}t^{d-1}dt\underset{\sigma\to 0}\sim C\sqrt{\sigma^d},$$ for some $C>0$ by the Laplace method (recalled in detailed below \eqref{Laplace}), which immediately gives 
$$p_k'(r)\underset{k\to \infty}\rightarrow +\infty \text{ uniformly in }\left[\frac{y}{2};y\right].$$ The conclusion follows.

\end{enumerate}
\end{proof}
\paragraph{Proof of Lemma \ref{NewLem2}}
Since $R>0$, there exists $\underline \sigma>0$ such that, for any $\sigma\leq \underline \sigma$, $R_\sigma^*<R$. As a consequence of the definition of $R_\sigma^*$, a non-trivial solution of \eqref{NewEq1} exists in $\mathbb B(0;R)$ for any $\sigma\leq \underline \sigma.$ 

\end{proof}

\subsection{Proof of Theorem \ref{NewTheo}: blocking phenomenon towards 0 }
The  relevant equation is, in this case,
\begin{equation}\label{NewEq11}
\begin{cases}
-\Delta p-\frac2\sigma \langle \nabla p,\frac{\nabla N}N\rangle=f(p)&\text{ in }\Omega;
\\p=0&\text{ on }\partial \O.\end{cases}\end{equation}
A non-trivial solution $p$ to this equation is called admissible if $0\leq p\leq 1$.

\begin{lemma}\label{NewLem11}
Assume $N$ satisfies \eqref{AssumptionDrift2}. There exists $\sigma_{N,0}>0$ such that, for any $\sigma \in (0;\sigma_{N,0})$, there exists a non-trivial admissible solution of \eqref{NewEq11}.
\end{lemma}
Point (2) of the Theorem is an immediate consequence of this Lemma.

 We will  use the Laplace method to prove that, for any $R>0$, there exists $\sigma_N>0$ such that, for any $\sigma\in (0;\sigma_N)$, the equation 
\begin{equation}\label{LK}
\begin{cases}
-\Delta p-\frac2\sigma \left\langle \frac{\n N}N,\n p\right\rangle=f(p)\text{ in }\O\,, \\ p=0\text{ on }\partial \O\end{cases}\end{equation} has a non trivial solution. In order to do so, we use the classical method of \cite{LionsBerestycki}. First of all, let us note that \eqref{LK} admits a variational formulation. Indeed, multiplying \eqref{LK} by $N^{\frac2\sigma}$ we obtain that any solution $p$ of \eqref{LK} satisfies
$$-N^{\frac2\sigma}\Delta p-\frac2\sigma N^{\frac2\sigma-1}\langle\n N,\n p\rangle=N^{\frac2\sigma}f(p)$$ that is, in other words, 
$$-\n \cdot\left(N^{\frac2\sigma}\n p\right)=N^{\frac2\sigma}f(p).$$ This leads to introducing the natural energy functional
$$\mathcal E_{N,\sigma}:W^{1,2}_0(\O)\color{black} \ni p\color{black}\mapsto \frac12\int_\O N^{\frac2\sigma}|\n p|^2-\int_\O N^{\frac2\sigma}F(p).$$
\color{black}
However, depending on the non-linearity $f$, this functional may not be coercive or admit minimisers. To overcome this difficulty,  we first follow the strategy of \cite{LionsBerestycki} and assume that $f$ was extended by 0 outside of $[0;1]$ (as we are looking for solutions between 0 and 1, if we can prove that, for such an extension, we have a minimiser between 0 and 1, then the only thing that matters is the definition of $f$ on $[0;1]$, so that the chosen extension is unimportant). We can now prove that $\mathcal E_{N,\sigma}$ admits a minimiser $p^*$ in $W^{1,2}_0(\O)$ that further satisfies $0\leq p^*\leq 1$: consider a minimising sequence $\{p_k\}_{k\in \N}\in W^{1,2}_0(\O)$. Define, for any $k\in \N$, 
\[ \tilde p_k:=p_k\left(\mathds 1_{\{p_k\geq 0\}}+\mathds 1_{\{p_k\leq 1\}}\right)\in W^{1,2}_0(\O).\]
As for any $x\leq 0$ we have $F(x)=0$ and, for any $x\geq 1$ $F(x)=F(1)$ we obtain 
\[\forall k\in \N\,, \int_\O N^{\frac2\sigma}F(p_k)=\int_\O N^{\frac2\sigma}F(\tilde p_k).\] Similarly,
\[\forall k\in \N\,, \int_\O N^{\frac2\sigma}|\n p_k|^2\geq\int_\O N^{\frac2\sigma}|\n \tilde p_k|^2,\]
so that, for any $k\in \N$, $\mathcal E_{N,\sigma}(p_k)\geq \mathcal E_{N,\sigma}(\tilde p_k).$ Hence, $\{\tilde p_k\}_{k\in \N}$ is a minimising sequence, which is moreover bounded in $L^\infty(\O)$. Consequently, it follows from the definition of $\mathcal E_{N,\sigma}$ that $\{\tilde p_k\}_{k\in \N}$ is bounded in $W^{1,2}_0(\O)$, and hence, up to a subsequence, converges (strongly in $L^2(\O)$, weakly in $W^{1,2}_0(\O)$) to a minimiser $p^*$ that further satisfies $0\leq p^*\leq 1$ almost everywhere.

The fact that \color{black}the \color{black} minimum is non-zero is a consequence of the next Lemma:

\begin{lemma}\label{Le:Energie}
Assume $N$ satisfies \eqref{AssumptionDrift2}. There exists $\sigma_{N,0}>0$ such that, for any $\sigma\in (0;\sigma_{N,0})$, there holds
$$\min_{p\in W^{1,2}_0(\O)}\mathcal E_{N,\sigma}(p)<0.$$
\end{lemma}

Then, since $\mathcal E_{N,\sigma}$ admits a minimum \color{black}at a minimiser which is admissible (i.e. between 0 and 1 almost everywhere) \color{black}, and since Lemma \ref{Le:Energie} ensures this minimum is not identically 0, the existence of a non-trivial solution follows.

 Lemma \ref{Le:Energie} relies on the Laplace method and, more precisely, on the Watson's Lemma; this method is presented in \cite{Laplace,Advanced}.  We briefly recall the following conclusion of this method (see for instance \cite[Theorem 1]{Laplace}): let $r_1>0$. If $\phi:[0;R]\to \color{black}\R\color{black}$ is a $\mathscr C^1$ function such that $\phi(0)\neq 0$, if $\alpha>0$ is a positive parameter, then
\begin{equation}\label{Laplace}\int_0^{r_1} t^{\alpha-1}\phi(t)e^{-c_0\frac{t^2}{\epsilon}}dt\underset{\epsilon \to 0^+}{\sim}M(c_0,d) {\phi(0)}{\epsilon^{\frac{\alpha}2}},\end{equation} where $M(c_0,d)$ is a constant that only depends on $c_0$ and $\alpha$.

\begin{proof}[Proof of Lemma \ref{Le:Energie}]
We fix $c_0,c_1>0$ as given by Assumption \eqref{AssumptionDrift2}.

We construct a function $\eta>0$ such that, whenever $\sigma$ is small enough, 
$$\mathcal E_{N,\sigma}(\eta)<0.$$ To do so, we define $\eta$ as follows: let $\delta\in \left(0;\frac{R}2\right)$. Let $\eta\equiv 1$ in $\mathbb B(0;\delta)$, $\eta\equiv 0 \in \mathbb B(0;R)\backslash \mathbb B(0;2\delta)$. We extend this function to a radially symmetric non-increasing function function $\eta \in \mathscr C^1(\mathbb B(0;R)).$

Let us split the energy $\mathcal E_{N,\sigma}$ in two parts:
\begin{enumerate}
\item The first part corresponds to the gradient: we note that
\begin{align*}
0\leq \int_\O N^{\frac2\sigma}|\n \eta|^2&\leq \int_\O e^{-c_1\frac{||x||^2}\sigma}|\n \eta(x)|^2\color{black}dx\color{black}
\\&=\int_{\mathbb B(0;R)\backslash \mathbb B(0;\delta)}e^{-c_1\frac{||x||^2}\sigma}|\n \eta(x)|^2\color{black}dx\color{black}&\text{ because }\n \eta\equiv 0 \text{ in }\mathbb B(0;\delta)
\\&\leq  |\mathbb B(0;R)| e^{-c_1\frac{\delta^2}\sigma}\Vert \n \eta\Vert_{L^\infty}=M_{\bold I} e^{-c_1\frac{\delta^2}\sigma}.
\end{align*}
Here, $M_{\bold{I}}>0$.
\item The second part is trickier. Let us consider 
$$\int_\O F(\eta (x))N^{\frac2\sigma}(x)\color{black}dx\color{black}.$$
Since $\eta$ is radially non-increasing and since $F(\eta(0))=F(1)>0$, let $r_1>0$ be the first real number such that 
\[\color{black}\forall x\in \O\,, \Vert x\Vert=r_1\Rightarrow F(\eta(x))=0.\color{black}\] 
We then have 
\begin{align*}
\int_\O F(\eta\color{black}(x)\color{black})N^{\frac2\sigma}\color{black}(x)dx\color{black}&=\int_{\mathbb B(0;r_1)} F(\eta\color{black}(x)\color{black})N^{\frac2\sigma}\color{black}(x)dx\color{black}+\int_{\mathbb B(0;R)\backslash \mathbb B(0;r_1)} F(\eta\color{black}(x)\color{black})N^{\frac2\sigma}\color{black}(x)dx\color{black}.
\end{align*}
We note that 
\begin{equation}\label{Decay}\left|\int_{\mathbb B(0;R)\backslash \mathbb B(0;r_1)} F(\eta\color{black}(x)\color{black})N^{\frac2\sigma}\color{black}(x)dx\color{black}\right|\leq \Vert F(\eta)\Vert_{L^\infty}M'e^{-c_1\frac{r_1^2}\sigma}\end{equation} for some constant $M'$ by the same arguments that gave us $M_{\bold{I}}$, so that this part decays exponentially as $\sigma \to 0^+$. For the first part, since $F(\eta)\geq 0$ in $\mathbb B(0;r_1)$ by definition of $r_1$, we have

$$\int_{\mathbb B(0;r_1)} F(\eta\color{black}(x)\color{black})N^{\frac2\sigma}\color{black}(x)dx\color{black}\geq  \int_{\mathbb B(0;r_1)}F(\eta\color{black}(x)\color{black})e^{-c_0\frac{\Vert x\Vert^2}\sigma}\color{black}dx\color{black}.$$ Since all the functions involved are now radially symmetric,
passing to polar coordinates gives
$$ \int_{\mathbb B(0;r_1)}F(\eta\color{black}(x)\color{black})e^{-c_0\frac{\Vert x\Vert^2}\sigma}\color{black}dx\color{black}=S_d \int_0^{r_1} F(\eta(r))r^{d-1}e^{-c_0\frac{ r^2}\sigma}\color{black}dr\color{black},$$ where, with a slight abuse of notation and since $\eta$ is radially symmetric, we keep the notation $F(\eta)$ for its one-dimensional counterpart. In the formula above, $S_d$ only depends on the dimension. From the Laplace method, it follows that
$$ \int_{\mathbb B(0;r_1)}F(\eta\color{black}(x)\color{black})e^{-c_0\frac{\Vert x\Vert^2}\sigma}\color{black}dx\color{black}\underset{\sigma \to 0^+}\sim M''{F(1)}\sigma^{\frac{d}2}$$ for some constant $M''>0$. Combining this with \eqref{Decay}, we obtain, for some constant $M_{\bold{II}}>0$
$$\int_\O F(\eta\color{black}(x)\color{black})N^{\frac2\sigma}\color{black}(x)dx\color{black}\geq M_{\bold{II}}F(1) \sigma^{\frac{d}2}.$$
\end{enumerate}
Combining these two steps, we obtain
\begin{align*}
\mathcal E_{N,\sigma}(\eta)&\leq M_{\bold I} e^{-c_1\frac{\delta^2}\sigma}-M_{\bold{II}}F(1) \sigma^{\frac{d}2}.
\\&<0
\end{align*}whenever $\sigma$ is small enough. The conclusion follows.
\end{proof}
As a consequence, whenever $N$ satisfies Assumption \eqref{AssumptionDrift2}, a non-trivial solution to \eqref{LK} exists.

This concludes the proof of the Theorem.\color{black}

\color{black}
\section{Proof of Theorem \ref{Th:Fin}: unblocking phenomenon}\label{Proof3}
\begin{proof}[Proof of Theorem \ref{Th:Fin}]
The key point is that, when $\sigma>0$ is small enough, we have uniqueness of solutions to 
\begin{equation}\label{Ok}\begin{cases}
-\Delta p-\frac2\sigma \langle\frac{ \n N}N,\n p\rangle=f(p)&\text{ in }\O=\mathbb B(0;R)\,, 
\\ p\equiv a &\text{ on }\partial \O,\end{cases}\end{equation} where $a=0,1$ or $\theta$. Indeed, should this uniqueness hold, a static control $u\equiv a$ will drive any initial condition to $z_a$ (in finite time for $\theta$, and in infinite time for $0$ and $1$).
We first note that the main equation of \eqref{Ok} is equivalent to 
$$-\nabla \cdot (N^{\frac2\sigma}\n p)=N^{\frac2\sigma}f(p).$$

However, defining $$M:=\sup_{x,y}\left|\frac{f(x)-f(y)}{x-y}\right|$$  this new form allows us to state that uniqueness for \eqref{Ok} holds provided 
$$\lambda_\sigma(\O,N):=\inf_{\psi \in W^{1,2}_0(\O)\,, \psi\neq 0}\frac{\int_\O N^{\frac2\sigma}|\n \psi|^2}{\int_\O N^{\frac2\sigma}\psi^2}>M.$$ This is readily seen by taking the difference of two different solutions of \eqref{Ok}. We are now going to prove that, with $N(x)=e^{||x||^2}$, we have 
\begin{equation}\label{Okk}\lambda_\sigma(\O,N)\underset{\sigma \to 0}\rightarrow +\infty.\end{equation}
\eqref{Okk} follows from an elementary observation: we obviously have 
$$\lambda_\sigma(\O,N)\geq \lambda_\sigma(\R^d,N)= \inf_{\psi \in W^{1,2}_0(\R^d)\,, \psi\neq 0}\frac{\int_{\R^d} N^{\frac2\sigma}|\n \psi|^2}{\int_{\R^d}  N^{\frac2\sigma}\psi^2}.$$ By a simple change of variables (since $N:x\mapsto e^{\frac{\Vert x\Vert^2}2}$), we have $$\lambda_\sigma(\R^d,N)=\frac1{\sigma}\lambda_1(\R^d,N)=\frac1\sigma \inf_{\psi \in W^{1,2}_0(\R^d)\,, \psi\neq 0}\frac{\int_{\R^d} N^2|\n \psi|^2}{\int_{\R^d}  N^2\psi^2}$$ and, from \cite[Corollary 1.10]{EscobedoKavian}, 
$$\lambda_1(\R^d,N)>0.$$ \eqref{Okk} follows immediately.

\end{proof}

\color{black}

\section{Proof of Theorem \ref{Th:Rate}: radial drifts}\label{Proof4}

\begin{proof}[Proof of Theorem \ref{Th:Rate}]
Proceeding along the same lines as in Theorem \ref{Th:SHSlow}, we prove that for any drift $N\in \mathscr C^\infty(\O;\R)$ (regardless of whether or not it is the restriction of a radial drift $N$ to the domain $\O$), if condition \eqref{Eq:ConSpectral} holds,
then $z_0\equiv 0$ is the only solution to
\begin{equation}\label{Eq:Diri}
\left\{\begin{array}{ll}
-\Delta p-2\langle \frac{\n N}{N},\n p\rangle=f(p)&\text{ in }\O\,,
\\p=0&\text{ on }\partial \O,
\\0\leq p\leq1,&
\end{array}\right.\end{equation} and note that the main equation is equivalent to 
$$-\n\cdot\left(N^2\n p\right)=f(p)N^2.$$
Indeed, assuming there exists a non-trivial solution $p$ to \eqref{Eq:Diri} then from the mean value theorem, we can write 
$$f(p)=f'(y)p$$ for some function $y$ and, multiplying the equation by $p$ and integrating by parts gives, using the Rayleigh quotient formulation of {\color{black}$\lambda_1^D(\O,N)$}:

\textcolor{black}{
\begin{equation*}
 \lambda_1^D(\O,N)\int_\O N^2p^2\leq \int N^2|\n p|^2=\int_\O N^2 f'(y)p^2\leq \|f'\|_{L^\infty}\int_\O N^2p^2
\end{equation*}
}
which is \color{black} a \color{black} contradiction unless \color{black}$p\equiv 0$\color{black}.

Once we have uniqueness for \eqref{Eq:Diri} we follow, for any initial datum $p_0$, the staircase method explained in the proof of Theorem \ref{Th:SHSlow}: we first set the static control $u=0$, we drive the solution to a $\mathscr C^0$ neighbourhood of $z_0$, then to a steady-state solution of \eqref{Eq:AdvectionControl} in this neighborhood. Thus, we only need to prove the existence of a path of steady states linking $z_0$ to $z_\theta$. In order to prove that such a path of steady states exists under assumption \eqref{Eq:Rate}, we use an energy method.

Let $R>0$ be such that $\O\subset \mathbb B(0;R)$. As in \cite{BRZ}, we define, for any $s\in[0,1]$, $p_s$ as the unique solution of 
\begin{equation}\label{Eq:ConsRate}
\left\{\begin{array}{ll}
-\Delta p_s-2\langle \frac{\n N}{N},\n p_s\rangle=f(p_s)\,,&\text{ in }\mathbb B(0;R)
\\p_s\text{ is radial in }\mathbb B(0;R),&
\\p_s(0)=s\theta.&\end{array}\right.\end{equation}
We notice that the first equation in \eqref{Eq:ConsRate} rewrites as 
$$-\n \cdot(N^2\n p_s)=f(p_s)N^2.$$
Since $N$ is radially symmetric, this amounts to solving, in radial coordinates

\begin{equation}\label{Eq:ConsRate1}
\left\{\begin{array}{ll}
-\frac1{r^{d-1}}\left(r^{d-1}N^2p_s'\right)'=f(p_s)N^2\text{ in }[0;R]\,,&
\\p_s(0)=s\theta\,, p_s'(0)=0.&
\end{array}
\right.
\end{equation}
We prove the existence and uniqueness of solutions to \eqref{Eq:ConsRate1} below but underline that the core difficulty here is ensuring that $$0\leq p_s\leq 1.$$
\begin{claim}\label{Cl:ExisRate}
For any $s\in [0,1]$, there exists a unique solution to \eqref{Eq:ConsRate1}.
\end{claim}
\begin{proof}[Proof of Claim \ref{Cl:ExisRate}]
{\color{black}This follows from a standard contraction argument. Define, on $L^\infty(0;r_1)$ where $r_1<R$ will be fixed later on, the map

$$T:\p\mapsto s\theta+\int_0^r\frac{1}{l^{d-1}N^2}\int_0^l-t^{d-1}f(\p)N^2dtdl$$

We have the following estimate\begin{align*}
 \|T\p-T\phi\|_{L^\infty}&\leq \int_0^r\frac{1}{l^{d-1}N^2}\int_0^lt^{d-1}M\|\p-\phi\|_{L^\infty}N^2dtdl\\
 &\leq M \|\p-\phi\|_{L^\infty}\|N^2\|_{L^\infty}\left\|\frac{1}{N^2}\right\|_{L^\infty}\frac{r^2}{d}
\end{align*}
where $M$ is the Lipschitz constant of $f$. If $r_1$ is small enough, $T$ is a contraction in $L^\infty(0;r_1)$ and so  existence and uniqueness of a solution follows in $(0;r_1)$. In $(r_1;R)$, the standard Cauchy-Lipschitz theory applies.
}\end{proof}
\begin{claim}\label{Cl:Adm}
Under Assumption \ref{Eq:Rate} the path is admissible: we have, for any $s\in [0,1]$, 
\begin{equation}\label{Eq:AdmEst}0\leq p_s\leq 1.\end{equation}
Furthermore, the path $\{p_s\}_{s\in [0,1]}$ is continuous in the $\mathscr C^0$ topology.
\end{claim}
\begin{proof}[Proof of Claim \ref{Cl:Adm}]
\begin{enumerate}
\item\underline{Admissibility of the path under Assumption \ref{Eq:Rate}:} We now prove  Estimate \eqref{Eq:AdmEst}. To do so, we introduce the energy functional 
$$\mathscr E_1:x\mapsto \frac12 (p_s'(x))^2+F(p_s(x)),$$ where $F:x\mapsto \int_0^x f$ is the antiderivative of $f$.
Differentiating $\mathscr E_1$ with respect to $x$, we get
\begin{align*}
\mathscr E_1'(x)&=\left(p_s''(x)+f(p_s)\right)p_s'(x)&
\\&=\left(-\frac{d-1}r-2\frac{N'(r)}{N(r)}\right)(p_s'(r))^2&\text{ from Equation \eqref{Eq:ConsRate1}}
\\&\leq 0&\text{ from Hypothesis \ref{Eq:Rate}}.
\end{align*}
In particular, we have, for any $s\neq 0$,  $p_s\neq 0$ in $(0;R)$: arguing by contradiction if, for $\underline x \in (0;R)$ we had $p_s(\underline x)=0$ then 
$$\mathscr E_1(\underline x)=\frac12\left(p_s'(\underline x)\right)^2\geq 0.$$
However, $\mathscr E_1(0)=F(s\theta)<0$, so that a contradiction follows. For the same reason, $p_s\neq 1$ in $[0;R]$, for otherwise , if $p_s(\overline x)=1$ at some $\overline x\in [0,1]$ we would have 
$$\mathscr E_1(\overline x)\geq F(1)>0,$$ which is once again a contradiction. It follows that, for any $s\in (0;1]$, 
$$0\leq p_s\leq 1,$$ as claimed. 
\item\underline{Continuity of the path:} 
We want to prove the $\mathscr C^0$ continuity of the path. Let $s\in [0,1]$ and let $\{s_k\}_{k\in \N}\in [0,1]^\N$ be a sequence such that 
$$s_k\underset{k\to \infty}\rightarrow s.$$
Let $p_k:=p_{s_k}$. Our goal is to show that
\begin{equation}\label{Eq:ConvPath}p_k\underset{k\to \infty}{\overset{\mathscr C^0(\mathbb B(0;R))}\rightarrow}p_s.\end{equation} We will use elliptic regularity to ensure that. We first derive a $W^{1,\infty}$ estimate from the one-dimensional equation, and use it to obtain, \color{black}for any $\alpha\in (0;1)$, \color{black} a $\mathscr C^{2,\alpha}$ estimate for the equation set in $\mathbb B(0;R)$. By the admissibility of the path we have, for every $k\in \N$,
$$0\leq p_k\leq 1.$$
Passing into radial coordinates and integrating Equation \eqref{Eq:ConsRate1} between 0 and $x$ gives
\begin{equation}\label{Eq:Integral}-p_k'(x)=\frac1{N^2(x)x^{d-1}}\int_0^x f\left(p_k(s)\right)N^2(s)s^{d-1}ds.\end{equation}
Thus  the sequence $\{p_k\}_{k\in \N}$ is uniformly bounded in $W^{1,\infty}((0;1))$. We now consider Equation \eqref{Eq:ConsRate}. Since, \color{black} by the first step, \color{black} $\{p_k\}_{k\in \N}$ is uniformly bounded in $\mathscr C^{0,\alpha}(\mathbb B(0;R))$ \color{black} for any $\alpha \in(0;1)$, \color{black} and since $N\in \mathscr C^\infty(\mathbb B(0;R))$, it follows from H\"{o}lder elliptic regularity (see \cite{GilbargTrudinger}) that, \color{black} for any $\alpha \in (0;1)$,  there exists $M_\alpha \in \R$ \color{black}such that, for every $k\in \N$,
$$||p_k||_{\mathscr C^{2,\alpha}(\mathbb B(0;R))}\leq M_{\color{black}\alpha\color{black}}$$ hence $\{p_k\}_{k\in \N}$ converges in $\mathscr C^1(\mathbb B(0;R))$, up to a subsequence, to 
$p_\infty$. Passing to the limit in the weak formulation of the equation, we see that $p_\infty$ satisfies 
$$-\n \cdot \left(N^2\n p_\infty\right)=f(p_\infty)N^2.$$
Passing to the limit in $$\forall k\in \N\,, p_k(0)=s_k\theta$$ we get $p_\infty(0)=s\theta$ and, finally, since for every $k\in \N$, $p_k$ is radial, i.e
$$\forall k\in \N\,, \forall i,j\in \{1,\dots,d\}\,, x_j\frac{\partial p_k}{\partial x_i}-x_i\frac{\partial p_k}{\partial x_j}=0,$$ we can pass to the limit in this identity to obtain that $p_\infty$ is radial. In particular, 
$$p_\infty=p_s$$ and so the continuity of the path holds.
\end{enumerate}
\end{proof}
To conclude the proof of  Theorem \ref{Th:Rate}, it suffices to apply the staircase method.
\end{proof}

\section{Proof of Proposition \ref{Th:GF}: high-infection rate models}\label{Proof5}
\begin{proof}[Proof of Proposition \ref{Th:GF}]
The proof consists in transforming the equation 
\begin{equation}\label{Eq:GF1}\frac{\partial p}{\partial t}-\Delta p-2 \frac{N'}{N}(p)|\n p|^2=f(p),\end{equation} in a simpler one. This is done by following the idea of \cite[Proof of Theorem 1]{NadinStrugarekVauchelet}.  Let us introduce the anti-derivative of $N^2$ as
$$\mathscr N:x\mapsto \int_0^x N^2(\xi)d\xi,$$ We first note that multiplying $N$ by any factor $\lambda$ leaves the equation \eqref{Eq:GF1} invariant. We thus fix 
$$\int_0^1 N^2(\xi)d\xi=1.$$
Multiplying \eqref{Eq:GF1} by $ N^2$ we get 
$$N^2(p)\frac{\partial p}{\partial t}-N^2(p)\Delta p-2N(p)N'(p)|\n p|^2=\left(\mathscr N(p)\right)_t-\n \cdot\left(N^2(p)\n p\right)=\left(\mathscr N(p)\right)_t-\Delta (\mathscr N(p)).$$
Hence, as $\mathscr N$ is a diffeomorphism, the function $\tilde p:=\mathscr N(p)$ satisfies
$$ \frac{\partial \tilde p}{\partial t}-\Delta \tilde p=f\left(\mathscr N^{-1}(\tilde p)\right)N^2\left(\mathscr N^{-1}(\tilde p)\right)=:\tilde f(\tilde p).$$
However, it is easy to see that, $f$ being bistable, so is $\tilde f$. Furthermore,  $\mathscr N$ is a $\mathscr C^1$ diffeomorphism of $[0,1]$, and it is easy to see that $p$ is controllable to $0\,, \mathcal{N}(\theta)$ or $1$ if and only if $\tilde p$ is controllable to $0,\theta$ or $1$, and we are thus reduced to the statement of \cite[Theorem 1.2]{BRZ}, from which the conclusion follows.
\end{proof}

\section{Conclusion}\label{Conclusion}
\subsection{Obtaining the results for general coupled systems}
As explained in the introduction, the equations considered in this article correspond to some scaling limits for more general coupled systems of reaction-diffusion equations, and it seems interesting to investigate whether or not the results we obtained in this article might be generalized to encompass the case of such general systems. As was explained in the introduction, these models can be used to control populations of infected mosquitoes and arise in evolutionary dynamics. Obtaining a finer understanding of the real underlying dynamics rather than the simplified version under scrutiny here seems, however, challenging. Indeed, although controllability results for linear systems of equations exist (see for instance \cite{LissyZuazua}), the non-linear case has not yet been completely studied. 

{\color{black}
From the application point of view, we observe that a qualitative understanding  of the heterogeneity is a must. Indeed, the mildness of the Assumptions \eqref{AssumptionDrift}-\eqref{AssumptionDrift2} prove that, whenever a localized sharp transition in this heterogeneity occurs, controllability to steady-states may fail.
}

However, given that, as explained in the Introduction, gene-flow models and spatially heterogeneous models are limits in a certain scaling of such systems, it would be interesting to see whether or not our perturbation arguments, that were introduced to pass from the spatially homogeneous model \color{black} to \color{black} the slowly varying one, could work to pass from this scaling limit to the whole system in a certain regime. 

{\color{black}
 In the homogeneous case, when $\int_0^1f=0$, there does not exist any nontrivial solution with boundary values $0$ or $1$ \cite{BRZ}. However, in the heterogeneous setting, there can exists such nontrivial solutions. Note that in the proof of the first point of Theorem \ref{NewTheo}, that is, for the blocking phenomenon towards 1,  we have not used the fact that the primitive at $1$ has a particular sign.
}

\color{black}
\subsection{Open problem}
Let us now list a few questions which, to the best of our knowledge, are still open and seem worth investigating.

\begin{enumerate}
\item[$\bullet$] The qualitative properties of time optimal controls:

As suggested in \cite{PoucholTrelatZuazua} one might try to optimize the control with respect to the controllability time. Indeed, its is known that, under constraints on the control, parabolic equations have a minimal controllability time, see for instance \cite{ZW,PighinZuazua}. 

For constrained controllability it is known that there exists a minimal controllability time to control, for instance, from 0 to $\theta$ (see \cite{PoucholTrelatZuazua}). We may try to optimize the control strategies so as to minimize the controllability time. In our case, that is, the spatially heterogeneous case, are these controls of bang-bang type?  Another qualitative question that is relevant in this context is that of  symmetry: in the one dimensional case, when working on an interval $[-L,L]$, are time-optimal controls symmetric? In the multi-dimensional case, when the domain $\O$ is a ball, is it possible to prove radial symmetry of time optimal controls?

\item[$\bullet$] The influence of spatial heterogeneity on controllability time: 

Adding a drift (which corresponds to the spatially heterogeneous model) modifies the controllability time. As we have seen, such heterogeneities might lead to a lack of controllability. However, it is also suggested in the numerical experiments shown below that adding a drift might be beneficial for the controllability time. It might be interesting to consider the following question: given $L^\infty$ and $L^1$ bounds on the spatial heterogeneity $N$, which is the drift yielding the minimal controllability time? In other terms: how can we design the domain so as to minimize the controllability time? In the simulation below, we thus considered the following optimization problem: letting, for any drift $m=N_x/N$, $T(m)$ be the minimal controllability time from 0 to $\theta$ of the spatially heterogeneous equation \eqref{Eq:AdvectionControl} (with $T(m)\in (0;+\infty]$), solve
$$\inf_{-M\leq m\leq M\,}T(m).$$
We obtain the following graph with $M=250$ and $L=2.5$:
\begin{center}
\begin{figure}[H]
\begin{center}
\includegraphics[scale=0.3]{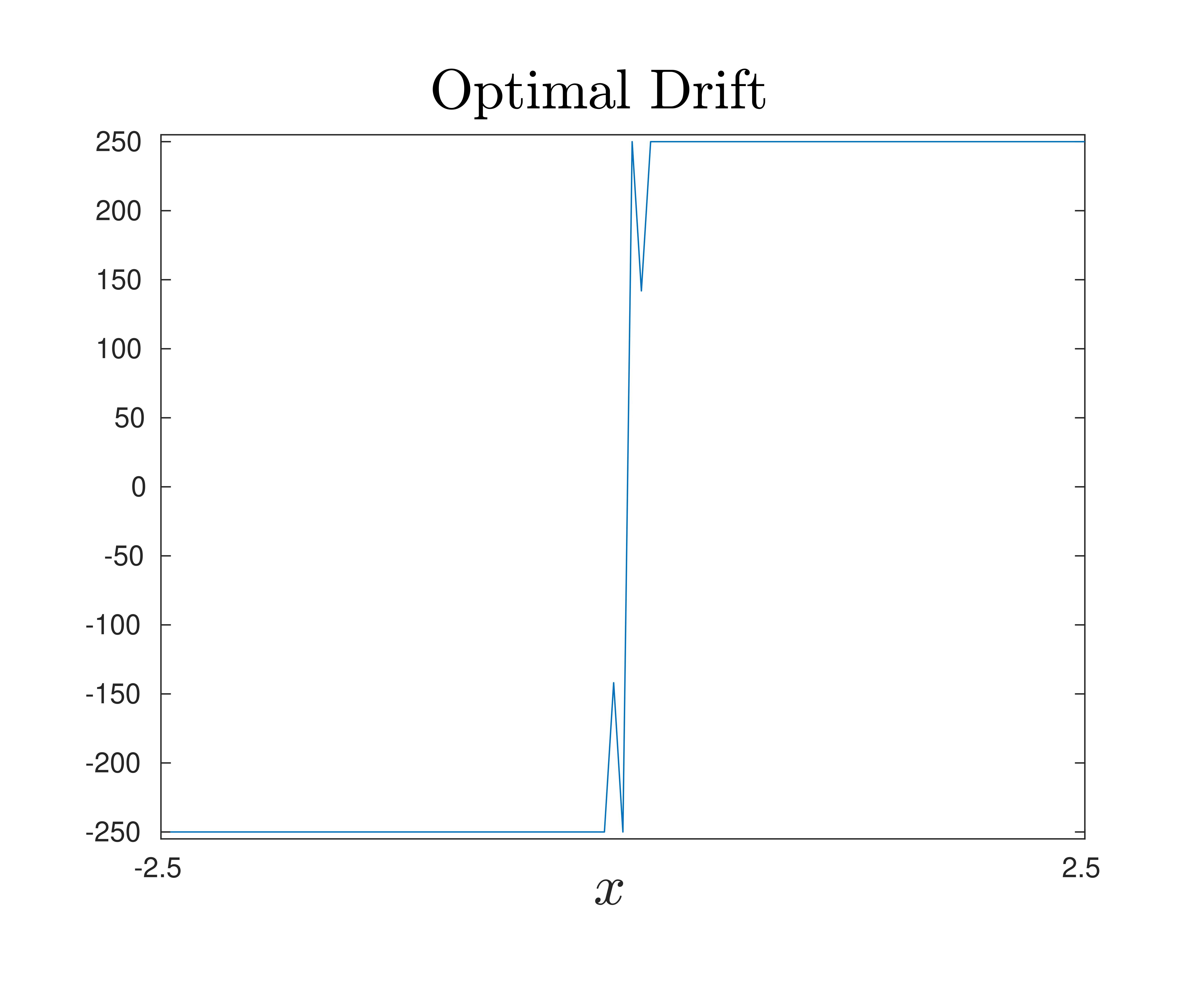}
\caption{Time optimal spatial heterogeneity.}\end{center}
\end{figure}\end{center}
\begin{center}
\begin{figure}[H]
\begin{center}
\includegraphics[scale=0.3]{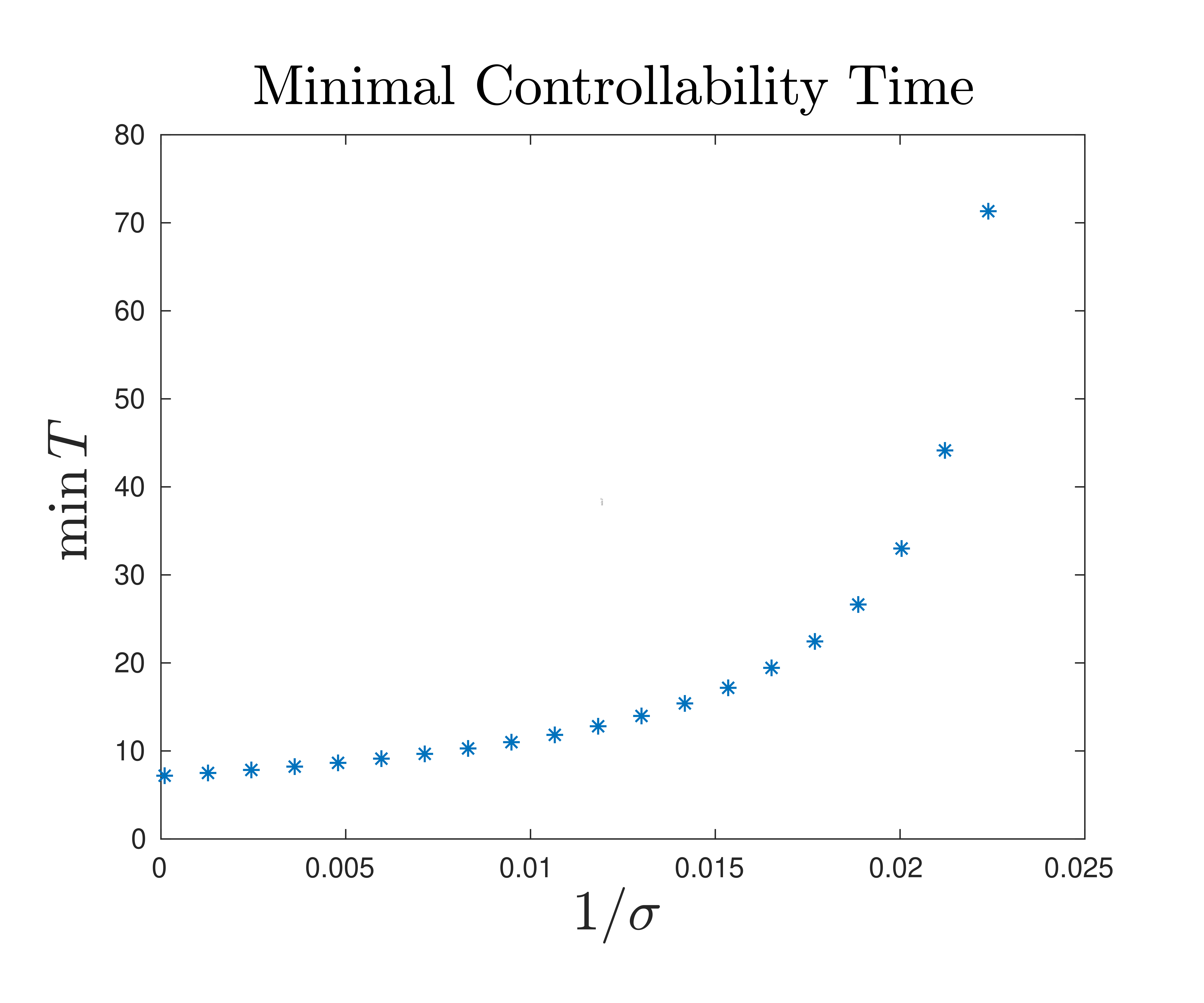}
\includegraphics[scale=0.3]{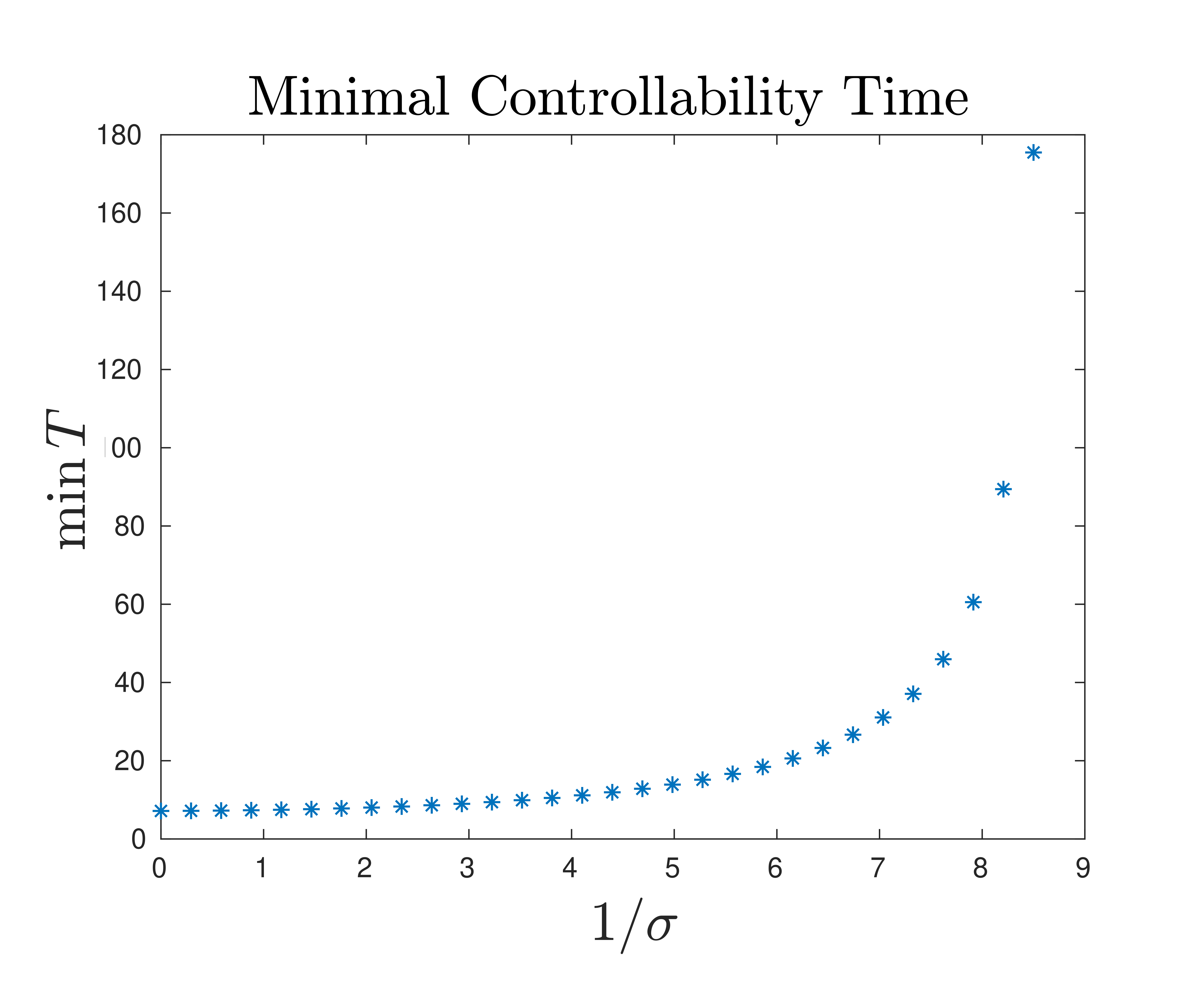}
\includegraphics[scale=0.3]{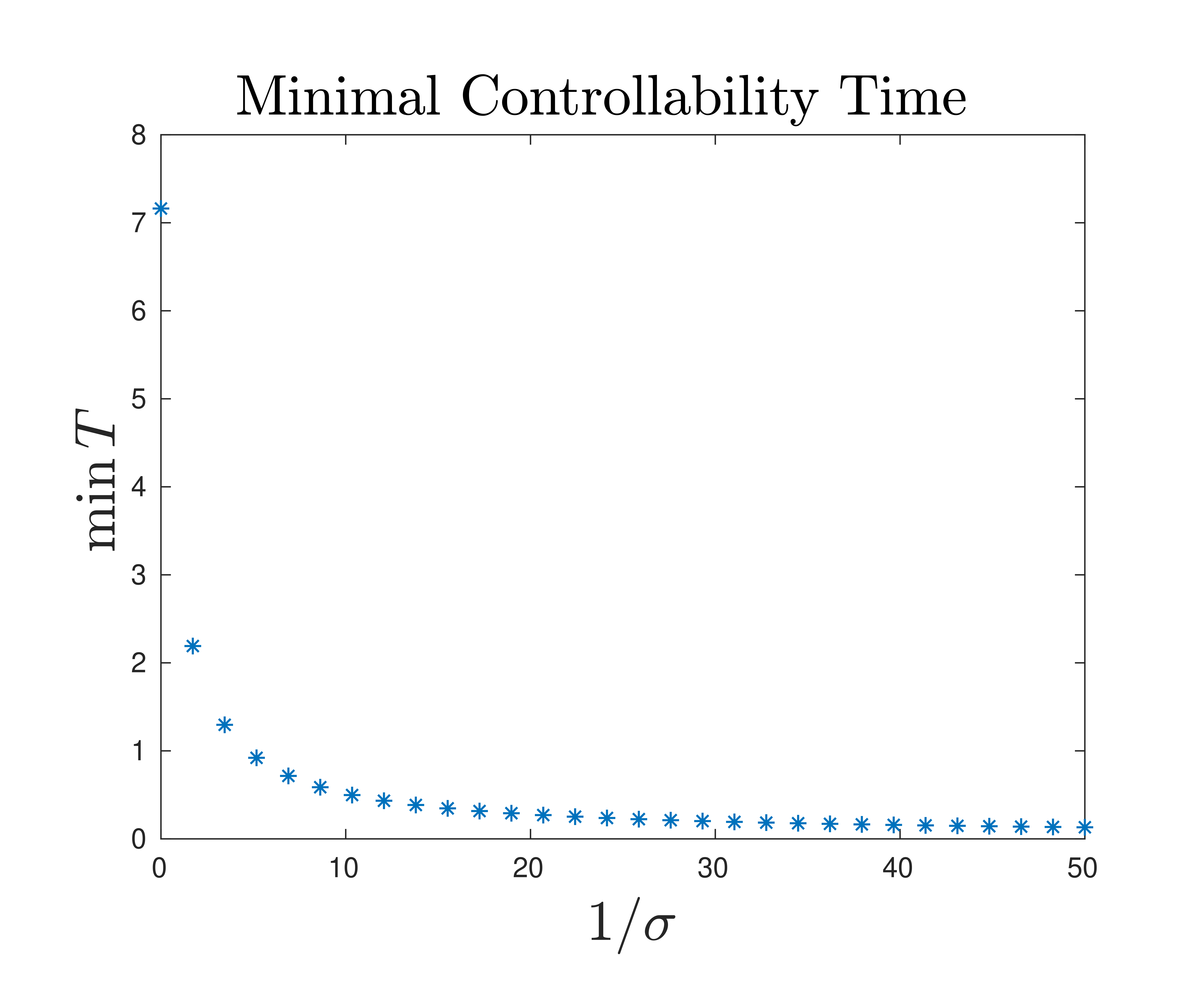}
\end{center}
\caption{Minimal controllability time depending on the strength of the drift. (Left) for $N=e^{-\frac{x^2}{\sigma}}$. (Center) $N'=\frac{\sin(x)}{\sigma}N$, (Right) $N=e^{\frac{x^2}{\sigma}}$}\label{Fig:MinContr}
\end{figure}\end{center}

In Figure \ref{Fig:MinContr}, we numerically observe that the minimal controllability time goes to zero for the case in which the radial derivative goes inwards while it blows up in the other two cases. For the case of the Gaussian, we observe the emergence of an upper barrier as the drift becomes stronger, the same is happening for the case of the sinusoidal drift. Even if these simulations may  fit  the intuition, a proper analysis of the minimal controllability time should be carried on.
\end{enumerate}
%
%
%
%
%
%
%
%
%
%
%

\appendix
\section{Proof of Lemma \ref{Le:NonTrivial}}\label{Ap:NonTrivial}

\begin{proof}[Proof of Lemma \ref{Le:NonTrivial}]
Let us first remark that \eqref{Eq:Le} has a variational structure. Indeed, $p$ is a solution of 
$$-\Delta p+\e\langle \n n\,, \n p\rangle=f(p)\,, p\in W^{1,2}_0(\O)$$ if and only if
\begin{equation}\label{Eq:Le2}-\n \cdot \left(e^{\e n}\n p\right)=f(p)e^{\e n}\,, p\in W^{1,2}_0(\O).\end{equation}
Following the arguments of \cite[Remark II.2]{LionsBerestycki}, we introduce the energy functional associated with \eqref{Eq:Le2}: let 
$$\mathscr E_1:W^{1,2}_0(\O)\ni p\mapsto \frac12\int_\O e^{\e n} |\n p|^2-\int_\O e^{\e n}F(p),$$
From standard arguments in the theory of sub and super solutions \cite{LionsBerestycki}, if there exists $v\in W^{1,2}_0(\O)$ such that 
\begin{equation}\label{Eq:EnergieNegative}\mathscr E_1(v)<0\end{equation} then there exists a non-trivial solution to \eqref{Eq:Le}. We now prove that there exists $v\in W^{1,2}_0(\O)$ such that \eqref{Eq:EnergieNegative} holds, by adapting the construction and computations of \cite{BRZ} (we only sketch the $d\geq 2$ case): let $\mathbb B(\overline x;\rho_\O)$ be one of the ball of maximum radius inscribed in $\O$. Up to a translation, we assume that $\overline x=0$.
\\Let $\delta>0$. We define $v_\delta$ as follows
$$
v_{\delta}:\left\{\begin{array}{ll}
x\in \mathbb B(0;\rho_\O-\delta)\mapsto 1,&
\\x\in \mathbb B(0;\rho_\O)\backslash\mathbb B(0;\rho_\O-\delta) \mapsto \frac{\rho_\O^2-||x||^2}{\rho_\O^2-(\rho_\O-\delta)^2},&
\\x\in  \O\backslash \mathbb B(0;\rho_\O)\mapsto 0.&\end{array}\right.$$
An explicit computation yields
$$\int_\O e^{\epsilon n}|\n v_\delta|^2\sim_{\delta \to0} C_1\delta \rho_\O^{d-1}e^{\epsilon n(\rho_\Omega)}$$ for some constant $C_1>0$, and 
$$\int_\O e^{\epsilon n}F(v_\delta)\geq C_2 F(1) (\rho_\O-\delta)^d .$$
Hence, since $n$ is bounded, as $\rho_\O$ grows, the second term in the energy functional will dominate and the conclusion follows: as $\rho_\O\to \infty$ and $\delta\to 0$ the energy of $v_1$ is negative.

\end{proof}

\paragraph{Acknowledgment.}The authors wish to thank B. Geshkovski for his numerous comments and P. Lissy for fruitful conversations. 

 The work of the authors has been funded by the Alexander von Humboldt-Professorship program, the European Research Council (ERC) under the European Union's Horizon 2020 research and innovation programme (grant agreement No. 694126-DyCon), grant MTM2017-92996 of MINECO (Spain), ELKARTEK project KK-2018/00083 ROAD2DC of the Basque Government, ICON of the French ANR and "Nonlocal PDEs: Analysis, Control and Beyond", AFOSR Grant FA9550-18-1-0242.

I. Mazari was  partially supported by the ANR Project ANR-18-CE40-0013 - SHAPO on Shape Optimization and by the Austrian Science Fund (FWF) through the grant I4052-N32.


\bibliographystyle{abbrv}
\bibliography{BiblioPrinc}

\begin{thebibliography}{10}

\bibitem{PrivatVaucheletStrugarek}
L.~Almeida, Y.~Privat, M.~Strugarek, and N.~Vauchelet.
\newblock {Optimal releases for population replacement strategies, application
  to Wolbachia}.
\newblock {\em {SIAM Journal on Mathematical Analysis}}, 51(4):3170--3194,
  2019.

\bibitem{Aronson1978}
D.~Aronson and H.~Weinberger.
\newblock Multidimensional nonlinear diffusion arising in population genetics.
\newblock {\em Advances in Mathematics}, 30(1):33--76, Oct. 1978.

\bibitem{Barton}
N.~H. Barton.
\newblock The effects of linkage and density-dependent regulation on gene flow.
\newblock {\em Heredity}, 57(3):415--426, Dec. 1986.

\bibitem{BartonTurelli}
N.~H. Barton and M.~Turelli.
\newblock Spatial waves of advance with bistable dynamics: Cytoplasmic and
  genetic analogues of allee effects.
\newblock {\em The American Naturalist}, 178(3):E48--E75, Sept. 2011.

\bibitem{BBC}
BBC-News.
\newblock Florida mosquitoes: 750 million genetically modified insects to be
  released.
\newblock \url{https://www.bbc.com/news/world-us-canada-53856776}, 2020.

\bibitem{Advanced}
C.~M. Bender and S.~A. Orszag.
\newblock {\em Advanced Mathematical Methods for Scientists and Engineers I}.
\newblock Springer New York, 1999.

\bibitem{BHR}
H.~Berestycki, F.~Hamel, and L.~Roques.
\newblock Analysis of the periodically fragmented environment model : I --
  species persistence.
\newblock {\em Journal of Mathematical Biology}, 51(1):75--113, 2005.

\bibitem{LionsBerestycki}
H.~Berestycki and P.~L. Lions.
\newblock Some applications of the method of super and subsolutions.
\newblock In C.~Bardos, J.~M. Lasry, and M.~Schatzman, editors, {\em
  Bifurcation and Nonlinear Eigenvalue Problems}, pages 16--41, Berlin,
  Heidelberg, 1980. Springer Berlin Heidelberg.

\bibitem{bohonak1999dispersal}
A.~J. Bohonak.
\newblock Dispersal, gene flow, and population structure.
\newblock {\em The Quarterly review of biology}, 74(1):21--45, 1999.

\bibitem{Bolnick}
D.~I. Bolnick and B.~M. Fitzpatrick.
\newblock Sympatric speciation: Models and empirical evidence.
\newblock {\em Annual Review of Ecology, Evolution, and Systematics},
  38(1):459--487, Dec. 2007.

\bibitem{bolnick2007natural}
D.~I. Bolnick and P.~Nosil.
\newblock Natural selection in populations subject to a migration load.
\newblock {\em Evolution}, 61(9):2229--2243, 2007.

\bibitem{BrascoDePhilippis}
L.~Brasco and G.~Philippis.
\newblock {\em Shape optimization and spectral theory}, chapter Spectral
  inequalities in quantitative formS.
\newblock De Gruyter, 01 2017.

\bibitem{Brezis}
H.~Brezis and J.~L. Vazquez.
\newblock Blow-up solutions of some nonlinear elliptic problems.
\newblock 1997.

\bibitem{Cabre}
X.~Cabr{\'{e}} and Y.~Martel.
\newblock Weak eigenfunctions for the linearization of extremal elliptic
  problems.
\newblock {\em Journal of Functional Analysis}, 156(1):30--56, June 1998.

\bibitem{CantrellCosner}
R.~S. Cantrell and C.~Cosner.
\newblock {\em Spatial Ecology via Reaction-Diffusion Equations}.
\newblock John Wiley \& Sons, 2003.

\bibitem{CaseTaper}
T.~J. Case and M.~L. Taper.
\newblock Interspecific competition, environmental gradients, gene flow, and
  the coevolution of species borders.
\newblock {\em The American Naturalist}, 155(5):583--605, May 2000.

\bibitem{CoronTrelat}
J.-M. Coron and E.~Tr{\'e}lat.
\newblock Global steady-state controllability of one-dimensional semilinear
  heat equations.
\newblock {\em SIAM J. Control and Optimization}, 43:549--569, 01 2004.

\bibitem{LeDret2018}
H.~L. Dret.
\newblock {\em Nonlinear Elliptic Partial Differential Equations}.
\newblock Springer International Publishing, 2018.

\bibitem{EscobedoKavian}
M.~Escobedo and O.~Kavian.
\newblock Variational problems related to self-similar solutions of the heat
  equation.
\newblock {\em Nonlinear Analysis: Theory, Methods {\&} Applications},
  11(10):1103--1133, Jan. 1987.

\bibitem{Fisher}
R.~A. Fisher.
\newblock The wave of advances of advantageous genes.
\newblock {\em Annals of Eugenics}, 7(4):355--369, 1937.

\bibitem{Gavrilet}
S.~Gavrilets.
\newblock {\em Fitness landscapes and the origin of species}.
\newblock {Princeton University Press}, 2004.

\bibitem{gemmell2018genetic}
M.~R. Gemmell, S.~A. Trewick, J.~S. Crampton, F.~Vaux, S.~F. Hills, E.~E. Daly,
  B.~A. Marshall, A.~G. Beu, and M.~Morgan-Richards.
\newblock Genetic structure and shell shape variation within a rocky shore
  whelk suggest both diverging and constraining selection with gene flow.
\newblock {\em Biological Journal of the Linnean Society}, 125(4):827--843,
  2018.

\bibitem{GilbargTrudinger}
D.~Gilbarg and N.~S. Trudinger.
\newblock {\em Elliptic Partial Differential Equations of Second Order}.
\newblock Springer Berlin Heidelberg, 1983.

\bibitem{henrot2006}
A.~Henrot.
\newblock {\em Extremum problems for eigenvalues of elliptic operators}.
\newblock Frontiers in Mathematics. Birkh\"auser Verlag, Basel, 2006.

\bibitem{Hofbauer1997}
J.~Hofbauer, V.~Hutson, and G.~Vickers.
\newblock Travelling waves for games in economics and biology.
\newblock {\em Nonlinear Analysis: Theory, Methods {\&} Applications},
  30(2):1235--1244, Dec. 1997.

\bibitem{Kielhfer}
H.~Kielh\"{o}fer.
\newblock {\em Bifurcation Theory}.
\newblock Springer New York, 2012.

\bibitem{KOLMOGOROV37}
A.~Kolmogorov.
\newblock {\'E}tude de l'{\'e}quation de la diffusion avec croissance de la
  quantit{\'e} de mati{\`e}re et son application {\`a} un probl{\`e}me
  biologique.
\newblock {\em Bull. Univ. Moskow, Ser. Internat., Sec. A}, 1:1--25, 1937.

\bibitem{KPP}
A.~Kolmogorov, I.~Pretrovski, and N.~Piskounov.
\newblock \'etude de l'\'equation de la diffusion avec croissance de la
  quantit\'e de mati\`ere et son application \`a un probl\`eme biologique.
\newblock {\em Moscow University Bulletin of Mathematics}, 1:1--25, 1937.

\bibitem{LamboleyLaurainNadinPrivat}
J.~Lamboley, A.~Laurain, G.~Nadin, and Y.~Privat.
\newblock {Properties of optimizers of the principal eigenvalue with indefinite
  weight and Robin conditions}.
\newblock {\em {Calculus of Variations and Partial Differential Equations}},
  55(6), Dec. 2016.

\bibitem{lenormand2002gene}
T.~Lenormand.
\newblock Gene flow and the limits to natural selection.
\newblock {\em Trends in Ecology \& Evolution}, 17(4):183--189, 2002.

\bibitem{levin1974gene}
D.~A. Levin and H.~W. Kerster.
\newblock Gene flow in seed plants.
\newblock In {\em Evolutionary biology}, pages 139--220. Springer, 1974.

\bibitem{Lions82}
P.~L. Lions.
\newblock On the existence of positive solutions of semilinear elliptic
  equations.
\newblock {\em SIAM Review}, 24(4):441--467, 1982.

\bibitem{LissyZuazua}
P.~Lissy and E.~Zuazua.
\newblock Internal observability for coupled systems of linear partial
  differential equations.
\newblock {\em {SIAM} Journal on Control and Optimization}, 57(2):832--853,
  Jan. 2019.

\bibitem{Matano}
H.~Matano.
\newblock Convergence of solutions of one-dimensional semilinear parabolic
  equations.
\newblock {\em J. Math. Kyoto Univ.}, 18(2):221--227, 1978.

\bibitem{Mayr}
E.~Mayr.
\newblock {\em Animal Species and Evolution}.
\newblock Harvard University Press, 1963.

\bibitem{MazariNadinPrivat}
I.~Mazari, G.~Nadin, and Y.~Privat.
\newblock {Optimal location of resources maximizing the total population size
  in logistic models}.
\newblock {\em To appear in Journal de math\'ematiques pures et appliqu\'ees},
  2019.

\bibitem{mcdermott1993gene}
J.~M. McDermott and B.~A. McDonald.
\newblock Gene flow in plant pathosystems.
\newblock {\em Annual review of phytopathology}, 31(1):353--373, 1993.

\bibitem{Mirrahimi}
S.~Mirrahimi and G.~Raoul.
\newblock Dynamics of sexual populations structured by a space variable and a
  phenotypical trait.
\newblock {\em Theoretical Population Biology}, 84:87--103, Mar. 2013.

\bibitem{Murray}
J.~D. Murray.
\newblock {\em Mathematical Biology}.
\newblock Springer Berlin Heidelberg, 1993.

\bibitem{NadinStrugarekVauchelet}
G.~Nadin, M.~Strugarek, and N.~Vauchelet.
\newblock Hindrances to bistable front propagation: application to wolbachia
  invasion.
\newblock {\em Journal of Mathematical Biology}, 76(6):1489--1533, May 2018.

\bibitem{NadinToledo}
G.~Nadin and A.~I. Toledo~Marrero.
\newblock {On the maximization problem for solutions of reaction-diffusion
  equations with respect to their initial data}.
\newblock working paper or preprint, July 2019.

\bibitem{Perthame}
B.~Perthame.
\newblock {\em Parabolic Equations in Biology}.
\newblock Springer International Publishing, 2015.

\bibitem{PighinZuazua}
D.~Pighin, , and E.~Zuazua.
\newblock Controllability under positivity constraints of semilinear heat
  equations.
\newblock {\em Mathematical Control {\&} Related Fields}, 8(3):935--964, 2018.

\bibitem{PoucholTrelatZuazua}
C.~Pouchol, E.~Tr{\'e}lat, and E.~Zuazua.
\newblock Phase portrait control for 1d monostable and bistable
  reaction-diffusion equations.
\newblock {\em Nonlinearity}, 05 2018.

\bibitem{Protter}
M.~H. Protter and H.~F. Weinberger.
\newblock {\em Maximum Principles in Differential Equations}.
\newblock Springer New York, 1984.

\bibitem{Raddi1128}
G.~Raddi, A.~B.-F. Barletta, M.~Efremova, J.-L. Ramirez, R.~Cantera, S.-A.
  Teichmann, C.~Barillas-Mury, and O.~Billker.
\newblock Mosquito cellular immunity at single-cell resolution.
\newblock {\em Science}, 369(6507):1128--1132, 2020.

\bibitem{BRZ}
D.~Ruiz-Balet and E.~Zuazua.
\newblock Control under constraints for multi-dimensional reaction-diffusion
  monostable and bistable equations.
\newblock {\em Journal de Math{\'e}matiques Pures et Appliqu{\'e}es}, 2020.

\bibitem{ShigesadaKawaski}
N.~Shigesada and K.~Kawasaki.
\newblock {\em Biological Invasions: Theory and Practice}.
\newblock Oxford University Press, 1997.

\bibitem{slarkin1985gene}
M.~Slarkin.
\newblock Gene flow in natural populations.
\newblock {\em Annual review of ecology and systematics}, 16(1):393--430, 1985.

\bibitem{slatkin1987gene}
M.~Slatkin.
\newblock Gene flow and the geographic structure of natural populations.
\newblock {\em Science}, 236(4803):787--792, 1987.

\bibitem{StrugarekVauchelet}
M.~Strugarek and N.~Vauchelet.
\newblock Reduction to a single closed equation for 2-by-2 reaction-diffusion
  systems of lotka-volterra type.
\newblock {\em SIAM Journal of Applied Mathematics}, 76:2060--2080, 2016.

\bibitem{TrelatZhuZuazua}
E.~Tr{\'{e}}lat, J.~Zhu, and E.~Zuazua.
\newblock Allee optimal control of a system in ecology.
\newblock {\em Mathematical Models and Methods in Applied Sciences},
  28(09):1665--1697, Aug. 2018.

\bibitem{UriarteIriberri}
J.-R. Uriarte and N.~Iriberri.
\newblock Minority language and the stability of bilingual equilibria.
\newblock {\em Rationality and Society}, 24, 01 2011.

\bibitem{ZW}
G.~Wang and E.~Zuazua.
\newblock On the equivalence of minimal time and minimal norm controls for
  internally controlled heat equations.
\newblock {\em {SIAM} Journal on Control and Optimization}, 50(5):2938--2958,
  Jan. 2012.

\bibitem{Laplace}
R.~A. Zalik.
\newblock The method of laplace and watson's lemma.
\newblock {\em Journal of Concrete and Applicable Mathematics}, 10, 2012.

\end{thebibliography}

\end{document}